\documentclass{amsart}

\usepackage[normalem]{ulem}
\usepackage{amsmath,amsthm,amsfonts,amscd,amssymb, tikz-cd} 
\usepackage{verbatim}      	
\usepackage{listings} 
\usepackage{mathtools}
\usepackage{epsfig}         	
\usepackage{url}		
\usepackage{epic,latexsym}
\usepackage{graphicx}
\usepackage{xcolor}

\usepackage{lcd}
\usepackage{mathrsfs}
\usepackage[centering,text={15.5cm,22cm},
        marginparwidth=20mm,dvips]{geometry}

\usepackage{hyperref}
\hypersetup{%
pdftitle={Preprint},
pdfauthor={},
pdfkeywords={},
bookmarksnumbered,
pdfstartview={FitH},
breaklinks=true,
colorlinks=true,
urlcolor=blue,
citecolor=blue,
linktocpage=true
}%
\usepackage[hypcap=true]{caption}

\usepackage{breakurl}

\usepackage{lscape}
\usepackage{tabularx}
\usepackage{marginnote}
\usepackage[all]{xy}
\usepackage{enumitem}
\usepackage{stmaryrd}
\usepackage{tipa}
\usepackage{upgreek}

\usepackage{tikz}
\usetikzlibrary{
	decorations.pathreplacing, 
	decorations.markings,      
	calc,                     
	knots                     
}
\usepackage{caption}          
\usetikzlibrary{patterns}

\usepackage{float}
\usepackage{bm}

\tikzset{
	marked/.style={fill=black, circle, inner sep=1.5pt}
}

\DeclareMathOperator{\tw}{tw}

\def\!{\mskip-\thinmuskip}
\let\?=\overline
\let\:=\colon

\let\s=\mathcal

\newcommand {\mm} {{\bf m}}

\newcommand {\sd} {{\bf s}}

\theoremstyle{plain}
\ifdefined\thm
\else
	\theoremstyle{plain}
 	\newtheorem{thm}{Theorem}[section]
\fi

\ifdefined\prop
\else
	\theoremstyle{plain}
 	\newtheorem{prop}[thm]{Proposition}
\fi\ifdefined\cor
\else
	\theoremstyle{plain}
 	\newtheorem{cor}[thm]{Corollary}
\fi

\ifdefined\lem
\else
	\theoremstyle{plain}
 	\newtheorem{lem}[thm]{Lemma}
\fi

\ifdefined\def
\else
	\theoremstyle{definition}
	\newtheorem{def}[thm]{Definition}
\fi

\ifdefined\defn
\else
	\theoremstyle{definition}
	\newtheorem{defn}[thm]{Definition}
\fi

\ifdefined\example
\else
	\theoremstyle{definition}
    \newtheorem{example}[thm]{Example}
\fi

\ifdefined\eg
\else
		\theoremstyle{definition}
        \newtheorem{eg}[thm]{Example}
\fi

\ifdefined\rem
\else
		\theoremstyle{remark}
        \newtheorem{rem}[thm]{Remark}
\fi

  \newtheorem{dfn}[thm]{Definition}

  \newtheorem{conj}[thm]{Conjecture}

\newcommand{\supp}{\operatorname{supp}}

\newcommand{\teps}{{\tilde{\varepsilon}}}

\usepackage{tikz-cd}

\usepackage[autostyle, english=american]{csquotes}
\MakeOuterQuote{"}

\newcommand{\Hf}{\frac{1}{2}}

\newcommand{\clAlg}{\s{A}}
\newcommand{\bClAlg}{\overline{\clAlg}}
\newcommand{\upClAlg}{\s{U}}
\newcommand{\bUpClAlg}{\overline{\upClAlg}}

\renewcommand{\tw}{\operatorname{tw}}

\newcommand{\Arcs}{\mathbf{\Gamma}}
\newcommand{\Loops}{\mathbf{L}}
\newcommand{\Unknots}{\mathbf{N}}
\newcommand{\Unarcs}{\Arcs^0}

\ifdefined\supp
\else
\newcommand{\supp}{\operatorname{supp}}
\fi

\ifdefined\tw
\else
	\newcommand{\tw}{{\opname tw}}
\fi

\ifdefined\G
    \renewcommand{\G}{{\mathbb G}}
\else
   \newcommand{\G}{{\mathbb G}}
\fi

\ifdefined\C
    \renewcommand{\C}{{\mathbb C}}
\else
   \newcommand{\C}{{\mathbb C}}
\fi

\usepackage{comment}

\title{Band bases as common triangular bases in cluster algebras from surfaces}

\author{Fan Qin}
\address{Beijing Normal University \\ China}
\email{qin.fan.math@gmail.com}

\author{Chao Shen}
\address{Tsinghua University \\ China}
\email{sc20@mails.tsinghua.edu.cn}

\thanks{}

\begin{document}

\begin{abstract}	
    We consider the skein algebra of an unpunctured marked surface. Thurston previously constructed its band basis topologically. We show that this band basis coincides with the common triangular basis, which is a Kazhdan-Lusztig type basis for quantum cluster algebras analogous to the dual canonical basis of quantum groups. Our result confirms a conjecture of Thurston. It also provides new cases for the existence of the common triangular basis. In addition, we discover a phenomenon where certain unknots are arranged in configurations resembling beads on a necklace.
\end{abstract}
		
\maketitle
 
\setcounter{tocdepth}{1} \tableofcontents{}
	
\section{Introduction}	\label{intro}

\subsection{Backgrounds}

Cluster algebras, introduced by Fomin--Zelevinsky \cite{fomin2002cluster}, are commutative algebras generated by cluster variables which are constructed recursively via a combinatorial algorithm called mutation. They serve as an algebraic framework for studying the dual canonical basis \cite{Lusztig90, Lusztig91} \cite{Kashiwara90} and the theory of total positivity \cite{Lusztig96}. Subsequently, Berenstein--Zelevinsky \cite{BerensteinZelevinsky05} introduced a natural quantization of this structure, giving rise to the theory of quantum cluster algebras. 

A primary objective in the study of (quantum) cluster algebras is the construction of "good" bases. These bases are expected to share similar properties with the dual canonical bases of quantum groups. Important families of these bases include: 
\begin{itemize}
    \item The generic basis in the sense of Dupont \cite{dupont2011generic} generalizes the dual semi-canonical basis \cite{Lusztig00} to the cluster algebra setting \cite{GeissLeclercSchroeer10}. For explicit computations on tame quivers, see \cite{DingXiaoXu08} \cite{dupont2011generic}. Its general construction and existence can be found in \cite{plamondon2013generic} \cite{qin2021cluster}.

    \item The common triangular basis, introduced by the first author \cite{qin2017triangular}, is a Kazhdan--Lusztig type basis that extends the dual canonical basis to cluster algebras \cite{qin2017triangular, qin2020dual}. In acyclic cases, it coincides with the triangular basis of Berenstein--Zelevinsky \cite{BerensteinZelevinsky2012}, see \cite{qin2019compare} \cite[Remark 1.2.7]{qin2020analog}.
    
    \item The theta basis, introduced by Gross--Hacking--Keel--Kontsevich \cite{gross2018canonical}, consists of theta functions arising from mirror symmetry and scattering diagrams. Its quantum analogue was subsequently developed by Davison--Mandel \cite{davison2019strong}.
\end{itemize}
All three families of bases above are parameterized by tropical points, in accordance with the expectation of Fock--Goncharov \cite{fock2006moduli, FockGoncharov09}. It is believed that all known "good" bases for cluster algebras in the literature belong to these families.

In the setting of cluster algebras from marked surfaces (via skein algebras), several bases are constructed using topological objects, such as the bangle, bracelet, and band bases \cite{MusikerSchifflerWilliams11} \cite{thurston2014positive}. We refer the reader to Section \ref{sec:skein-alg} for the background on the marked surface $\Sigma$ and its skein algebra, and to Sections \ref{sec:wt-simple-multi-curve} and \ref{sec:band} for the details of the bangle and band bases, respectively. These bases have been compared with the general frameworks mentioned above. For a marked surface with non-empty boundary, Geiß--Labardini-Fragoso--Wilson \cite{geiss2023bangle} showed that the generic basis coincides with the bangle basis. Mandel--Qin \cite{MandelQin2021} established that the bracelet basis coincides with the theta basis for general marked surfaces.

Although the bangle and bracelet bases fit into the general frameworks, the band basis remains less understood. 
Thurston made the following bold conjecture \cite{thurston2014positive}:

\begin{center}
    \it The (quantized) band basis is an analogue of the dual canonical basis for the skein algebra.
\end{center}

Recall that the common triangular basis provides an analogue of the dual canonical basis \cite{qin2017triangular, qin2020analog}. Therefore, Thurston's conjecture follows from the conjecture below.

\begin{conj}\label{conj:band-is-triangular}
    The (quantized) band basis is the common triangular basis of the skein algebra after localization at the frozen variables.
\end{conj}

\begin{rem}
    The only known evidence for this conjecture is the Kronecker case. In this setting, the common triangular basis is identified with the dual canonical basis. The dual canonical basis is known to be the band basis because its elements (which correspond to imaginary roots) satisfy the Chebyshev recursion \cite{Lamp10}. Therefore, the common triangular basis is realized as the band basis in the Kronecker case.
\end{rem}

\subsection{Main results and our approach}
Our main result is a proof of Conjecture \ref{conj:band-is-triangular} for unpunctured surfaces. 

\begin{thm}\label{thm:main}
    For any unpunctured surface $\Sigma=(\mathcal{S}, \mathcal{M})$, the band basis of the skein algebra $\mathrm{Sk}_q^\circ(\Sigma)$ coincides with the common triangular basis.
\end{thm}
Theorem \ref{thm:main} provides the existence of the common triangular basis for a new family of cluster algebras, which was previously known only for cluster algebras arising from Lie theory \cite{qin2023analogs}. It also allows us to visualize Kazhdan--Lusztig type basis elements using topological objects.

\begin{rem}
    The main result provides strong evidence for the existence of a monoidal category whose simple objects categorify the band basis. Indeed, the common triangular basis of a cluster algebra is typically categorified by the set of classes of simple objects in a suitable monoidal category, possibly after mild modification of the cluster algebra (a change of the frozen part) \cite{qin2017triangular, qin2023analogs}. In this light, the fact that the band basis coincides with the common triangular basis suggests that there might exist a monoidal category where the band elements correspond to its simple objects.
\end{rem}

In this paper, we often consider the following expansion in the skein algebra:
\begin{align}\label{eq:XY-decomp}
[\mathsf{X}][\mathsf{Y}]=\sum_\varepsilon q^{\alpha_\varepsilon}[\mathsf{E}_\varepsilon],    
\end{align}
where $\mathsf{X}$ and $\mathsf{Y}$ are weighted simple multicurves (Section \ref{sec:wt-simple-multi-curve}), $[\mathsf{X}]$ and $[\mathsf{Y}]$ are their homotopy classes, such that $[\mathsf{X}]$ is a cluster monomial (i.e., $\mathsf{X}$ only consists of arc components). The sum is over the choices $\varepsilon$ of smoothings at the crossings in $\mathsf{X} \cap \mathsf{Y}$. These two types of local smoothings, namely the positive and negative smoothings, are illustrated in Figure \ref{fig:skein-resolutions}.
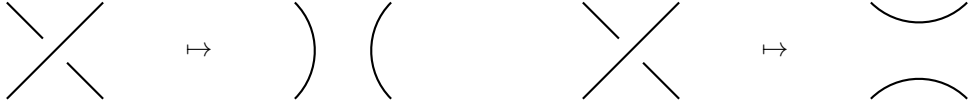
\begin{figure}[H]
	\centering
	\begin{tikzpicture}[scale=1.5]
		\begin{scope}[xshift=-1in]
			\begin{scope}[xshift=-.5in,scale=.15]
				\draw[thick] (-2.83,-2.83) to (2.83,2.83);
				\draw[thick] (-2.83,2.83) to (-.71,.71);
				\draw[thick] (.71,-.71) to (2.83,-2.83);
									
			\end{scope}
			\node (=) at (0,0) {$\mapsto$};
			\begin{scope}[xshift=.5in,scale=.15]
				\draw[thick] (-2.83,-2.83) to [out=45,in=-45]  
				(-2.83,2.83);
				\draw[thick] (2.83,-2.83) to [out=135,in=-135] 
				(2.83,2.83);
			\end{scope}
		\end{scope}
		
		\begin{scope}[xshift=1in]
			\begin{scope}[xshift=-.5in,scale=.15]
				\draw[thick] (-2.83,-2.83) to (2.83,2.83);
				\draw[thick] (-2.83,2.83) to (-.71,.71);
				\draw[thick] (.71,-.71) to (2.83,-2.83);
			\end{scope}
			\node (=) at (0,0) {$\mapsto$};
			\begin{scope}[xshift=.5in,scale=.15]
				\draw[thick] (-2.83,-2.83) to [out=45,in=135]
			    (2.83,-2.83);
				\draw[thick] (-2.83,2.83) to [out=-45,in=-135]
				 (2.83,2.83);
			\end{scope}
		\end{scope}
	\end{tikzpicture}
	\caption{The $+$ smoothing (left) and $-$ smoothing (right) of a crossing.}
	\label{fig:skein-resolutions}
\end{figure}
For each $\varepsilon$, we have $\alpha_\varepsilon\in \mathbb{Z}$ and $\mathsf{E}_\varepsilon$ is the multicurve obtained after resolving the crossings. These multicurves might contain contractible components. More details are provided in Section \ref{subsec:attachment}.

In Theorem \ref{thm:leading-g}, we characterize the properties of the term $[\mathsf{E}_{\varepsilon_+}]$ in the expansion of $[\mathsf{X}][\mathsf{Y}]$ as a corollary of our main result, where $\varepsilon_+$ denotes the choice where all crossings receive positive smoothings. In particular, the following natural properties are of independent interest:
	$$\deg ([\mathsf{E}_{\varepsilon_+}])=\deg ([\mathsf{X}])+\deg ([\mathsf{Y}]),$$ 
  	 $$\Lambda(\deg ([\mathsf{X}]),\deg ([\mathsf{Y}]))=2\mu(\mathsf{X}, \mathsf{Y})+2a,$$
where $\mu(\mathsf{X}, \mathsf{Y})$ is the minimal intersection number and $a$ is the boundary intersection number in \eqref{eq:decompose-bd-intersection}.

In our approach to Theorem \ref{thm:main}, a key step is to understand the decomposition \eqref{eq:XY-decomp} where $[\mathsf{X}]$ is a cluster variable (i.e., $\mathsf{X}$ is an arc). In this case, we write $\mathsf{E}_\varepsilon=\Unknots_\varepsilon \cup \Unarcs_\varepsilon \cup  \widehat{\mathsf{E}}_\varepsilon$, where $\Unknots_\varepsilon$ is the set of unknots with cardinality $|\Unknots_\varepsilon|$, $\Unarcs_\varepsilon$ consists of contractible arcs, and $\widehat{\mathsf{E}}_\varepsilon$ contains no contractible components. By evaluating each unknot to $-q^2-q^{-2}$ and setting contractible arcs to $0$, we obtain
\begin{equation}\label{eq:decompose}
[\mathsf{X}][\mathsf{Y}] = \sum_{\varepsilon: \Unarcs_\varepsilon = \emptyset} q^{\alpha_\varepsilon} (-q^2 - q^{-2})^{|\Unknots_\varepsilon|} [\widehat{\mathsf{E}}_\varepsilon].
\end{equation}
It is crucial to show that the sum of the terms $(-1)^{|\Unknots_\varepsilon|} q^{\alpha_\varepsilon + 2|\Unknots_\varepsilon|} [\widehat{\mathsf{E}}_\varepsilon]$ with highest $q$-degree $\alpha_\varepsilon + 2|\Unknots_\varepsilon|$ in \eqref{eq:decompose} (cf. discussion near \eqref{eq:detail-dec}) forms a band basis element. Let $\mathcal{E}_{\mathrm{top}}$ denote the set of $\varepsilon$ corresponding to these terms with highest $q$-degree. 

We examine the structure of the resolved multicurves and discover a new phenomenon: the unknots (contractible loops) of any $\mathsf{E}_\varepsilon$ for $\varepsilon\in \mathcal{E}_{\mathrm{top}}$ form bead-chain configurations (see Example \ref{ex: bead-chain} for an illustration of such a configuration), where:
\begin{itemize}
    \item A {bead} is an isolated unknot.
    \item A {chain} is an unknot formed by the concatenation of multiple beads.
\end{itemize}
We introduce strip operations to describe how these chains merge or split when the smoothings are changed. This process is analogous to restringing beads on  necklaces. Applying strip operations to $\mathsf{E}_\varepsilon$ yields a new $\mathsf{E}_\varepsilon'$ where $\varepsilon'\in \mathcal{E}_{\mathrm{top}}$ still holds (Lemma \ref{lem:top-chains}). We refer the reader to Section \ref{sec:3-1} for details. In that section, we use Muller’s analysis of skein algebras via tubular neighborhoods \cite{muller2016skein}. This technique gives concise proofs of several important statements that would otherwise require more cumbersome arguments.

\begin{example}\label{ex: bead-chain}
      Consider the marked surface $\Sigma$ illustrated in Figure \ref{fig:surface_arc_multicurve_2}, where an arc $\mathsf{x}$ (red) and a multicurve $\mathsf{Y}$ (blue) form the superposition $\mathsf{x} \cdot \mathsf{Y}$.

      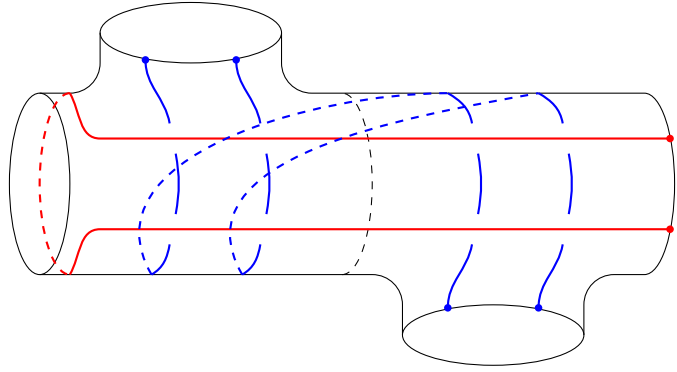
\begin{figure}[H]
   	\centering
   	\begin{tikzpicture}[scale=.4]
   		\begin{scope}
	        \draw[solid] (-11,0) arc (-180: 180 : 1 and 3);
	        \draw [solid] (-10,3) to (-9, 3)
	        to [out = 0, in= 270] (-8, 4) 
	        to (-8, 5);
	        
	        \draw [solid] (-8,5) arc (-180: 180 : 3 and 1);
	        
	        \draw [solid] (-2,5) to (-2, 4) 
	        to [out=270, in=180] (-1, 3) 
	        to (10, 3) arc (90: -90: 1 and 3)
	        to (9, -3) 
	        to [out=180, in=90] (8, -4)
	        to (8, -5); 
	        
	        \draw [solid] (2,-5) arc (-180: 180 : 3 and 1)
	        to (2,-4)
	        to [out=90, in=0] (1,-3)
	        to (-10,-3);
	        
	        \draw [dashed] (0,-3) arc (-90: 90: 1 and 3);
	        
	        \draw [thick,red] (-9,3) to [out=300, in=180] (-8, 1.5) to (10.85,1.5) node[circle, fill=red, inner sep=1pt] {};
	        \draw [thick,red] (-9,-3) to [out=60, in=180] (-8, -1.5) to (10.85,-1.5) node[circle, fill=red, inner sep=1pt] {};
	        \draw [dashed, red,thick] (-9,3) arc (90: 270: 1 and 3);
	        
	        \draw [thick, blue] (-6.5, 4.1) node[circle, fill=blue, inner sep=1pt] {} to [out=270, in=100] (-5.7, 2);
	        \draw [thick, blue] (-5.5, 1) to [out=-80, in=80] (-5.5, -1);
	        \draw [thick, blue] (-5.7, -2) to [out=-100, in=30] (-6.3, -3);
	        
	        \draw [thick, blue] (-3.5, 4.1) node[circle, fill=blue, inner sep=1pt] {} to [out=270, in=100] (-2.7, 2);
	        \draw [thick, blue] (-2.5, 1) to [out=-80, in=80] (-2.5, -1);
	        \draw [thick, blue] (-2.7, -2) to [out=-100, in=30] (-3.3, -3);
	        
	        \draw [thick, blue] (6.5, -4.1) node[circle, fill=blue, inner sep=1pt] {} to [out=90, in=-100] (7.3, -2);
	        \draw [thick, blue] (7.5, -1) to [out=80, in=-80] (7.5, 1);
	        \draw [thick, blue] (7.3, 2) to [out=100, in=-30] (6.5, 3);
	        
	        \draw [thick, blue] (3.5, -4.1) node[circle, fill=blue, inner sep=1pt] {} to [out=90, in=-100] (4.3, -2);
	        \draw [thick, blue] (4.5, -1) to [out=80, in=-80] (4.5, 1);
	        \draw [thick, blue] (4.3, 2) to [out=100, in=-30] (3.5, 3);
	        
	        \draw [dashed, blue, thick] (-6.3,-3) to [out=120, in =180] (3.5,3);
	        \draw [dashed, blue, thick] (-3.3,-3) to [out=120, in =190] (6.5,3);
        \end{scope})
    \end{tikzpicture}
    \caption{The marked surface $\Sigma$ with an arc $\mathsf{x}$ and a multicurve $\mathsf{Y}$.}
    \label{fig:surface_arc_multicurve_2}
   \end{figure}

  Applying the choice of smoothings $\varepsilon = (+, -, +, +, +, -, +, +)$ along the arc $\mathsf{x}$ (from top to bottom) yields the configuration in Figure~\ref{fig:linear_beads}, where the two unknot components are beads. Attaching a strip results in the chain illustrated in Figure~\ref{fig:cyclic_chain}.

   \begin{figure}[H]
		\centering
		\begin{tikzpicture}[scale=.4]
			\begin{scope}
				\draw[solid] (-11,0) arc (-180: 180 : 1 and 3);
				\draw [solid] (-10,3) to (-9, 3)
				to [out = 0, in= 270] (-8, 4) 
				to (-8, 5);
				
				\draw [solid] (-8,5) arc (-180: 180 : 3 and 1);
				
				\draw [solid] (-2,5) to (-2, 4) 
				to [out=270, in=180] (-1, 3) 
				to (10, 3) arc (90: -90: 1 and 3)
				to (9, -3) 
				to [out=180, in=90] (8, -4)
				to (8, -5); 
				
				\draw [solid] (2,-5) arc (-180: 180 : 3 and 1)
				to (2,-4)
				to [out=90, in=0] (1,-3)
				to (-10,-3);
				
				\draw [dashed] (0,-3) arc (-90: 90: 1 and 3);
				
				\draw [thick] (-6.5, 4.1) node[circle, fill=black, inner sep=1pt] {} to [out=270, in=0] (-6, 1.5);
			    \draw [thick] (-9,3) to [out=300, in=180] (-8, 1.5) to (-6,1.5);
			    \draw [dashed,thick] (-9,3) arc (90: 270: 1 and 3);
			    \draw [thick] (-9,-3) to [out=60, in=180] (-8, -1.5) to (-6,-1.5)
			    to [out=0, in=180] (-6,1)
			    to [out=0, in=270] (-3,2);
			    \draw [thick] (-3.5, 4.1) node[circle, fill=black, inner sep=1pt] {} to [out=270, in=90] (-3, 2);
			    
			    \draw [thick] (6.5, -4.1) node[circle, fill=black, inner sep=1pt] {} to [out=90, in=180] (7.3, -2)
			    to (10.75,-2) node[circle, fill=black, inner sep=1pt] {};
			    \draw [thick] (3.5, -4.1) node[circle, fill=black, inner sep=1pt] {} to [out=90, in=180] (4.3, -1.5)
			    to [out=0, in=270] (7.3,-1) 
			    to [out=90, in=180] (7.5, 1)
			    to (10.95,1) node[circle, fill=black, inner sep=1pt] {};
			    
			    \draw [thick, blue] (-5.7, -2.5) to [out=-100, in=30] (-6, -3);
			    \draw [thick, blue] (-5.7, -2.5)to [out=90, in=90] (-2.7, -2.5);
			    \draw [thick, blue] (-2.7, -2.5) to [out=-100, in=30] (-3, -3);
			    
			    \draw [thick, blue] (7.3, 2) to [out=100, in=-30] (6.5, 3);
			    \draw [thick, blue] (4.3, 2) to [out=100, in=-30] (3.5, 3);
                \draw [thick, blue] (7.3, 2)to [out=270, in=270] (4.3, 2);
                
                 \draw [dashed, blue, thick] (-6,-3) to [out=120, in =180] (3.5,3);
                \draw [dashed, blue, thick] (-3,-3) to [out=120, in =190] (6.5,3);
                
                \draw [thick, red] (-2.5,-1.5) to (3,-1.5)
                to [out=0, in=0] (3.5,1.5)
                to (-2.2, 1.5)
                to [out=180, in=180] (-2.5, -1.5);

			\end{scope})
		\end{tikzpicture}
            \caption{Beads}
            \label{fig:linear_beads}
	\end{figure}

        \begin{figure}[H]
		\centering
		\begin{tikzpicture}[scale=.4]
			\begin{scope}
				\draw[solid] (-11,0) arc (-180: 180 : 1 and 3);
				\draw [solid] (-10,3) to (-9, 3)
				to [out = 0, in= 270] (-8, 4) 
				to (-8, 5);
				
				\draw [solid] (-8,5) arc (-180: 180 : 3 and 1);
				
				\draw [solid] (-2,5) to (-2, 4) 
				to [out=270, in=180] (-1, 3) 
				to (10, 3) arc (90: -90: 1 and 3)
				to (9, -3) 
				to [out=180, in=90] (8, -4)
				to (8, -5); 
				
				\draw [solid] (2,-5) arc (-180: 180 : 3 and 1)
				to (2,-4)
				to [out=90, in=0] (1,-3)
				to (-10,-3);
				
				\draw [dashed] (0,-3) arc (-90: 90: 1 and 3);
				
				\draw [thick] (-6.5, 4.1) node[circle, fill=black, inner sep=1pt] {} to [out=270, in=0] (-6, 1.5);
				\draw [thick] (-9,3) to [out=300, in=180] (-8, 1.5) to (-6,1.5);
				\draw [dashed,thick] (-9,3) arc (90: 270: 1 and 3);
				\draw [thick] (-9,-3) to [out=60, in=180] (-8, -1.5) to (-6,-1.5)
				to [out=0, in=180] (-6,1)
				to [out=0, in=270] (-3,2);
				\draw [thick] (-3.5, 4.1) node[circle, fill=black, inner sep=1pt] {} to [out=270, in=90] (-3, 2);
				
				\draw [thick] (6.5, -4.1) node[circle, fill=black, inner sep=1pt] {} to [out=90, in=180] (7.3, -2)
				to (10.75,-2) node[circle, fill=black, inner sep=1pt] {};
				\draw [thick] (3.5, -4.1) node[circle, fill=black, inner sep=1pt] {} to [out=90, in=180] (4.3, -1.5)
				to [out=0, in=270] (7.3,-1) 
				to [out=90, in=180] (7.5, 1)
				to (10.95,1) node[circle, fill=black, inner sep=1pt] {};
				
				\draw [thick, red] (-5.7, -2.5) to [out=-100, in=30] (-6, -3);
				\draw [thick, red] (-2.7, -2.5) to [out=-100, in=30] (-3, -3);
				
				\draw [thick, red] (7.3, 2) to [out=100, in=-30] (6.5, 3);
				\draw [thick, red] (4.3, 2) to [out=100, in=-30] (3.5, 3);
				\draw [thick, red] (7.3, 2)to [out=270, in=270] (4.3, 2);
				
				\draw [dashed, red, thick] (-6,-3) to [out=120, in =180] (3.5,3);
				\draw [dashed, red, thick] (-3,-3) to [out=120, in =190] (6.5,3);
				
				\draw [thick, red] (-2.5,-1.5) to (3,-1.5)
				to [out=0, in=0] (3.5,1.5)
				to (-2.2, 1.5)
				to [out=180, in=0] (-4, -1.5)
				to [out=180, in=90] (-5.7, -2.5);
				
				\draw [thick, red] (-2.5,-1.5) to [out=180, in=90] (-2.7, -2.5);
				
			\end{scope})
		\end{tikzpicture}
            \caption{A chain}
            \label{fig:cyclic_chain}

	\end{figure}
  \end{example}

Moreover, each bead-chain configuration is constrained within an annular neighborhood $\mathbb{A}_\ell$, which contains all loop components of $\mathsf{E}_{\varepsilon_+}$ isotopic to $\ell$ (Lemma \ref{lem:regular-neighborhoods} and Lemma \ref{lem:null-loop-cyclic}). We prove that the configurations in distinct neighborhoods can be constructed independently (Proposition \ref{prop:indep-config}). Therefore, the construction of $\mathsf{E}_{\varepsilon_+}$ reduces to separate constructions with in these annuli. Within each annulus, the problem involves only a cluster algebra associated with a quiver of type $\mathbb{A}_{m,m}$, allowing this local construction to be completely understood (Section \ref{sec:annulus}).

It is natural to ask what patterns the unknots of $\mathsf{E}_\varepsilon$ might exhibit when $\varepsilon$ is not restricted to $\mathcal{E}_{\mathrm{top}}$. In {Appendix \ref{sec:nested}}, we describe one such phenomenon: certain unknots form a nested sequence.

\subsection{Convention}
Throughout this paper, we fix the following notations and assumptions:
\begin{itemize}
    \item \textbf{Quantum parameter}: We fix a formal variable $q^{1/2}$ and denote $\mathbb{Z}_q := \mathbb{Z}\big[q^{\pm 1/2 }\big]$.

    \item \textbf{Surfaces}: $\Sigma = (\mathcal{S}, \mathcal{M})$ is an unpunctured, compact, oriented surface where $\partial \mathcal{S} \neq \emptyset$ and each boundary component contains $\geq 1$ marked point. We exclude disks with $\leq 2$ marked points.

    \item \textbf{Skein elements}: $[\mathrm{X}]$ is the equivalence class of a link $\mathrm{X}$ in the skein algebra $\mathrm{Sk}_q(\Sigma)$. For a simple multicurve $\mathsf{X}$, we denote $\langle \mathsf{X} \rangle \coloneqq \mathrm{Band}(\mathsf{X})$ and $U(\mathsf{X}) \coloneqq \langle \mathsf{X} \rangle$ if $\mathsf{X}$ consists of loops.
    
    \item \textbf{Seeds}: $\Delta^+$ is the set of seeds mutation equivalent to an initial seed $\mathbf{s}_0$. 
    
    \item \textbf{Bar involution}: The involutive ring anti-automorphism $\overline{(\cdot)}$ on $\mathrm{Sk}_q(\Sigma)$ satisfies $q^{1/2} \mapsto q^{-1/2}$ and $[\mathrm{X}] \mapsto [\overline{\mathrm{X}}]$.

    \item \textbf{Figures}: Several figures in the following sections are adapted from the TikZ source code of Muller \cite{muller2016skein}.
\end{itemize}

\section{Preliminaries}
 	    \subsection{The skein algebra}\label{sec:skein-alg}
 	    In this section, we briefly review the fundamental concepts of skein algebras on marked surfaces. For detailed definitions and proofs, we refer the reader to \cite{muller2016skein}.
 	    
	    Let $\mathcal{S}$ be a compact oriented surface with boundary, and $\mathcal{M} \subset \partial\mathcal{S}$ a finite set of marked points. The pair $\Sigma= ( \mathcal{S}, \mathcal{M})$ is called a marked surface (unpunctured).
	
	    A \textbf{curve} $\mathsf{x}$ in $\mathcal{S}$ is an immersion $\mathsf{x}: C \rightarrow \mathcal{S}$ from a compact connected $1$-manifold $C$ to $\mathcal{S}$ such that the boundary of $C$ maps into $\mathcal{M}$ and the interior of $C$ does not map to $\mathcal{M}$ or $\partial \mathcal{S}$. Two curves are \textbf{homotopic} if they can be related by a homotopy respecting $\mathcal{M}$ and orientation reversal. An \textbf{arc} is a curve with endpoints in $\mathcal{M}$, and a \textbf{loop} is a closed curve with no endpoints. An arc is a \textbf{boundary arc} if it is homotopic to the closure of a component of $\partial \mathcal{S}\backslash \mathcal{M}$. A \textbf{multicurve} $\mathsf{X}$ is an unordered finite collection of (possibly homotopic) curves in $\mathcal{S}$. Two multicurves are \textbf{homotopic} if there is a bijection between their constituent curves which takes a curve to a homotopic one. 

	    A \textbf{strand} in a multicurve $\mathsf{X}$ at a point $p\in \mathcal{S}$ is a connected component of $\mathsf{X} \cap \mathcal{D}_p$, where $\mathcal{D}_p$ is a sufficiently small disk neighborhood of $p$. A multicurve $\mathsf{X}$ is \textbf{transverse} if the strands have distinct tangent directions at each intersection, and each interior intersection (a \textbf{crossing}) involves exactly two strands. Moreover, every multicurve is homotopic to a transverse multicurve.     
	
	    A transverse multicurve is \textbf{simple} if it has no crossings and no contractible curves. Here, a \textbf{contractible curve} is either a null-homotopic loop (an \textbf{unknot}) or an arc that bounds a disk in $\mathcal{S}$ (a \textbf{contractible arc}).
	
	    A \textbf{link} $\mathrm{X}$ is a transverse multicurve $\mathsf{X}$ equipped with an ordering (over/under) of the two strands at each crossing, as well as an equivalence relation and a total ordering on the resulting equivalence classes of strands at each $p \in \mathcal{M}$. A simple multicurve $\mathsf{X}$ can be regarded as a link with a simultaneous ordering at each endpoint, also denoted by $\mathsf{X}$. Homotopies between links are performed within the class of transverse multicurves, preserving crossing data. 
	
	    Let $\mathbb{Z}_q := \mathbb{Z}\big[q^{\pm 1/2 }\big]$. 
	    Denote by $\mathbb{Z}_q^{\mathrm{Links}}(\Sigma)$ the free $\mathbb{Z}_q$-module 
	    spanned by the equivalence classes of links in $\Sigma$. 
	    The skein module is defined as the quotient
	    \[
	    \mathrm{Sk}_q(\Sigma) := \mathbb{Z}_q^{\mathrm{Links}}(\Sigma) \big/ I,
	    \]
	    where $I$ is generated by the following relations:
        
        \begin{itemize}
        	\item The \textbf{Kauffman skein relation}
        	\begin{center}
        		\begin{tikzpicture}[scale=1.4]
        			\path[use as bounding box] (-1.15in,-.3in) rectangle (1.15in,.3in);
        			\begin{scope}[xshift=-.85in,scale=.15]
        				\draw[thick] (-2.83,-2.83) to (2.83,2.83);
        				\draw[thick] (-2.83,2.83) to (-.71,.71);
        				\draw[thick] (.71,-.71) to (2.83,-2.83);
        			\end{scope}
        			\node (=) at (-.5in,0) {$=$};
        			\node (q) at (-.375in,0) {$q$};
        			\begin{scope}[xshift=-.05in,scale=.15]        		
        				\draw[thick] (-2.83,-2.83) to [out=45,in=-45] (-2.83,2.83);
        				\draw[thick] (2.83,-2.83) to [out=135,in=-135] (2.83,2.83);
        			\end{scope}
        			\node (+) at (.3in,0) {$+$};
        			\node (q') at (.5in,.02in) {$q^{-1}$};
        			\begin{scope}[xshift=.85in,scale=.15]
        				\draw[thick] (-2.83,-2.83) to [out=45,in=135] (2.83,-2.83);
        				\draw[thick] (-2.83,2.83) to [out=-45,in=-135] (2.83,2.83);
        			\end{scope}
        		\end{tikzpicture}
        	\end{center}
            The first link on the right is called the \textbf{positive smoothing} of the crossing, and the second is called the \textbf{negative smoothing}. 
        	
        	\item The \textbf{boundary skein relation}
        	\begin{center}
        		\begin{tikzpicture}[scale=1.4]
        			\path[use as bounding box] (-.8in,-.3in) rectangle (1.6in,.3in);
        			\node (q) at (-.625in,0.04in) {$q^{-1/2}$};
        			\begin{scope}[xshift=-.25in,scale=.15]	    	
        				\clip (0,0) circle (4);
        				\draw[thick] (-5,-3) to [in=180,out=30] (0,-2) to [in=150,out=0] (5,-3) to [line to] (5,5) to (0,5) to (-5,5);
        				\node (1) at (0,-2) [marked] {};
        				\draw[thick] (1) to (4,3);
        				\draw[thick] (-.8,-1) to (-4,3);
        			\end{scope}
        			\node (=) at (.1in,0) {$=$};
        			\begin{scope}[xshift=.45in,scale=.15]
        				\clip (0,0) circle (4);
        				\draw[thick] (-5,-3) to [in=180,out=30] (0,-2) to [in=150,out=0] (5,-3) to [line to] (5,5) to (0,5) to (-5,5);
        				\node (1) at (0,-2) [marked] {};
        				\draw[thick] (1) to (4,3);
        				\draw[thick] (1) to (-4,3);
        			\end{scope}
        			\node (=) at (.8in,0) {$=$};
        			\node (q) at (.975in,0.04in) {$q^{1/2}$};
        			\begin{scope}[xshift=1.3in,scale=.15]
        				\clip (0,0) circle (4);
        				\draw[thick] (-5,-3) to [in=180,out=30] (0,-2) to [in=150,out=0] (5,-3) to [line to] (5,5) to (0,5) to (-5,5);
        				\node (1) at (0,-2) [marked] {};
        				\draw[thick] (1) to (-4,3);
        				\draw[thick] (.8,-1) to (4,3);
        			\end{scope}
        		\end{tikzpicture}
        	\end{center}
            The boundary skein relation describes the exchange of strands at a marked point. 
        	
        	\item The \textbf{value of the unknot}
        	\begin{center}
        		\begin{tikzpicture}[scale=1.5]
        			\path[use as bounding box] (-1.15in,-.3in) rectangle (.9in,.3in);
        			\begin{scope}[xshift=-.85in,scale=.15]
        				\draw[thick] (0,0) circle (2);
        			\end{scope}
        			\node (=) at (-.5in,0) {$=$};
        			\node (2) at (-.05in,.02in) {$-(q^2+q^{-2})$};
        			\begin{scope}[xshift=.6in,scale=.15]
        			\end{scope}
        		\end{tikzpicture}
        	\end{center}
        	
        	\item The \textbf{value of a contractible arc}
        	\begin{center}
        		\begin{tikzpicture}[scale=1.5]
        			\path[use as bounding box] (-1.4in,-.3in) rectangle (.875in,.3in);
        			\begin{scope}[xshift=-1.1in,scale=.15]
        				\clip (0,0) circle (4);
        				\draw[thick] (-5,-3) to [in=180,out=30] (0,-2) to [in=150,out=0] (5,-3) to [line to] (5,5) to (0,5) to (-5,5);
        				\node (1) at (0,-2) [marked] {};
        				\draw[thick] (1) to [out=45,in=0] (0,2) to [out=180,in=135] (-.8,-1);
        			\end{scope}
        			\node (='') at (-.75in,0) {$=$};
        			\begin{scope}[xshift=-.4in,scale=.15]
        				\clip (0,0) circle (4);
        				\draw[thick] (-5,-3) to [in=180,out=30] (0,-2) to [in=150,out=0] (5,-3) to [line to] (5,5) to (0,5) to (-5,5);
        				\node (1) at (0,-2) [marked] {};
        				\draw[thick] (1) to [out=45,in=0] (0,2) to [out=180,in=135] (1);
        			\end{scope}
        			\node (=) at (-.05in,0) {$=$};
        			\begin{scope}[xshift=.3in,scale=.15]
        				\clip (0,0) circle (4);
        				\draw[thick] (-5,-3) to [in=180,out=30] (0,-2) to [in=150,out=0] (5,-3) to [line to] (5,5) to (0,5) to (-5,5);
        				\node (1) at (0,-2) [marked] {};
        				\draw[thick] (1) to [out=135,in=180] (0,2) to [out=0,in=45] (.8,-1);
        			\end{scope}
        			\node (=') at (.65in,0) {$=$};
        			\node (0) at (.775in,.01in) {$0$};
        		\end{tikzpicture}
        	\end{center}
        \end{itemize}
	    
	    For a link $\mathrm{X}$, its equivalence class in $\mathrm{Sk}_q(\Sigma)$ is denoted by $[\mathrm{X}]$. The $\mathbb{Z}_q$-module $\mathrm{Sk}_q(\Sigma)$ admits a $\mathbb{Z}_q$-bilinear, non-commutative superposition product. Let $\mathrm{X}$ and $\mathrm{Y}$ be links such that the union of the underlying  multicurves $\mathsf{X} \cup \mathsf{Y}$ is transverse. Define the \textbf{superposition} $\mathrm{X}\cdot \mathrm{Y}$ to be the link with underlying multicurve $\mathsf{X} \cup \mathsf{Y}$, where at each new crossing, the strand of $\mathrm{X}$ crosses over that of $\mathrm{Y}$, and all other crossings retain their original ordering from $\mathrm{X}$ and $\mathrm{Y}$. Since $[\mathrm{X} \cdot \mathrm{Y}]$ depends only on the homotopy classes of $\mathrm{X}$ and $\mathrm{Y}$, the \textbf{superposition product} is defined on equivalence classes by:
	     $$[\mathrm{X}][\mathrm{Y}]:=[\mathrm{X}'\cdot \mathrm{Y}'],$$
	     where $\mathrm{X}' \sim \mathrm{X}$ and $\mathrm{Y}' \sim \mathrm{Y}$ are representatives such that $\mathrm{X}'\cup \mathrm{Y}'$ transverse. 
	    
	    This product endows $\mathrm{Sk}_q(\Sigma)$ with the structure of an associative $\mathbb{Z}_q$-algebra with unit $[\emptyset]$ (the empty link), called the \textbf{Kauffman skein algebra} of $\Sigma$. The \textbf{localized skein algebra} $\mathrm{Sk}^{\circ}_q(\Sigma)$ 
	    is the Ore localization of $\mathrm{Sk}_q(\Sigma)$ 
	    at the multiplicative set generated by the boundary arcs.
	
	    Given a link $\mathrm{X}$, let $\overline{\mathrm{X}}$ denote the link 
	    with the same underlying multicurve but with all crossing data reversed. There exists an involutive ring anti-automorphism $\overline{(\cdot)}: \mathrm{Sk}_q(\Sigma) \to \mathrm{Sk}_q(\Sigma)$, called the \textbf{bar involution}, defined by:
	    \begin{align*}
	    	q^{1/2} &\mapsto q^{-1/2}, \\
	    	[\mathrm{X}] &\mapsto [\overline{\mathrm{X}}],
	    \end{align*}
	    for all links $\mathrm{X}$. Elements fixed by the bar involution are called \textbf{bar-invariant}.
	 
	    \subsection{Weighted simple multicurves and bangles} \label{sec:wt-simple-multi-curve}
	    Let $\mathrm{SMulti}(\Sigma)$ denote the set of homotopy classes of simple multicurves. For any simple multicurve $\mathsf{X}$, there exist finitely many pairwise non-homotopic and non-intersecting simple curves $\mathsf{x}_1,\ldots,\mathsf{x}_k$ such that $\mathsf{X}$ is homotopic to the union $\bigcup_{i=1}^k w_i \mathsf{x}_i$ for some integers $w_i \geq 1$. We write $\mathsf{X} = \bigcup_{i=1}^k w_i \mathsf{x}_i$ to denote this decomposition.
        

        A \textbf{weighted crossingless multicurve} is a collection of pairwise non-homotopic and non-intersecting curves $\mathsf{x}_1,\ldots,\mathsf{x}_k$, each assigned an integer weight $w_i$, such that the weights on non-boundary curves are positive. We identify weight-$0$ curves with the empty set $\emptyset$. It is further called a \textbf{weighted simple multicurve} if all components are non-contractible. Denote by $\mathrm{SMulti}^\circ(\Sigma)$ the set of homotopy classes of weighted simple multicurves.
	    
	    Any simple multicurve $\mathsf{X} = \bigcup_{i=1}^k w_i \mathsf{x}_i$ defines an element $[\mathsf{X}] \in \mathrm{Sk}_q(\Sigma)$ given by:
	    $$
	    [\mathsf{X}] := q^{\lambda/2} \prod_{i=1}^k [\mathsf{x}_i]^{w_i},
	    $$
	    where $q^{\lambda/2}$ is the unique power of $q$ such that $[\mathsf{X}]$ is bar-invariant. This definition extends naturally to weighted multicurves.
	 
	    \begin{dfn}[Bangle]
	    	For each weighted simple multicurve $\mathsf{X}=\bigcup_{i=1}^kw_i \mathsf{x}_i$, the element $[\mathsf{X}] \in \mathrm{Sk}_q ^\circ(\Sigma)$ is called the \textbf{bangle} element associated with $\mathsf{X}$ (or simply a \textbf{bangle}), denoted by $\mathrm{Bang}(\mathsf{X}) := [\mathsf{X}]$.
	    \end{dfn}
	    
	    Let $\mathrm{Bang}_q(\Sigma)$ denote the set of bangles with non-negative weights, 
	    and $\mathrm{Bang}_q^\circ(\Sigma)$ the set of all bangles.
        
	    \begin{lem}[{\cite[Proposition 4.10]{thurston2014positive}}, {\cite[Lemma 4.1]{muller2016skein}}] \label{lem:bangle-bases}
	    	The set of bangles $\mathrm{Bang}_q(\Sigma)$ forms a $\mathbb{Z}_q$-basis for $\mathrm{Sk}_q(\Sigma)$. Similarly, $\mathrm{Bang}_q^\circ(\Sigma)$ is a $\mathbb{Z}_q$-basis for $\mathrm{Sk}_q^\circ(\Sigma)$.
	    \end{lem}

        By Lemma \ref{lem:bangle-bases}, the homotopy classes of multicurves form a $\mathbb{Z}_q$-basis. We shall also consider their $q^{\mathbb{Z}}$-shifts as a $\mathbb{Z}$-basis. For convenience, we refer to this $\mathbb{Z}$-basis as the \textbf{shifted multicurve basis}. 
	    
	    \subsection{Bands}\label{sec:band}
	    The \textbf{Chebyshev polynomials of the second kind} $\{U_k(z)\}_{k \in \mathbb{N}}$ are defined recursively by:
	    $$U_0(z)=1,$$ 
	    $$U_1(z)=z,$$
	    $$U_{k+1}(z)=zU_k(z)-U_{k-1}(z).$$
	    
	    These polynomials satisfy the following identities for all $k, l \in \mathbb{N}$:
	    $$U_k(z)U_l(z)=U_{k+l}(z)+U_{k+l-2}(z)+\cdots+U_{|k-l|}(z), $$
	    $$U_k(e^x+e^{-x})=e^{kx}+e^{(k-2)x}+\cdots+e^{-kx}.$$
	
	    \begin{defn}[Band]
	    	\label{def:band-element}
	        Given a weighted simple multicurve $\mathsf{X} = \bigcup_{i=1}^k w_i \mathsf{x}_i$, the \textbf{band} element $\mathrm{Band}(\mathsf{X})$ is defined as
	        \[
	            \mathrm{Band}(\mathsf{X}) := q^{\lambda/2} \prod_{i=1}^k \mathrm{Band}^{w_i}(\mathsf{x}_i),
	        \]
	    	where:
	    	\begin{enumerate}
	    		\item $q^{\lambda/2}$ is the unique power of $q$ such that $\mathrm{Band}(\mathsf{X})$ is {bar-invariant}. 
	    		
	    		\item Each factor $\mathrm{Band}^{w_i}(\mathsf{x}_i)$ is determined by the type of the curve $\mathsf{x}_i$:
	    		\begin{itemize}
	    			\item If $\mathsf{x}_i$ is a simple arc, then $\mathrm{Band}^{w_i}(\mathsf{x}_i) = [\mathsf{x}_i]^{w_i}$.
	    			
	    			\item If $\mathsf{x}_i$ is a simple loop, then $\mathrm{Band}^{w_i}(\mathsf{x}_i) = U_{w_i}([\mathsf{x}_i])$.
	    		\end{itemize}
	    	\end{enumerate}
	    \end{defn}
            For simplicity, we denote $\langle \mathsf{X}\rangle:=\mathrm{Band}(\mathsf{X})$. When $\mathsf{X}$ consists only of loop components, we also write $U(\mathsf{X}):=\langle \mathsf{X} \rangle$.
        
	    Denote the sets of band elements in $\mathrm{Sk}^\circ_q(\Sigma)$ by:
	    \begin{align*}
	    	\mathrm{Band}_q(\Sigma) & := \left\{ \langle\mathsf{X}\rangle \in \mathrm{Sk}^\circ_q(\Sigma) \mid \mathsf{X} \in \mathrm{SMulti}(\Sigma) \right\},
	    	\\
	    	\mathrm{Band}^\circ_q(\Sigma) & := \left\{ \langle\mathsf{X}\rangle \in \mathrm{Sk}^\circ_q(\Sigma) \mid \mathsf{X} \in \mathrm{SMulti}^\circ(\Sigma) \right\}.
	    \end{align*}
	    
	    \begin{lem}[{\cite[Proposition 4.10]{thurston2014positive}}] \label{lem:band-bases}
	    	The set $\mathrm{Band}_q(\Sigma)$ forms a $\mathbb{Z}_q$-basis for $\mathrm{Sk}_q(\Sigma)$. Similarly, $\mathrm{Band}^\circ_q(\Sigma)$ is a $\mathbb{Z}_q$-basis for $\mathrm{Sk}^\circ_q(\Sigma)$.
	    \end{lem}

           \subsection{Quantum cluster algebras}
	        A \textbf{quantum seed} of skew-symmetric type in a skew-field $\mathcal{F}$ is a triple $(B, \Lambda, M)$, where:
	        \begin{itemize}
	        	\item The \textbf{exchange matrix} $B$ is an $N \times \mathbf{ex}$ integer matrix for a subset $\mathbf{ex} \subset \{1, \ldots, N \}$, such that the \textbf{principal part} $\pi B$ is skew-symmetric, where $\pi$ is the $\mathbf{ex} \times N$ matrix projecting $\mathbb{Z}^N$ onto $\mathbb{Z}^{\mathbf{ex}}$.
	        	
	        	\item The \textbf{compatibility matrix} $\Lambda$ is an $N \times N$ skew-symmetric integer matrix satisfying the \textbf{compatibility condition} $\Lambda B = D \iota$, where $\iota$ is the $N \times \mathbf{ex}$ matrix that includes $\mathbb{Z}^{\mathbf{ex}}$ into $\mathbb{Z}^N$, and $D= \mathrm{diag}(d_{11}, \ldots, d_{NN})$ is a diagonal matrix with entries $d_{ii} > 0$.
	        	
	        	\item $ M: \mathbb{Z}^N \to \mathcal{F} \setminus \{0\} $ satisfies the multiplicative relation  
	        	$$M(\alpha) M(\beta) =q^{\frac{1}{2} \Lambda (\alpha, \beta )} M(\alpha + \beta)$$ such that the $\mathbb{Z}_q$-subalgebra generated by $M(\mathbb{Z}^N)\subset \mathcal{F}$ forms a \textbf{quantum torus} $\mathbb{T}_\Lambda$ defined by $\Lambda$, whose skew-field of fractions is $\mathcal{F}$. 
	        \end{itemize} 

	        A quantum seed $(B', \Lambda', M')$ in $\mathcal{F}$ is the \textbf{mutation} at $i \in \mathbf{ex}$ of $(B, \Lambda, M)$ if:
	        
	        \begin{itemize}
	        	\item The \textbf{exchange relation} holds:
	        	\begin{equation*}
	        		b'_{jk} = 
	        		    \begin{cases}
	        		    	-b_{jk}  &\text{if}\ i =j\ \text{or}\ i=k, \\
	        		    	b_{jk} + \frac{1}{2} (|b_{ji}| b_{ik} + b_{ji} |b_{ik}|) & \text{otherwise}.
	        		    \end{cases}
	        	\end{equation*}
	        	\item $M(\alpha) =M'(\alpha)$ for all $\alpha \in \mathbb{Z}^N$ satisfying $\alpha_i=0$.
	        	\item The \textbf{quantum cluster relation} holds:
	        	\[
	        	    M'(e_i) = M \left( -e_i + \sum_{j:b_{ji}>0} b_{ji} e_j \right)+ M \left(-e_i - \sum_{j:b_{ji}<0} b_{ji}e_j \right).
	        	\] 
	        \end{itemize}
	        
	        Two quantum seeds $(B, \Lambda, M)$ and $(B', \Lambda', M')$ are \textbf{mutation equivalent} if they can be related by a finite sequence of mutations and permutations of indices. An element of the form $M'(e_i)\in \mathcal{F}$ is called a \textbf{cluster variable} of $(B, \Lambda, M)$ if $(B', \Lambda', M')$ is mutation equivalent to $(B, \Lambda, M)$. For such a variable $M'(e_i)$:
	        
	        \begin{itemize}
	        	\item It is \textbf{mutable} if $i\in \mathbf{ex}$.
	        	\item It is \textbf{frozen} if $i \in \{1, \dots, N \} \setminus \mathbf{ex}$.
	        \end{itemize} 
	        
	        The \textbf{quantum cluster algebra} $\mathcal{A}_q$ associated with $(B, \Lambda, M)$ is the $\mathbb{Z}_q$-subalgebra of $\mathcal{F}$ generated by the cluster variables, together with the inverses of the frozen variables. The \textbf{quantum upper cluster algebra} $\mathcal{U}_q$ of a quantum seed $(B, \Lambda, M)$ is defined as the intersection of the quantum tori over all quantum seeds $(B', \Lambda', M')$ mutation equivalent to $(B, \Lambda, M)$, i.e., 
	        \[
	            \mathcal{U}_q =\bigcap _{(B', \Lambda', M') \sim (B, \Lambda, M)} \mathbb{Z}_q  M'(\mathbb{Z}^N). 
	        \]
	        
	        \begin{thm}[\cite{BerensteinZelevinsky05}]
                For any quantum seed $(B, \Lambda, M)$, we have the inclusion $$\mathcal{A}_q \subseteq \mathcal{U}_q.$$     	 
	        \end{thm}
            
               For an $n \times n$ integral skew-symmetric matrix $A$, a mutation at an index $i\in \{1,\dots, n \}$ is defined by the exchange relation given earlier. The \textbf{exchange type} of $A$ is its equivalence class under the equivalence relation generated by mutations and conjugations by permutation matrices. 
               
               For a quantum seed $\mathbf{s}=(B, \Lambda, M)$, the exchange type of $\pi B$ consists of matrices of the form $\pi B'$ for all quantum seeds $(B', \Lambda', M')$ mutation equivalent to $(B, \Lambda, M)$. The \textbf{exchange type} of a quantum seed $(B, \Lambda, M)$ is the exchange type of $\pi B$, and the \textbf{exchange type} of a cluster algebra $\mathcal{A}_q$ is defined as the exchange type of any of its quantum seeds.

               An $n\times n$ skew-symmetric matrix $A$ is \textbf{acyclic} if there is no sequence of indices $i_1, i_2,\dots, i_k=i_1$ with $k \geq 3$ such that $a_{i_{j+1}i_j} >0$ for all $j \in \{1, \dots, k-1\}$. An exchange type is \textbf{acyclic} if it contains at least one acyclic matrix.

               \begin{rem}
                   A skew-symmetric matrix $A$ is acyclic if and only if its associated quiver $Q(A)$ has no oriented cycles.
               \end{rem}

               \begin{prop}[{\cite[Theorem 7.5]{BerensteinZelevinsky05}}] \label{prop:ord-upper}
                   If $\mathcal{A}_q$ is of acyclic type, then $\mathcal{A}_q =\mathcal{U}_q$. 
               \end{prop}
            
    \subsection{Quantum cluster algebras of marked surfaces}\label{subsec:clAlg-surface}
	        
	        A marked surface $\Sigma= ( \mathcal{S}, \mathcal{M})$ is \textbf{triangulable} if it satisfies the following conditions:
	        \begin{itemize}
	        	\item $\partial \mathcal{S} \neq \emptyset$;  
	        	\item Each boundary component of $\mathcal{S}$ contains at least one marked point; 
	        	\item No connected component of $\mathcal{S}$ is a disk with $\leq 2$ marked points on its boundary.
	        \end{itemize}
      
	        An \textbf{ideal triangulation} of a triangulable $\Sigma$ is a maximal collection $\Delta$ of pairwise non-homotopic simple arcs in $\Sigma$ that are mutually non-intersecting in the interior of $\mathcal{S} \setminus \mathcal{M}$. In the following, we assume that $\Sigma$ is triangulable (i.e., it admits an ideal triangulation).
            
	        An \textbf{end} of an arc $\mathsf{x}$ is defined as a strand of $\mathsf{x}$ within a sufficiently small neighborhood of its endpoint. For any arc $\mathsf{x}$, we denote its two ends by $\partial_1(\mathsf{x})$ and $\partial_2(\mathsf{x})$.
	        
	        Given a triangulation $\Delta$ of $\Sigma$, we associate it with a quantum seed
	        \[
	            \mathbf{s}_\Delta = (B^\Delta, \Lambda^\Delta, M^\Delta),       
	        \]
	        where:
	        \begin{itemize}
	            \item $\mathbf{ex} \subset \{ 1, \ldots, N\} \simeq \Delta$ is the subset of non-boundary arcs in $\Delta$.
	            \item $B^\Delta = \mathsf{Q}^\Delta \iota$, where $\mathsf{Q}^\Delta$ is the \textbf{skew-adjacency matrix} of $\Delta$, defined for $\mathsf{x}_i, \mathsf{x}_j \in \Delta$ by:
	            \[
	                \mathsf{Q}^{\Delta}_{ij}:= \sum_{a,b \in \{ 1, 2 \}}
	                    \begin{cases}
            -1 & \text{if } \partial_a(\mathsf{x}_i) \text{ is immediately clockwise to } \partial_b(\mathsf{x}_j) \text{ at a common endpoint}, \\
            1  & \text{if } \partial_a(\mathsf{x}_j) \text{ is immediately clockwise to } \partial_b(\mathsf{x}_i) \text{ at a common endpoint}, \\
            0  & \text{otherwise}.
        \end{cases}
	            \]
	            \item $\Lambda^\Delta$ is the \textbf{orientation matrix} of $\Delta$ defined component-wise for $\mathsf{x_i}, \mathsf{x_j} \in \Delta$ by:
	            \[
	                \Lambda^\Delta_{ij} := \sum_{a,b \in \{ 1, 2 \}}
	                \begin{cases}
            -1 & \text{if } \partial_a(\mathsf{x}_i) \text{ is clockwise to } \partial_b(\mathsf{x}_j) \text{ at a common endpoint}, \\
            1  & \text{if } \partial_a(\mathsf{x}_j) \text{ is clockwise to } \partial_b(\mathsf{x}_i) \text{ at a common endpoint}, \\
            0  & \text{otherwise}.
        \end{cases}
	            \] 
	            \item $M^\Delta : \mathbb{Z}^N \to \mathcal{F}$ is defined by $M^\Delta(\alpha)=[ \Delta^\alpha]$. In particular, $M^\Delta(e_i)=[\Delta^{e_i}]=[\mathsf{x}_i]$, where $\mathsf{x}_i$ is the $i$-th arc in $\Delta$.
	        \end{itemize}
	    
	       \begin{prop}[\cite{muller2016skein}]
	           For an unpunctured marked surface $\Sigma$, the triple $(B^\Delta, \Lambda^\Delta, M^\Delta)$ is a quantum seed satisfying $\Lambda^\Delta B^\Delta= 4 \iota$.
	       \end{prop}
	       
	       Given a triangulation $\Delta$ and an interior arc $\mathsf{x} \in \Delta$, there is a unique quadrilateral $R$ in $\Delta$ with $\mathsf{x}$ as a diagonal. The \textbf{flip} of $\mathsf{x}$ in $\Delta$ is the triangulation obtained by replacing $\mathsf{x}$ with the other diagonal $\mathsf{x}'$ of $R$ while keeping the remaining arcs of $\Delta$ unchanged (see Figure~\ref{fig: flip}).
	       
	       \begin{figure}[H]
	       	\centering
	       	\begin{tikzpicture}[scale=.7]
	       		\begin{scope}[xshift=-2in, scale=.4]        				
	       			\node[marked] (1) at (-4, 4) {};
	       			\node[marked] (2) at (8, 4) {};
	       			\node[marked] (3) at (4,-4) {};
	       			\node[marked] (4) at (-8,-4) {};
	       			\draw[thick] (4) to [relative,out=-30,in=150] node[left] {$\mathsf{x}$} (2);
	       			\draw[thick] (1) to [relative,out=-30,in=150]  (2);
	       			\draw[thick] (2) to [relative,out=-30,in=150]  (3);
	       			\draw[thick] (3) to [relative,out=-30,in=150]  (4);
	       			\draw[thick] (4) to [relative,out=-30,in=150]  (1);
	       		\end{scope}
	       		\node (a) at (0,0) {$\rightarrow$};
	       		\begin{scope}[xshift=2in,scale=.4]
	       			\begin{scope}
	       				\node[marked] (1) at (-4,4) {};
	       				\node[marked] (2) at (8,4) {};
	       				\node[marked] (3) at (4,-4) {};
	       				\node[marked] (4) at (-8,-4) {};
	       				\draw[thick] (1) to [relative,out=-30,in=150] node[above] {$\mathsf{x}'$} (3);
	       				\draw[thick] (1) to [relative,out=-30,in=150]  (2);
	       				\draw[thick] (2) to [relative,out=-30,in=150]  (3);
	       				\draw[thick] (3) to [relative,out=-30,in=150]  (4);
	       				\draw[thick] (4) to [relative,out=-30,in=150]  (1);
	       			\end{scope}
	       		\end{scope}
	       	\end{tikzpicture}
	       	\caption{Flip}
	       	\label{fig: flip}
	       \end{figure}
	       
	       \begin{prop}[\cite{muller2016skein}]
	       	   For any triangulation $\Delta$ and any flip $\Delta '$ of $\Delta$ obtained from $\Delta$ by flipping a non-boundary arc $\mathsf{x}_j$, the quantum seed $(B^{\Delta'}, \Lambda^{\Delta'}, M^{\Delta'})$ is the mutation of $(B^{\Delta}, \Lambda^{\Delta}, M^{\Delta})$ at $j$.
	       \end{prop}
	        
	       \begin{prop}[{\cite[Theorem 7.15, Theorem 9.8]{muller2016skein}}]\label{prop:inclusion}
	           For any triangulable marked surface $\Sigma$, we have the inclusions
	           \[
	               \mathcal{A}_q (\Sigma) \subseteq \mathrm{Sk}_q^\circ (\Sigma) \subseteq \mathcal{U}_q (\Sigma).
	           \]
	           Moreover, if each component of $\Sigma$ contains at least two marked points, then these inclusions are isomorphisms.            	   
	       \end{prop}

        \subsection{Dominance order, degree and support}
        Let $\mathbf{s} = (B, \Lambda, M)$ be a quantum seed. Recall that the compatibility condition $\Lambda B = D \iota$ implies that $B$ has full rank, i.e., the column vectors $\{ \mathrm{col}_k B \}_{k \in \mathbf{ex}}$ are linearly independent. For simplicity, we denote $x_i(\mathbf{s}):= M^\mathbf{s}(e_i)$ for $i \in \{1, \dots, N\}$, $y_k(\mathbf{s}):=x(\mathbf{s})^{\mathrm{col}_k B^\mathbf{s}}$ for $k\in \mathbf{ex}$, and $y(\mathbf{s})^n:= x(\mathbf{s})^{B^\mathbf{s}n}$ for $n\in \mathbb{N}^{\mathbf{ex}}$. In the following, we shall omit the symbol $\mathbf{s}$ if it does not cause any ambiguity.  
        
        Let $\mathfrak{M}(\mathbf{s}):=\mathbb{Z}^N$ be the lattice with standard basis vectors $\{e_k\}_{k=1}^N$. The \textbf{dominance order} $\prec_\mathbf{s}$ on $\mathfrak{M}(\mathbf{s})$ is defined as follows:
           $g'\prec_\mathbf{s} g$ if and only if there exists a non-zero vector $n \in \mathbb{N}^{\mathbf{ex}}$ such that $g'=g+B^{\mathsf{s}}n$. 

        Denote by $\mathbb{Z}_q[\widehat{\mathbb{N}^{\mathbf{ex}}}]$ the completion of the monoid algebra $\mathbb{Z}_q[\mathbb{N}^{\mathbf{ex}}] := \mathbb{Z}_q[y^n]_{n \in \mathbb{N}^{\mathbf{ex}}}$ with respect to the maximal ideal generated by $\{ y^n \mid n \in \mathbb{N}^{\mathbf{ex}} \setminus \{0\} \}$. We define the \textbf{ring of formal Laurent series} as $\widehat{\mathbb{T}}_{\Lambda} := \mathbb{T}_{\Lambda} \otimes_{\mathbb{Z}_q[\mathbb{N}^{\mathbf{ex}}]} \mathbb{Z}_q[\widehat{\mathbb{N}^{\mathbf{ex}}}].$

Consider a formal sum $Z = \sum_{d \in \mathfrak{M}(\mathbf{s})} c_d x^d$ with $c_d \in \mathbb{Z}_q$. If the set $\{d \mid c_d \neq 0\}$ has a unique $\prec_{\mathbf{s}}$-maximal element $g$, we define its \textbf{degree} by $\deg^{\mathbf{s}}(Z) = g$. We say that $Z$ is $g$-\textbf{pointed} if $c_g = 1$. Similarly, if the set $\{d \mid c_d \neq 0\}$ has a unique $\prec_{\mathbf{s}}$-minimal element $h$, we define its \textbf{codegree} by $\mathrm{codeg}^{\mathbf{s}}(Z) = h$, and we say $Z$ is $h$-\textbf{copointed} if $c_h = 1$. $Z$ is said to be \textbf{bipointed} if it is both $g$-pointed and $h$-copointed.

For any vector $n = \sum_{k \in \mathbf{ex}} n_k f_k \in \mathbb{N}^{\mathbf{ex}}$, where $\{f_k\}_{k \in \mathbf{ex}}$ denotes the standard basis vectors of $\mathbb{Z}^{\mathbf{ex}}$, we define its \textbf{support} as
\[
    \mathrm{supp}(n) := \{ k \in \mathbf{ex} \mid n_k \neq 0 \}.
\]
For any $Z = x^g \sum_{n \in \mathbb{N}^{\mathbf{ex}}} c_n y^n \in \widehat{\mathbb{T}}_{\Lambda}$ with $c_0 \neq 0$, its \textbf{support} is defined as
\[
    \mathrm{supp}(Z) := \bigcup_{c_n \neq 0} \mathrm{supp}(n) \subseteq \mathbf{ex}.
\]
If $Z$ is a Laurent polynomial, we further define its \textbf{support dimension} $\mathrm{suppDim}(Z) \in \mathbb{N}^{\mathbf{ex}}$ by
\[
    (\mathrm{suppDim}(Z))_k := \max \{ n_k \mid c_n \neq 0 \}, \quad \forall k \in \mathbf{ex}.
\]

       \subsection{Tropical transformation} 

Fix a seed $\mathbf{s}$ and let $\mathbf{s}' = \mu_i \mathbf{s}$ be the seed obtained by mutation at $i \in \mathbf{ex}$. The \textbf{tropical transformation} $\phi_{\mathbf{s}', \mathbf{s}} : \mathfrak{M}(\mathbf{s}) \to \mathfrak{M}(\mathbf{s}')$ is a piecewise linear map defined as follows:

For $m = \sum_{k=1}^N m_k e_k^{\mathbf{s}} \in \mathfrak{M}(\mathbf{s})$, its image $m' = \phi_{\mathbf{s}', \mathbf{s}}(m) = \sum_{k=1}^N m'_k e_k^{\mathbf{s}'} \in \mathfrak{M}(\mathbf{s}')$ is given by
\[
m'_k = 
\begin{cases}
    -m_i & \text{if } k = i, \\
    m_k + b_{ki}^{\mathbf{s}} [m_i]_+ & \text{if } b_{ki}^{\mathbf{s}} \geq 0 \text{ and } k \neq i, \\
    m_k + b_{ki}^{\mathbf{s}} [-m_i]_+ & \text{if } b_{ki}^{\mathbf{s}} \leq 0 \text{ and } k \neq i,
\end{cases}
\]
where $[x]_+ := \max(x, 0)$ denotes the positive part.

For an arbitrary seed $\mathbf{s}' = \mu_{\underline{i}} \mathbf{s}$ obtained by a sequence of mutations $\underline{i} = (i_1, \dots, i_k)$, the tropical transformation $\phi_{\mathbf{s}', \mathbf{s}}$ is defined as the composition of the tropical transformations along the mutation sequence $\underline{i}$. This composition is well-defined, in the sense that it is independent of the choice of the mutation sequence connecting $\mathbf{s}$ to $\mathbf{s}'$.

We define a linear map $\psi_{\mathbf{s}', \mathbf{s}}: \mathfrak{M}(\mathbf{s}) \to \mathfrak{M}(\mathbf{s}')$ by:
\[
\psi_{\mathbf{s}', \mathbf{s}}\left( \sum_{k \in N} m_k e_k^{\mathbf{s}} \right) := \sum_{k \in N} m_k \phi_{\mathbf{s}', \mathbf{s}}(e_k^{\mathbf{s}}), \quad \forall (m_k)_{k \in N} \in \mathbb{Z}^N.
\]

Let $Z \in \widehat{\mathbb{T}}_{\Lambda^{\mathbf{s}}}$ be a $g$-pointed element. We say that $Z$ is \textbf{compatibly pointed} at seeds $\mathbf{s}$ and $\mathbf{s}'$ if $Z$ is also $\phi_{\mathbf{s}', \mathbf{s}}(g)$-pointed in $\widehat{\mathbb{T}}_{\Lambda^{\mathbf{s}'}}$. We say $Z$ is \textbf{compatibly pointed} if it is compatibly pointed at any pair of mutation-equivalent seeds $\mathbf{s}, \mathbf{s}'$.
           
A seed $\mathbf{s}$ is said to be \textbf{injective-reachable} if there exists a mutation sequence $\underline{i}$ and a permutation $\sigma$ of $\mathbf{ex}$ such that in the seed $\mathbf{s}[1] := \mu_{\underline{i}} \mathbf{s}$, the $g$-vectors of the cluster variables satisfy:
\[
\pi \left( g_{\sigma(k)}^{\mathbf{s}[1]} \right) = -f_k, \quad \forall k \in \mathbf{ex},
\]
where $\pi : \mathbb{Z}^N \to \mathbb{Z}^{\mathbf{ex}}$ is the canonical projection map. In this case, we may also denote the original seed by $\mathbf{s} = \mu_{\underline{i}} \mathbf{s}[-1]$.

Given an injective-reachable seed $\mathbf{s}$ and $g \in \mathfrak{M}(\mathbf{s})$, let $h := \psi^{-1}_{\mathbf{s}[-1], \mathbf{s}} \phi_{\mathbf{s}[-1], \mathbf{s}}(g)$. If there exists a vector $n \in \mathbb{N}^{\mathbf{ex}}$ such that
\[
    h = g + B^{\mathbf{s}} n,
\]
we then define the \textbf{support dimension} of $g$ as $\mathrm{suppDim}(g) := n$.
            
\begin{prop}[{Compatibility and support dimensions \cite[Proposition 3.4.8]{qin2019bases}}] \label{prop:compati}
    Let $\mathbf{s}$ and $\mathbf{s}[-1]$ be injective-reachable seeds, and let $z$ be a $g$-pointed Laurent polynomial with $g \in \mathfrak{M}(\mathbf{s})$. 
    \begin{enumerate}
        \item If $z$ is compatibly pointed at seeds $\mathbf{s}$ and $\mathbf{s}[-1]$, then $g$ has a support dimension. In this case, $z$ is bipointed and satisfies $\mathrm{suppDim}(z) = \mathrm{suppDim}(g)$.
        \item Conversely, if $g$ has a support dimension and $z$ is bipointed with $\mathrm{suppDim}(z) = \mathrm{suppDim}(g)$, then $z$ is compatibly pointed at seeds $\mathbf{s}$ and $\mathbf{s}[-1]$.
    \end{enumerate}
\end{prop}

\subsection{Common triangular basis}
 
    Throughout this paper, we assume that the injective-reachable condition holds for all seeds under consideration.
	
    \begin{rem}
        The injective-reachable condition is satisfied for cluster algebras arising from marked surfaces, with the exception of once-punctured closed surfaces; see \cite[Prop. 7.10]{FominShapiroThurston08}.
    \end{rem}

Let $\Delta^+$ be the set of seeds mutation equivalent to an initial seed $\mathbf{s}_0$.
    \begin{defn}[Triangular basis]
    For any seed $\mathbf{s} \in \Delta^+$, a triangular basis $\mathbf{L}^{\mathbf{s}}$ of $\mathcal{U}_q$ is a $\mathbb{Z}_q$-basis satisfying the following conditions:
    \begin{itemize}
        \item $\mathbf{L}^{\mathbf{s}}$ contains the quantum cluster monomials in seeds $\mathbf{s}$ and $\mathbf{s}[1]$. 
        \item (\textbf{Pointedness}) $\mathbf{L}^{\mathbf{s}} = \{L^{\mathbf{s}}_g \mid g \in \mathfrak{M}(\mathbf{s})\}$ such that each $L^{\mathbf{s}}_g$ is $g$-pointed.
        \item (\textbf{Bar-invariance}) Each element $L^{\mathbf{s}}_g$ is invariant under the bar-involution $\overline{(\cdot)}$.
        \item (\textbf{Degree triangularity}) For any $i \in \{1, 2, \ldots, N\}$, there exists $\alpha \in \mathbb{Z}$ such that
        \[
        q^{{\alpha}/2} x_i(\mathbf{s}) * L^{\mathbf{s}}_g \in L^{\mathbf{s}}_{g+e_i} + \sum_{g' \prec_{\mathbf{s}} g+e_i} q^{-1/2} \mathbb{Z}[q^{-1/2}] L^{\mathbf{s}}_{g'}.
        \]
    \end{itemize}
\end{defn}

\begin{defn}[Common triangular basis]
    A $\mathbb{Z}_q$-basis $\mathbf{L}$ of $\mathcal{U}_q$ is called the \textbf{common triangular basis} if, for every seed $\mathbf{s} \in \Delta^+$, $\mathbf{L}$ is a triangular basis with respect to $\mathbf{s}$.
\end{defn}

\section{Main Results}\label{sec:3}
     \subsection{Attachment Points and Angles}\label{subsec:attachment}
    
   We first introduce the preliminary concepts and constructions. Let $\mathsf{x}$ be a simple arc and $\mathsf{Y}$ a weighted simple multicurve such that the superposition $\mathsf{x}\cdot \mathsf{Y}$ is in minimal position with transverse crossings. Let $\mathsf{x}\cap \mathsf{Y}$ denote the set of crossings. For any sign function $\varepsilon: \mathsf{x} \cap \mathsf{Y} \to \{+, -\}$, the \textbf{weighted crossingless multicurve} $\mathsf{E}_\varepsilon$ is obtained by applying:
	\begin{itemize}
		\item the $+$ smoothing at each crossing $c \in \varepsilon^{-1}(+)$ (as defined in Section \ref{sec:skein-alg});
		\item the $-$ smoothing at each crossing $c \in \varepsilon^{-1}(-)$.
	\end{itemize}
    
    The multicurve $\mathsf{E}_\varepsilon$ admits the decomposition:
    \begin{align}\label{eq:decompose-E}  
    \begin{split}
    \mathsf{E}_\varepsilon=&\Arcs_\varepsilon \cup \Loops_\varepsilon \cup \Unknots_\varepsilon \cup \Unarcs_\varepsilon\\
        :=&\big(\bigcup_i {w_i} \gamma_i\big)\cup \big(\bigcup_j {w_j}\ell_j \big) \cup \big( \bigcup_k \mathsf{n}_k\big) \cup \big(\bigcup_h \gamma_h\big),  
    \end{split}
    \end{align}
	where $\gamma_i$ are pairwise non-homotopic simple arcs, $\ell_j$ are pairwise non-homotopic simple loops, $\mathsf{n}_k$ are unknots, and $\gamma_h$ are contractible arcs. The exponents satisfy $w_i, w_j > 0$, except that $w_i$ may be negative if $\gamma_i$ is a boundary arc. Furthermore, the components $\gamma_i, \ell_j, \mathsf{n}_k$, and $\gamma_h$ are pairwise disjoint. 
    
	Resolving the crossings in $\mathsf{x}\cdot \mathsf{Y}$ via the skein relations yields:
    \begin{equation}\label{eq:decompose-bd-intersection}
        [\mathsf{x}][\mathsf{Y}]=q^a \sum_{\varepsilon: \mathsf{x} \cap \mathsf{Y} \to \{+, - \}} q^{|{\varepsilon^{-1}(+)}|-|{\varepsilon^{-1}(-)}|} [\mathsf{E_{\varepsilon}}],
    \end{equation}
    where $a \in \frac{1}{2}\mathbb{Z}$ is the \textbf{boundary intersection number}, an exponent that makes the endpoints of $\mathsf{x}\cdot \mathsf{Y}$ simultaneous via the boundary skein relations.
    
	For each crossing $c\in \mathsf{x}\cap \mathsf{Y}$, let $\mathcal{D}_c$ be a small disk neighborhood at $c$ that contains only the two intersecting strands. The smoothing operation at $c$ modifies $\mathsf{x} \cdot \mathsf{Y}$ locally within $\mathcal{D}_c$. 
	The points resulting from the smoothing at $c$ are called \textbf{attachment points}, labeled $\varepsilon(c)c'$ and $\varepsilon(c)c''$ respectively (see Figure~\ref{fig:trace-pts}). These points trace the original crossing locations and their neighborhoods. Thus, the choice of labels $c'$ and $c''$ is arbitrary.
     
    \begin{figure}[H]
    	\centering
    	\begin{tikzpicture}[scale=1.5]
    		\begin{scope}[xshift=-1in]
    			\begin{scope}[xshift=-.5in,scale=.15]
    				\draw[thick] (-2.83,-2.83) to (2.83,2.83);
    				\draw[thick] (-2.83,2.83) to (-.71,.71);
    				\draw[thick] (.71,-.71) to (2.83,-2.83);
    				\node[marked] (1) at (0,0) {} node[right=1pt] {$c$};						
    			\end{scope}
    			\node (=) at (0,0) {$\mapsto$};
    			\begin{scope}[xshift=.5in,scale=.15]
    				\draw[thick] (-2.83,-2.83) to [out=45,in=-45] 
    				    node[marked, pos=0.5] {} node[left=.5pt] 
    				    {$+c'$}
    				(-2.83,2.83);
    				\draw[thick] (2.83,-2.83) to [out=135,in=-135]
    				node[marked, pos=0.5] {} node[right=.5pt] 
    				{$+c''$}
    				(2.83,2.83);
    			\end{scope}
    		\end{scope}
    		
    		\begin{scope}[xshift=1in]
    			\begin{scope}[xshift=-.5in,scale=.15]
    				\draw[thick] (-2.83,-2.83) to (2.83,2.83);
    				\draw[thick] (-2.83,2.83) to (-.71,.71);
    				\draw[thick] (.71,-.71) to (2.83,-2.83);
    				\node[marked] (1) at (0,0) {} node[right=1pt] {$c$};
    			\end{scope}
    			\node (=) at (0,0) {$\mapsto$};
    			\begin{scope}[xshift=.5in,scale=.15]
    				\draw[thick] (-2.83,-2.83) to [out=45,in=135]
    				node[marked, =0.5] {} node[below=.5pt] 
    				{$-c''$}
    			    (2.83,-2.83);
    				\draw[thick] (-2.83,2.83) to [out=-45,in=-135]
    				node[marked, =0.5] {} node[above=.5pt] 
    				{$-c'$}
    				 (2.83,2.83);
    			\end{scope}
    		\end{scope}
    	\end{tikzpicture}
    	\caption{Attachment points}
    	\label{fig:trace-pts}
    \end{figure}
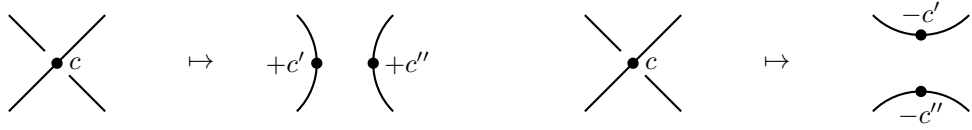

    \begin{defn}[Angle]
    	\label{def:angles}
    	Let $p = \varepsilon(c)c^\bullet$ be an attachment point for crossing $c$ where $\bullet \in \{',''\}$. A strand at $p$ determines an \textbf{angle} at $p$ within the neighborhood $\mathcal{D}_c$, represented by a small shaded sector in Figure~\ref{fig:angles}. The angle has the sign $\varepsilon(c)$, making it \textbf{positive} or \textbf{negative}. We say the strand possesses the angle, or equivalently, that the angle is attached to the strand. Two angles are \textbf{adjacent} if they share a common crossing.
    \end{defn}
    
    \begin{figure}[H]
    	\centering
    	\begin{tikzpicture}[scale=.7]
    		\begin{scope}[xshift=-2in, scale=.5]
    			\draw[dashed] (0,0) circle (4);
    			\begin{scope}
    				\clip (0,0) circle (2.75);
    				\draw[fill=black!20] (-2.83,-2.83) to [out=45,in=-45] 
    				(-2.83,2.83) to (-5,0);
    				\draw[fill=black!20] (2.83,-2.83) to [out=135,in=-135]
    				(2.83,2.83) to (5,0);
    			\end{scope}
    			\draw[thick] (-2.83,-2.83) to [out=45,in=-45] 	
    			node[marked, pos=0.5] {} node[left=2] {$+$} 
    			(-2.83,2.83);
    			\draw[thick] (2.83,-2.83) to [out=135,in=-135]
    			node[marked, pos=0.5] {} node[right=2] {$+$}
    			(2.83,2.83);
    		\end{scope}
    		
    		\begin{scope}[xshift=2in, scale=.5]
    			\draw[dashed] (0,0) circle (4);
    			\begin{scope}
    				\clip (0,0) circle (2.75);
    				\draw[fill=black!20] (-2.83,-2.83) to [out=45,in=135] 
    				(2.83,-2.83) to (0, -5);
    				\draw[fill=black!20] (-2.83,2.83) to [out=-45,in=-135]
    				(2.83,2.83) to (0, 5);
    			\end{scope}
    			\draw[thick] (-2.83,-2.83) to [out=45,in=135]
    			node[marked, =0.5] {} node[below=2] 
    			{$-$}
    			(2.83,-2.83);
    			\draw[thick] (-2.83,2.83) to [out=-45,in=-135]
    			node[marked, =0.5] {} node[above=2] 
    			{$-$}
    			(2.83,2.83);
    		\end{scope}
    	\end{tikzpicture}
    	\caption{Shaded sectors represent the angles at attachment points.}
    	\label{fig:angles}
    \end{figure}
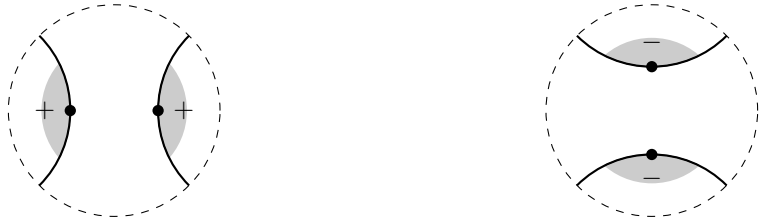

\subsection{Unknots and contractible arcs}\label{sec:3-1}
    Recall that $\mathsf{E}_\varepsilon$ is a term appearing in the expansion of $[\mathsf{x}][\mathsf{Y}]$.
	\begin{defn}[Angle orientation]
		\label{def:angle-orientation}
		Let $\mathsf{e}$ be an unknot or a contractible arc in $\mathsf{E}_\varepsilon$, 
		bounding a disk or monogon $D$. For an attachment point on $\mathsf{e}$, the associated angle is \textbf{inward} if it points into the interior of $D$, and \textbf{outward} if it points into the complement.
	\end{defn}

	\begin{lem}\label{lem:min-angle-count}
		Let $\mathsf{e}$ be a contractible component of $\mathsf{E}_\varepsilon$. 
		\begin{enumerate}
			\item If $\mathsf{e}$ is a contractible arc, it possesses at least one angle.
			\item If $\mathsf{e}$ is an unknot, it possesses at least two angles.
		\end{enumerate}
	\end{lem}
	
	\begin{proof}
		Since the contractible component $\mathsf{e}$ is not a component of the original simple multicurves $\mathsf{x}$ or $\mathsf{Y}$, it must be formed by segments of $\mathsf{x}$ and $\mathsf{Y}$ joined via the smoothing of crossings in $\mathsf{x} \cap \mathsf{Y}$. Each such junction corresponds to an attachment point. 
        
        If $\mathsf{e}$ is a contractible arc, it must contain at least one such junction, and thus possesses at least one attachment point and its associated angle. If $\mathsf{e}$ is an unknot, the simplicity of $\mathsf{x}$ and $\mathsf{Y}$ implies that $\mathsf{e}$ cannot be formed from a single segment; it must consist of at least two segments joined at two or more distinct attachment points. Consequently, $\mathsf{e}$ possesses at least two distinct angles.
	\end{proof}

       To describe the merging and splitting of components across different $\mathsf{E}_\varepsilon$, we introduce the following notions.
	\begin{defn}[Strip]
		\label{def:strip}
        Let $\mathsf{e}$ be a contractible component (either an unknot or a contractible arc) of $\mathsf{E_{\varepsilon}}$. If for some crossing $c$, the intersection $\mathsf{e} \cap \mathcal{D}_c$ contains two outward angles, then $\mathsf{e} \cap \mathcal{D}_c$ is called a \textbf{strip}. In this case, we say $\mathsf{e}$ {admits a strip} at crossing $c$.
	\end{defn}
	
    \begin{defn}[Strip operations]
    	\label{def:strip-ops}
    	Let $\mathsf{e}$ be a contractible component of $\mathsf{E}_{\varepsilon}$ 
    	admitting a strip at crossing $c$.
    	\begin{enumerate}
    		\item \textbf{Strip removing}: 
    		Replace the smoothing at $c$ with the alternative one to obtain:
    		\begin{itemize}
    			\item Two unknots $\mathsf{e}_1, \mathsf{e}_2$, if $\mathsf{e}$ is an unknot;
    			\item A contractible arc $\mathsf{e}_1$ and an unknot $\mathsf{e}_2$, if $\mathsf{e}$ is a contractible arc.
    		\end{itemize}
    		
    		\item \textbf{Strip attaching}: 
    		The inverse operation that recovers $\mathsf{e}$ from $\mathsf{e}_1$ and $\mathsf{e}_2$ 
    		by switching the smoothing at $c$.
    	\end{enumerate} 
    \end{defn} 
    Figure~\ref{fig:strip-strip ops} illustrates these operations, where the hatched regions represent the disks or monogons bounded by contractible components. Note that if $\mathsf{e}$ is a component of $\mathsf{E}_{\varepsilon}$, then $\mathsf{e}_1$ and $\mathsf{e}_2$ are components of the multicurve $\mathsf{E}_{\varepsilon'}$ obtained by switching the smoothing at $c$.
    
    \begin{figure}[H]
    	\centering
    	\begin{tikzpicture}[scale=.7]
    		\begin{scope}[xshift=-2in, scale=.5]
    			\begin{scope}
    				\clip (0,0) circle (2.75);
    				\draw[fill=black!20] (-2.83,-2.83) to [out=45,in=-45] 
    				(-2.83,2.83) to (-5,0);
    				\draw[fill=black!20] (2.83,-2.83) to [out=135,in=-135]
    				(2.83,2.83) to (5,0);
    			\end{scope}
             \clip (0,0) circle (4);
    		 \draw[pattern=north east lines, pattern color=black!40, thick] (-2.83,-2.83) to [out=45,in=-45] 	
    			node[marked, pos=0.5] {} 
    			(-2.83,2.83) to (-2.83, 5) to (2.83, 5)
    			to (2.83,2.83) to [out=-135,in=135]
    			node[marked, pos=0.5] {} 
    			(2.83,-2.83) to (2.83, -5) to (0, -5);
    		\draw[dashed] (0,0) circle (4);
    		\end{scope}
 \node[align=center] (=) at (0,0) {$\xrightarrow{\text{strip removing}}$\\\\$\xleftarrow{\text{strip attaching}}$};   		
    		\begin{scope}[xshift=2in, scale=.5]  			
    			\begin{scope}
    				\clip (0,0) circle (2.75);
    				\draw[fill=black!20] (-2.83,-2.83) to [out=45,in=135] 
    				(2.83,-2.83) to (0, -5);
    				\draw[fill=black!20] (-2.83,2.83) to [out=-45,in=-135]
    				(2.83,2.83) to (0, 5);
    			\end{scope}
    			\clip (0,0) circle (4);
    			\draw[pattern=north east lines, pattern color=black!40, thick] (-2.83,-2.83) to [out=45,in=135]
    			node[marked, =0.5] {} 
    			(2.83,-2.83) to (2.83, -5) to (-2.83, -5);
    			\draw[pattern=north east lines, pattern color=black!40, thick] (-2.83,2.83) to [out=-45,in=-135]
    			node[marked, =0.5] {} 
    			(2.83,2.83) to (2.83, 5) to (-2.83, 5);
                \draw[dashed] (0,0) circle (4);
    		\end{scope}
    	\end{tikzpicture}
    	\caption{Strips and strip operations}
    	\label{fig:strip-strip ops}
    \end{figure}
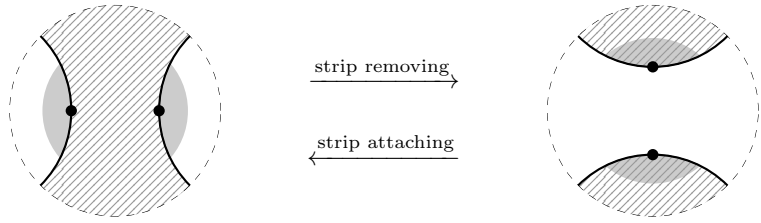
    
   	Let $\mathsf{e}$ be a contractible component of $\mathsf{E}_{\varepsilon}$, and let $D$ be the disk or monogon bounded by $\mathsf{e}$. For ease of exposition, we recall the tubular neighborhood decomposition of $D$ in \cite[Proof of Lemma C.1]{muller2016skein}. This decomposition serves as a convenient framework; alternatively, the arguments below could be carried out by only analyzing the attachment points. Specifically, we choose a thin tubular neighborhood $\mathcal{T}$ of $\mathsf{x}$ satisfying:
    \begin{enumerate}
   	    \item $\mathsf{Y}$ intersects $\mathcal{T}$ a minimal number of times.
   	    \item $\mathcal{T}$ contains the neighborhoods $\mathcal{D}_c$ for all crossings $c \in \mathsf{x} \cap \mathsf{Y}$.
   \end{enumerate}
   
   \begin{figure}[H]
   	   \centering
   	   \begin{tikzpicture}[scale=1.1]
   	   	   \begin{scope}[xshift=-1.5in]
   	   	   \begin{scope}[xshift=-.75in,scale=.15]
   	   	   	\draw[fill=gray!10,dashed] (-4,-4)  to (4,-4) arc (-90:90:4) to (-4,4) arc (90:270:4);
   	   	   	\draw[thick] (-4,4) to (-4,1);
   	   	   	\draw[thick] (-4, -4) to (-4,-1);
   	   	   	\draw[thick] (4,4) to (4,1);
   	   	   	\draw [thick] (4,-4) to (4,-1);
   	   	   	\draw[thick] (-8,0) to (8,0);
   	   	   \end{scope}
   	   	   \node (=) at (0,0) {$\mapsto$};
   	   	   \begin{scope}[xshift=.75in,scale=.15]
   	   	   	\draw[fill=gray!10,dashed] (-4,-4)  to (4,-4) arc (-90:90:4) to (-4,4) arc (90:270:4);
   	   	   	\draw[thick] (-8,0) to [out=0,in=270] (-4,4);
   	   	   	\draw[thick] (-4,-4) to [out=90,in=180] (0,0) to [out=0,in=270] (4,4);
   	   	   	\draw[thick] (4,-4) to [out=90,in=180] (8,0);
   	   	   \end{scope}
   	   	   \end{scope}
   	   	   
   	   	   \begin{scope}[xshift=1.5in]
   	   	   \begin{scope}[xshift=-.75in,scale=.15]
   	   	   	\draw[fill=gray!10,dashed] (-4,-4)  to (4,-4) arc (-90:90:4) to (-4,4) arc (90:270:4);
   	   	   	\draw[thick] (-4,4) to (-4,1);
   	   	   	\draw[thick] (-4, -4) to (-4,-1);
   	   	   	\draw[thick] (4,4) to (4,1);
   	   	   	\draw [thick] (4,-4) to (4,-1);
   	   	   	\draw[thick] (-8,0) to (8,0);
   	   	   \end{scope}
   	   	   \node (=) at (0,0) {$\mapsto$};
   	   	   \begin{scope}[xshift=.75in,scale=.15]
   	   	   	\draw[fill=gray!10,dashed] (-4,-4)  to (4,-4) arc (-90:90:4) to (-4,4) arc (90:270:4);
   	   	   	\draw[thick] (-8,0) to [out=0,in=270] (-4,4);
   	   	   	\draw[thick] (-4,-4) to [out=90,in=180] (0,0) to [out=0,in=90] (4,-4);
   	   	   	\draw[thick] (4,4) to [out=270,in=180] (8,0);
   	   	   \end{scope}
   	   	   \end{scope}
   	   \end{tikzpicture}
   	   \caption{Local pictures of $\mathcal{T}$ after smoothing, up to reflection across $\mathsf{x}$}
       \label{Fig:local-pic}
   \end{figure}
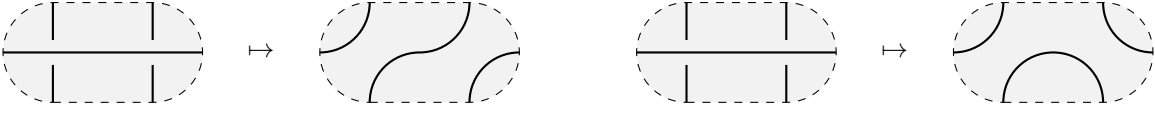
   
   We define a graph $\Gamma$ as follows:
   \begin{itemize}
   	   \item \textbf{Vertices} correspond to connected components of $D \setminus \partial\mathcal{T}$.
   	   \item \textbf{Edges} connect vertices whose corresponding components share a common boundary in $D \cap \partial\mathcal{T}$.
   \end{itemize}
   It follows that $\Gamma$ is a tree (\cite[Proof of Lemma C.1]{muller2016skein}). Since $\mathsf{e} \cap \mathcal{T} \not= \emptyset$  but $\mathsf{e} \not\subset \mathcal{T}$, the following properties hold:
   \begin{enumerate}
   	   \item There is at least one $\mathcal{T}$-intersecting component.
   	   \item There is at least one $\mathcal{T}$-disjoint component.
   	   \item $\Gamma$ contains at least two leaves (vertices of degree $1$).
   \end{enumerate}
   
   Each leaf component $D_0 \subset D \setminus \partial\mathcal{T}$ is either Type A or Type B, as illustrated in Figure~\ref{fig:leaf-components}. In the figure, $\mathcal{T}$ is shown in gray and $D_0$ is the hatched region:
   \begin{description}
      	   \item[Type A] \textbf{$D_0 \subset \mathcal{T}$ and $\partial D_0$ contains no marked points.} This implies that $\partial D_0$ contains a segment from $\mathsf{x} \setminus \mathsf{Y}$ whose endpoints are attachment points with opposite signs. These are the only attachment points on $\partial  D_0$, and both associated angles are inward. Furthermore, in the clockwise order of $\partial D_0$, these signs appear in the order $(+, -)$.
       
   	   \item[Type B] \textbf{$\partial D_0$ contains a marked point.} This implies that $\mathsf{e}$ is a contractible arc. When such a component exists, it is unique.
   \end{description}
   	
   \begin{figure}[H]
   	   \centering
   	   \begin{tikzpicture}[scale=.7]
   		   \begin{scope}[xshift=-2.5in,scale=.4]
   			   \draw[dashed, fill=gray!10, thin] (-4,-4)  to (4,-4) arc (-90:90:4) to (-4,4) arc (90:270:4);
   			   \draw[thick, dashed] (-8,0) to [out=0,in=270] (-4,4);
   			   \draw[pattern=north east lines, pattern color=black!40, thick] (-4,-4) to [out=90,in=180]
   			   (0,0) to [out=0,in=90] (4,-4);
   			   \draw[thick] (-4,-4) to [out=90,in=180] node[marked, pos=.5] {} node[above left] {$+$}
   			   (0,0);
   			   \draw [thick] (0,0) to [out=0,in=90] node[marked, =.5] {} node[above right] {$-$} (4,-4);
   			   \draw[thick, dashed] (4,4) to [out=270,in=180] (8,0);
   		   \end{scope}
   			
   		   \begin{scope}[xshift=2.5in,scale=.4]
   			   \begin{scope}
   				   \clip (-4,-4)  to (4,-4) arc (-90:90:4) to (-4,4) arc (90:270:4);
   				   \draw[thin] (-12,-8) to [out=0,in=270] (-4,0) to [out=90,in=0] (-12,8) to [line to] (12,8) to (12,-8) to (-12,-8);
   				   \node[marked] (1) at (-4,0) {};
   				   \fill[pattern=north east lines, pattern color=black!40] (4,-4) to (1) to (4,4) to (4,-4);
   				   \fill[gray!10] (4,-4) arc (-90:90:4) to (4,4);
   				   \draw [thick] (4, -4) to (-4, 0) to (4, 4);
   				   \draw[dashed, thick] (4,4) to (4, -4);
   			   \end{scope}
   			   \draw[dashed, thin] (-4,-4)  to (4,-4) arc (-90:90:4) to (-4,4) arc (90:270:4);
   		   \end{scope}
   	   \end{tikzpicture}
   	   \caption{Leaf components of Type A (left) and Type B (right). $\mathcal{T}$ is shown in gray and $D_0$ is hatched.}
   	   \label{fig:leaf-components}
   \end{figure}
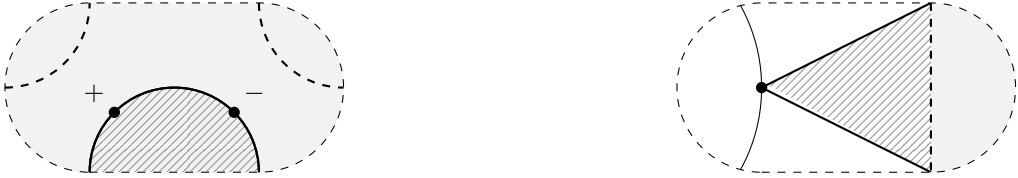
   	
   The tubular neighborhood decomposition yields the following constraint on the components of $\mathsf{E}_\varepsilon$:
   \begin{lem}[{\cite[Lemma C.1]{muller2016skein}}]\label{lem:Muller-C1}
       For any $\mathsf{E}_\varepsilon$, we have
       \begin{equation*}
    	   |\{\text{unknots}\}| + \frac{1}{2}|\{\text{contractible arcs}\}| \leq |\{\text{negative smoothings}\}|,
       \end{equation*}
       where equality holds if and only if every strand of each $-$ smoothing belongs to a contractible component.
   \end{lem}

\begin{lem}\label{lem:unknot-neg-angles}
    Let $\mathsf{C}$ be an unknot component of some multicurve $\mathsf{E}_{\varepsilon}$. Then $\mathsf{C}$ possesses at least two negative attachment points.
\end{lem}
\begin{proof}
    Let $\mathcal{T}$ be a thin tubular neighborhood of $\mathsf{x}$, and let $D$ be the disk bounded by $\mathsf{C}$. As $\Gamma$ is a tree and $\mathsf{C}$ is not contained in $\mathcal{T}$, the decomposition $D \setminus \partial\mathcal{T}$ must have at least two leaf components. Since $\mathsf{C}$ has no marked points, every such leaf component is of Type A. Each Type A component contains exactly one negative attachment point according to its $(+, -)$ sign configuration. Consequently, $\mathsf{C}$ possesses at least two negative attachment points.
\end{proof}

The proof of Lemma~\ref{lem:unknot-neg-angles} also yields the following result.

\begin{prop}\label{prop:unknot-attatchmentpoints}
Let $\mathsf{C}$ be a contractible component of some multicurve $\mathsf{E}_{\varepsilon}$. If $\mathsf{C}$ is an unknot, it possesses at least four attachment points. If $\mathsf{C}$ is a contractible arc, it possesses at least two attachment points.
\end{prop}

\begin{proof}
Let $D$ be the disk or monogon bounded by $\mathsf{C}$, and let $\mathcal{T}$ be a thin tubular neighborhood of $\mathsf{x}$. Note that $D\setminus\partial \mathcal{T}$ has at least two leaf components if $D$ is a disk, or at least one leaf component of Type A if $D$ is a monogon. Since each leaf component of Type A contains exactly two attachment points, $\mathsf{C}$ possesses at least four attachment points in the former case and at least two in the latter. The desired claim follows.
\end{proof}
  Consequently, we obtain the following:
        \begin{itemize}
            \item If an unknot $\mathsf{e}$ possesses exactly two attachment points, then $\mathsf{x}$ and $\mathsf{Y}$ are not in minimal position.
            \item If a contractible arc $\mathsf{e}$ possesses exactly one attachment point, then $\mathsf{x}$ and $\mathsf{Y}$ are not in minimal position.
        \end{itemize} 

\begin{lem}\label{lem:leading-term-no-new-loop}
    Assume $\mathsf{x}$ and $\mathsf{Y}$ are in minimal position. Let $\ell$ be a loop in the marked surface $\Sigma$ that is disjoint from $\mathsf{x} \cup \mathsf{Y}$ and not freely homotopic to any component of $\mathsf{x} \cup \mathsf{Y}$. Then $\mathsf{E}_{\varepsilon_+}$ does not contain any loop component $\ell'$ that is isotopic to $\ell$.
\end{lem}

\begin{proof}
    Suppose, for contradiction, that $\mathsf{E}_{\varepsilon_+}$ contains a loop component $\ell'$ such that $\ell' \simeq \ell$. Since $\ell \cap \mathsf{x} = \emptyset$, we have the intersection number $\mu(\ell, \mathsf{x}) = 0$, and by homotopy invariance, $\mu(\ell', \mathsf{x}) = 0$.

    Consider the local picture of $\ell'$ within a tubular neighborhood $\mathcal{T}$ of $\mathsf{x}$. By construction, $\ell'$ is composed of segments of $\mathsf{x}$ and $\mathsf{Y}$ reconnected via $+$ smoothings at each crossing $c \in \mathsf{x} \cap \mathsf{Y}$. In particular, at least one $\mathsf{Y}\text{-}\mathsf{x}\text{-}\mathsf{Y}$ segment of $\ell' \cap \mathcal{T}$ crosses the neighborhood $\mathcal{T}$ transversely (see Figure~\ref{Fig:local-pic} (left)). Although $\ell'$ geometrically intersects $\mathsf{x}$, the condition $\mu(\ell', \mathsf{x}) = 0$ implies that they are not in minimal position. Consequently, $\ell'$ and $\mathsf{x}$ must form at least one empty bigon. Specifically, the boundary of such a bigon consists of a segment along $\ell'$ and a segment along $\mathsf{x}$, both bounded by two intersection points in $\ell' \cap \mathsf{x}$. 
    
    However, the existence of such a bigon is impossible in $\mathsf{E}_{\varepsilon_+}$. For $\ell'$ and $\mathsf{x}$ to form a bigon, the segment of $\ell'$ must enter and exit the tubular neighborhood $\mathcal{T}$ through the same boundary component. Conversely, because every crossing in $\mathsf{E}_{\varepsilon_+}$ receives a $+$ smoothing, the resulting strands of $\ell'$ strictly traverse from one side of $\mathcal{T}$ to the other, as illustrated in Figure~\ref{Fig:local-pic} (left). This yields a contradiction.

     Consequently, since no empty bigons can be formed to eliminate these geometric intersections, $\ell'$ must intersect $\mathsf{x}$ essentially. This forces $\mu(\ell', \mathsf{x}) > 0$, which contradicts the assumption that $\mu(\ell', \mathsf{x}) = 0$. Therefore, $\mathsf{E}_{\varepsilon_+}$ cannot contain any loop component isotopic to $\ell$.
    
\end{proof}

  \begin{defn}\label{def:bead-chain}
      We introduce the following terminology:
      \begin{enumerate}
          \item A \textbf{chain} is an unknot which may contain strips.
          \item A \textbf{bead} is an unknot that is strip-free.

          \item Suppose a chain possesses exactly two angles whose signs differ from the signs of all other angles. We say the chain is:
              \begin{itemize}
  			\item \textbf{cyclic} if these two distinct angles are adjacent; 
  			\item \textbf{linear} if these two distinct angles are non-adjacent.
  		\end{itemize}
      \end{enumerate}
  \end{defn}

  \begin{rem}
      The arguments in the subsequent sections are developed mainly for cyclic and linear chains.
  \end{rem}

  Building upon the definitions and properties above, we characterize the \textbf{bead-chain configuration} as follows:
  
  For a weighted crossingless multicurve $\mathsf{E}_\varepsilon$, we have:
  \[
  [\mathsf{E}_\varepsilon] = (-q^2 - q^{-2})^{|\Unknots_\varepsilon|} \delta_\varepsilon [\widehat{\mathsf{E}}_\varepsilon],
  \]
  where, via the decomposition \eqref{eq:decompose-E}:
  \begin{itemize}
  	\item $\widehat{\mathsf{E}}_\varepsilon$ is the weighted simple multicurve $\Arcs_\varepsilon\cup \Loops_\varepsilon$ obtained by removing all contractible components from $\mathsf{E}_\varepsilon$.
  	\item $|\Unknots_\varepsilon|=| \{ \text{unknot components in} \mathsf{E}_\varepsilon \} |$
  	\item $\delta_\varepsilon$ is the indicator function:
    $$  \delta_\varepsilon := \delta_{|\Unarcs_\varepsilon|, 0} = 
  \begin{cases} 
      1 & \text{if } \mathsf{E}_\varepsilon \text{ contains no contractible arcs}, \\ 
      0 & \text{if } \mathsf{E}_\varepsilon \text{ contains a contractible arc}. 
  \end{cases}$$
  \end{itemize}
  Therefore, the superposition product expands to:
\begin{equation}\label{eq:detail-dec}
    [\mathsf{x}][\mathsf{Y}] =q^a \sum_{\substack{\varepsilon: \mathsf{x} \cap \mathsf{Y}  \to \{+, -\}}} q^{|\varepsilon^{-1}(+)| - |\varepsilon^{-1}(-)|} (-q^2 - q^{-2})^{|\Unknots_\varepsilon|} \delta_\varepsilon [\widehat{\mathsf{E}}_\varepsilon].
\end{equation}
  
  Let $\varepsilon_\pm$ denote the constant sign assignment mapping every crossing to $\pm$. For any $\mathsf{E}_\varepsilon$ containing no contractible arcs, Lemma~\ref{lem:Muller-C1} yields the inequality:
  \begin{align}\label{eq:q-deg-inequality}
  |\varepsilon^{-1}(+)| - |\varepsilon^{-1}(-)| + 2|\Unknots_\varepsilon| \leq |\varepsilon^{-1}(+)| + |\varepsilon^{-1}(-)| = |\varepsilon_+^{-1}(+)|.      
  \end{align}
  The right-hand side $|\varepsilon_+^{-1}(+)|$ represents the total number of crossings $|\mathsf{x} \cap \mathsf{Y}|$. 
  We say $\mathsf{E}_\varepsilon$ \textbf{contributes to the top $q$-degree term} if equality holds in \eqref{eq:q-deg-inequality}, which is equivalent to the condition:
  \begin{equation}\label{eq:q-deg-top}
|\Unknots_\varepsilon| = |\varepsilon^{-1}(-)|.
\end{equation}
  
 We refer to the set $\Unknots_\varepsilon$ of unknot components as the \textbf{bead-chain configuration} of $\mathsf{E}_\varepsilon$.  Let $\mathcal{E}_{\mathrm{top}}$ denote the set of sign assignments $\varepsilon$ for which $[\mathsf{E}_\varepsilon]$ contributes to the top $q$-degree term in \eqref{eq:detail-dec}. We define
 $$\Unknots_{\mathrm{top}} := \bigsqcup_{\varepsilon \in \mathcal{E}_{\mathrm{top}}} \Unknots_\varepsilon,$$ 
 and make the following observations:
  
  \begin{lem}\label{lem:top-chains}
      For any sign assignment $\varepsilon \in \mathcal{E}_{\mathrm{top}}$, the following properties hold:
  	\begin{enumerate}
  		\item The set $\Unknots_\varepsilon$ consists exclusively of cyclic and linear chains (as defined in \ref{def:bead-chain}).
  		
  		\item Any configuration $\Unknots_\varepsilon$ containing linear chains can be transformed into one consisting solely of cyclic chains through strip attachments.
  		
  		\item The property of contributing to the top $q$-degree term is preserved under strip operations. That is, attaching or removing a strip in $\Unknots_\varepsilon$ yields a configuration $\Unknots_{\varepsilon'}$ for some $\varepsilon' \in \mathcal{E}_{\mathrm{top}}$.
  	\end{enumerate}
  \end{lem}
  
  \begin{proof}
  	For each unknot $\mathsf{n} \in \Unknots_{\varepsilon}$, let $k_{\mathsf{n}}$ denote the number of its negative attachment points.  By Lemma \ref{lem:unknot-neg-angles}, each unknot possesses at least two such points, i.e., $k_{\mathsf{n}} \geq 2$. Since every negative attachment point in $\mathsf{E}_\varepsilon$ belongs to exactly one unknot component, we have the following inequality:   
    $$2|\varepsilon^{-1}(-)| \geq \sum_{\mathsf{n} \in \Unknots_{\varepsilon}} k_{\mathsf{n}} \geq 2|\Unknots_\varepsilon|.$$
    
The condition for contributing to the top $q$-degree term, $|\Unknots_\varepsilon| = |\varepsilon^{-1}(-)|$, forces equality throughout. This implies
\begin{align}\label{eq:unknot-2-neg-pts}
k_{\mathsf{n}} = 2 \quad \text{for every } \mathsf{n} \in \Unknots_{\varepsilon}.    
\end{align}
According to Definition \ref{def:bead-chain}, an unknot with exactly two negative attachment points is either a cyclic or a linear chain, which proves (1).

For (2), observe that attaching a strip between two linear chains consumes exactly one negative attachment point from each, reducing the total count of negative attachment points by two. Simultaneously, this operation merges two components into one, reducing $|\Unknots_\varepsilon|$ by one. Since the equality $k_{\mathsf{n}} = 2$ is maintained for the resulting chain, we can iterate this process to transform all linear chains into cyclic chains as long as linear chains remain. This yields (2).

Regarding (3), any strip operation (attachment or removal) changes the number of unknot components $|\Unknots_\varepsilon|$ by exactly one. At the same time, the number of $-$ smoothings $|\varepsilon^{-1}(-)|$ also changes by exactly one, as a strip operation corresponds to a switching the choice of smoothing at a crossing. Thus, if the equality $|\Unknots_\varepsilon| = |\varepsilon^{-1}(-)|$ holds for $\varepsilon$, it is preserved for the modified assignment $\varepsilon'$. Consequently, $\varepsilon'$ remains in $\mathcal{E}_{\mathrm{top}}$, proving (3).
  \end{proof}

For any curve component $\gamma \subset \mathsf{E}_{\varepsilon}$, let $\mathcal{C}_{\gamma} \subset \mathsf{x} \cap \mathsf{Y}$ denote the set of crossings whose smoothing provides the attachment points on $\gamma$. For a sign function $\varepsilon \in \mathcal{E}_{\mathrm{top}}$ and a chain $\mathsf{n}\in \Unknots_\varepsilon$, we capture its combinatorial skeleton via the generating set $\mathcal{G}_\mathsf{n} \subseteq \mathcal{C}_\mathsf{n}$:
$$ \mathcal{G}_\mathsf{n} := \mathrm{Str}_\mathsf{n} \sqcup \mathcal{C}^-_\mathsf{n}. $$
Here, $\mathrm{Str}_\mathsf{n}$ denotes the subset of crossings where $\mathsf{n}$ admits a strip (see Figure~\ref{fig:strip-strip ops}). The subset $\mathcal{C}^-_\mathsf{n}:=\varepsilon^{-1}(-)\cap \mathcal{C}_\mathsf{n}$ consists of exactly two negative attachment points of $\mathsf{n}$ (see \eqref{eq:unknot-2-neg-pts}), which consists of a single crossing if $\mathsf{n}$ is cyclic, and two crossings if $\mathsf{n}$ is linear. Globally, we define:
\begin{align}\label{eq:generating_set}
    \mathrm{Str}_\varepsilon &:= \bigcup_{\mathsf{n}\in \Unknots_\varepsilon} \mathrm{Str}_\mathsf{n}, \quad \text{and} \quad \mathcal{G}_\varepsilon := \bigcup_{\mathsf{n}\in \Unknots_\varepsilon} \mathcal{G}_\mathsf{n} = \mathrm{Str}_\varepsilon \cup \varepsilon^{-1}(-). 
\end{align}
The set $\mathcal{G}_\varepsilon$ is invariant under strip operations. Removing all strips from $\Unknots_\varepsilon$ yields a pure bead configuration $\Unknots_{\varepsilon'}$ (Lemma~\ref{lem:top-chains}(2)), whose induced sign function $\varepsilon'$ satisfies: $$ (\varepsilon')^{-1}(-) = \mathcal{G}_{\varepsilon}. $$

On the other hand, for $\varepsilon_+$, let $\mathcal{L}_{\varepsilon_+}$ be a maximal collection of pairwise non-homotopic simple loops in $\mathsf{E}_{\varepsilon_+}$. For each $\ell \in \mathcal{L}_{\varepsilon_+}$, let $\mathcal{H}_\ell$ denote the set of simple loops in $\mathsf{E}_{\varepsilon_+}$ freely homotopic to $\ell$. We define its associated crossing set as:
$$ \mathcal{C}({\mathbb{A}_\ell}) := \bigcup_{\gamma \in \mathcal{H}_\ell} \mathcal{C}_{\gamma}. $$
The transitions between bead-chain configurations via strip operations induce a regular neighborhood system governed by these homotopy classes, which provides the topological foundation for localizing the proof of Theorem~\ref{thm:main}:

  \begin{lem}[Regular neighborhood system]\label{lem:regular-neighborhoods}
  	 There exists a collection of annular neighborhoods $\{\mathbb{A}_\ell\}_{\ell \in \mathcal{L}_{\varepsilon_+}}$ satisfying the following:
  	\begin{enumerate}
  		\item For each $\ell \in \mathcal{L}_{\varepsilon_+}$, the neighborhood $\mathbb{A}_\ell$ contains all components of $\mathsf{E}_{\varepsilon_+}$ homotopic to $\ell$, as well as the disk neighborhoods $\mathcal{D}_c$ for all $c \in \mathcal{C}({\mathbb{A}_\ell})$.
  		
  		\item For distinct loops $\ell \neq \ell'$, the intersection of their neighborhoods is given by the finite union:
  		\[
    \mathbb{A}_\ell \cap \mathbb{A}_{\ell'} = \bigcup_{c \in \mathcal{C}({\mathbb{A}_\ell}) \cap \mathcal{C}({\mathbb{A}_\ell})} \mathcal{D}_c.
    \]
  	\end{enumerate}
  \end{lem}
  
  \begin{proof}
 Since $\mathcal{H}_\ell$ consists of pairwise disjoint, homotopic simple loops on an oriented surface, any two loops $\ell_i, \ell_j \in \mathcal{H}_\ell$ bound an embedded annulus $\mathbb{A}'_{ij}$. The annular neighborhood $\mathbb{A}'_\ell$ is defined as the union of all such annuli $\mathbb{A}'_{ij}$ together with sufficiently thin neighborhoods of each $\ell_i$. Since the components of $\mathsf{E}_{\varepsilon_+}$ are pairwise disjoint, $\{\mathbb{A}'_\ell\}_{\ell \in \mathcal{L}_{\varepsilon_+}}$ are pairwise disjoint, and $\mathbb{A}'_\ell \cap \gamma = \emptyset$ for any component $\gamma \subset \mathsf{E}_{\varepsilon_+}$ not homotopic to $\ell$.

For each crossing $c$ whose smoothing provides attachment points $\varepsilon(c)c^\bullet$ on some $\gamma \in \mathcal{H}_\ell$, let $\mathcal{D}_c$ be a small disk neighborhood of $c$ containing $\varepsilon(c)c^\bullet$. These disks are taken to be sufficiently small such that $\mathcal{D}_c \cap \mathcal{D}_{c'} = \emptyset$ for $c \neq c'$. The augmented annular neighborhood is then given by:
\[
\mathbb{A}_\ell := \mathbb{A}'_\ell \cup \left( \bigcup_{\gamma \in \mathcal{H}_\ell} \bigcup_{c \in \gamma} \mathcal{D}_c \right).
\]
This construction ensures that all smoothing sites and their associated attachment points are localized within the neighborhood system, satisfying the required conditions.
  \end{proof}

Via an ambient isotopy, the superposition $\mathsf{x} \cdot \mathsf{Y}$ is assumed to be in a position where all crossings are contained within $\bigcup \mathcal{D}_c \subset \bigcup \mathbb{A}_\ell$, and the complement $(\mathsf{x} \cdot \mathsf{Y}) \setminus \bigcup \mathbb{A}_\ell$ contains no segments that can be homotoped into $\bigcup \mathbb{A}_\ell$ relative to their endpoints. Such a configuration is obtained by translating any crossing outside the neighborhoods along $\mathsf{x}$ into the nearest $\mathbb{A}_\ell$, and retracting all such redundant segments.
  
  \begin{lem}\label{lem:null-loop-cyclic}
  	 For any cyclic chain $\mathsf{n}$ in $\mathsf{E}_\varepsilon$ with $\varepsilon \in \mathcal{E}_{\mathrm{top}}$, let $t_{\mathsf{n}}$ denote the unique crossing in $\mathcal{C}^-_\mathsf{n}$. Changing the smoothing at $t_{\mathsf{n}}$ to be positive yields two freely homotopic simple loops $\ell, \ell' \subset \mathsf{E}_{\varepsilon_+}$ such that $\mathsf{n}$ is contained within the annulus bounded by $\ell$ and $\ell'$. Furthermore, every chain $\mathsf{n} \in \Unknots_{\mathrm{top}}$ is contained within some annular neighborhood $\mathbb{A}_\ell$.
  \end{lem} 

  \begin{proof}
       By changing the smoothing at $t_{\mathsf{n}}$ to be positive, the cyclic chain $\mathsf{n}$ resolves into two freely homotopic simple loops $\ell, \ell'$ that bound an annular neighborhood containing $\mathsf{n}$. See Figures~\ref{fig:cyclic-positive}, \ref{fig:cyclic-strip}. 

         \begin{figure}[H]
	\centering
	\begin{tikzpicture}[scale=1.4]
		\begin{scope}
			\begin{scope}[xshift=-1.2in,scale=.2]
				\draw[fill=gray!10,dashed] (-8,-4)  to (8,-4) arc (-90:90:4) to (-8,4) arc (90:270:4);
				
				\fill[pattern=north east lines, pattern color=black!40] (-8,-4) to [out=90, in=180] (-4,0) to [out=0, in=90] (0,-4);
				
				\fill [pattern=north east lines, pattern color=black!40] (0, 4) to [out=270, in=180] (4,0) to [out=0,in=270] (8,4);
				
				\draw [thick, dashed] (-12,0) to [out=0, in=270] (-8,4);
				
				\draw [thick] (-8,-4) to [out=90, in=180] (-4,0) to [out=0, in=90] node[marked, pos=.5] {} node[right=.1] {$-t_\mathsf{n}'$} (0,-4) ;
				
				\draw [thick] (0, 4) to [out=270, in=180] node[marked, pos=.5] {} node[left=.1]{$-t_\mathsf{n}''$} (4,0) to [out=0,in=270] (8,4) ;
				
				\draw [thick, dashed] (8,-4) to [out=90,in=180] (12,0) ;
			\end{scope}
			\node (=) at (0,0) {$\mapsto$};
			\begin{scope}[xshift=1.2in,scale=.2]
				\draw[fill=gray!10,dashed] (-8,-4)  to (8,-4) arc (-90:90:4) to (-8,4) arc (90:270:4);
				
				\fill [pattern=north east lines, pattern color=black!40] (-8,-4) to [out=90,in=180] (-4,0) to [out=0,in=270] (0,4) to (8, 4) to [out=270, in =0] (4,0) to [out=180, in=90] (0,-4);
				
				\draw [thick, dashed] (-12,0) to [out=0, in=270] (-8,4);
				
				\draw[thick] (-8,-4) to [out=90,in=180] node[left=.1,pos=.5] {$\ell$} (-4,0) to [out=0,in=270] node[marked, pos=.5] {} node[left=.1]{$+t_\mathsf{n}''$} (0,4);
				
				\draw[thick] (0,-4) to [out=90,in=180] node[marked, pos=.5] {} node[right=.1]{$+t_\mathsf{n}'$} (4,0) to [out=0,in=270] node[right=.1,pos=.5] {$\ell'$} (8,4);
				
				\draw [thick, dashed] (8,-4) to [out=90,in=180] (12,0) ;
			\end{scope}
		\end{scope}
	\end{tikzpicture}
	\caption{Changing the smoothing at $t_\mathsf{n}$ to positive}
	\label{fig:cyclic-positive}
  \end{figure}
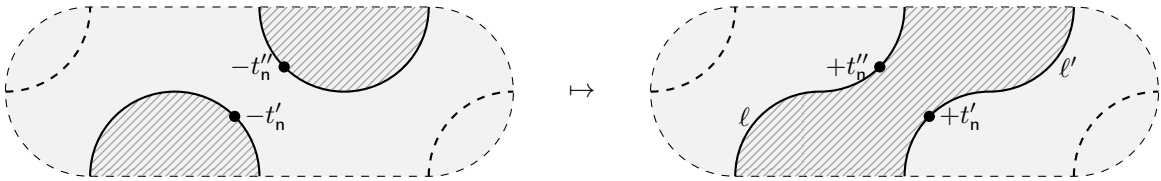

  \begin{figure}[H]
	\centering
	\begin{tikzpicture}[scale=1.4]
		\begin{scope}[scale=.2]
			\draw[fill=gray!10,dashed] (-8,-4)  to (8,-4) arc (-90:90:4) to (-8,4) arc (90:270:4);
			
			\fill [pattern=north east lines, pattern color=black!40] (-8,-4) to [out=90,in=180] (-4,0) to [out=0,in=270] (0,4) to (8, 4) to [out=270, in =0] (4,0) to [out=180, in=90] (0,-4);
			
			\draw [thick, dashed] (-12,0) to [out=0, in=270] (-8,4);
			
			\draw[thick] (-8,-4) to [out=90,in=180] node[marked, pos=.5] {} node[left=.1] {$\mathsf{+}$} (-4,0) node[above left=.1] {$\ell$} to [out=0,in=270] node[marked, pos=.5] {} node[left=.1]{$+$} (0,4);
			
			\draw[thick] (0,-4) to [out=90,in=180] node[marked, pos=.5] {} node[right=.1]{$+$} (4,0) node[below right=.1] {$\ell'$} to [out=0,in=270] node[marked, pos=.5] {} node[right=.1] {$\mathsf{+}$} (8,4);
			
			\draw [thick, dashed] (8,-4) to [out=90,in=180] (12,0) ;
		\end{scope}
	\end{tikzpicture}
	\caption{Components of $D \cap \mathcal{T}$ associated with the cyclic chain $\mathsf{n}$, where $D$ is the disk bounded by $\mathsf{n}$ and $\mathcal{T}$ is the tubular neighborhood of $\mathsf{x}$.}
	\label{fig:cyclic-strip}
  \end{figure}
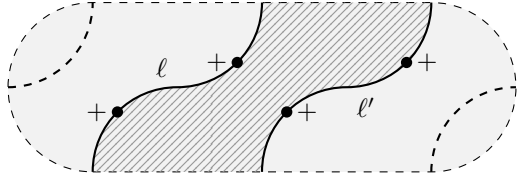

  Any chain $\mathsf{n} \in \Unknots_{\mathrm{top}}$ is obtained from some cyclic chain $\mathsf{n}_0 \in \Unknots_{\mathrm{top}}$ via a finite sequence of strip removing operations. By Lemma \ref{lem:top-chains}(2), these operations preserve containment within the annular neighborhood $\mathbb{A}_{\ell}$ associated with $\mathsf{n}_0$.
  \end{proof}

Let $\mathcal{G}(\mathbb{A}_{\ell})$ denote the set of crossings where both strands resulting from the $+$ smoothing belong to loops homotopic to $\ell$. We refer to $\mathcal{G}(\mathbb{A}_{\ell})$ as the {generating set in $\mathbb{A}_{\ell}$}, or simply the {generating set of type $\ell$}, denoted by $\mathcal{G}(\ell)$. By definition, these sets are mutually disjoint: $\mathcal{G}(\mathbb{A}_{\ell}) \cap \mathcal{G}(\mathbb{A}_{\ell'}) = \emptyset$ whenever $\ell$ and $\ell'$ are non-homotopic. 

For any $\varepsilon \in\mathcal{E}_{\mathrm{top}}$, attaching strips transforms the configuration $\Unknots_\varepsilon$ into a new configuration $\Unknots_{\teps}$ consisting entirely of cyclic chains (see Lemma~\ref{lem:top-chains}(2)). Recall that the global generating set is invariant under strip operations, yielding $\mathcal{G}_\varepsilon = \mathcal{G}_{\teps}$. As established previously, we naturally have $\varepsilon^{-1}(-) \subseteq \mathcal{G}_\varepsilon$. Moreover, Lemma~\ref{lem:null-loop-cyclic} guarantees that $\mathcal{G}_{\teps} \subseteq \bigcup_{\ell\in \mathcal{L}_{\varepsilon_+}} \mathcal{G}(\mathbb{A}_{\ell})$. We thus have the following observations:

\begin{prop}[Independent generation]\label{prop:indep-config}
For each \(\ell\in \mathcal{L}_{\varepsilon^{+}}\), choose a subset
$\mathcal{Z}_{\ell}\subseteq \mathcal{G}(\mathbb{A}_{\ell})$ and define $\varepsilon$, $\varepsilon_\ell$ by their negative loci:
\[
\varepsilon^{-1}(-)
=
\bigsqcup_{\ell\in \mathcal{L}_{\varepsilon_{+}}}
\mathcal{Z}_{\ell},\quad (\varepsilon_\ell)^{-1}(-)=\mathcal{Z}_\ell.
\]

(1) $\varepsilon \in \mathcal{E}_{\mathrm{top}}$ if and only if the corresponding multicurve
\(\mathsf{E}_{\varepsilon}\) contains exactly \(|\varepsilon^{-1}(-)|\) unknot components with each unknot possessing exactly two negative attachment points. Moreover, all \(\varepsilon\in\mathcal{E}_{\mathrm{top}}\) can be constructed in this way.

(2) We have $\varepsilon \in\mathcal{E}_{\mathrm{top}}$ if and only if $\varepsilon_\ell \in\mathcal{E}_{\mathrm{top}}$ for all $\ell \in \mathcal{L}_{\varepsilon_+}$. 
\end{prop}

\begin{proof}
    (1) 
    The first statement is implied by \eqref{eq:q-deg-top} and \eqref{eq:unknot-2-neg-pts}. To see that all $\varepsilon\in\mathcal{E}_{\mathrm{top}}$ arise this way, we use the property that any such choice can be transformed into the all-positive smoothing $\varepsilon_+$ by reverting its $-$ smoothings to $+$ smoothings one by one, with every step remaining in $\mathcal{E}_{\mathrm{top}}$ (See \cite[Lemma C.2]{muller2016skein}). Suppose for a contradiction that there exists a choice containing a $-$ smoothing at a crossing $c \notin \bigcup_\ell \mathcal{G}(\ell)$. By this step-by-step reduction, there must exist an $\varepsilon \in \mathcal{E}_{\mathrm{top}}$ where $c$ is the only crossing with a $-$ smoothing. Since $\varepsilon \in \mathcal{E}_{\mathrm{top}}$, this single $-$ smoothing at $c$ must form an unknot. Reverting this $-$ smoothing at $c$ to a $+$ smoothing transforms $\varepsilon$ into $\varepsilon_+$. Moreover, by Lemma \ref{lem:null-loop-cyclic}, the attachment points for the $+$ smoothing at $c$ lie within two isotopic loops of $E_{\varepsilon_+}$. Therefore $c \in \bigcup_\ell \mathcal{G}(\ell)$, which contradicts the assumption that $c$ is an external crossing. Thus, all $-$ smoothings must be confined to $\bigcup_\ell \mathcal{G}(\ell)$.
    
    (2) First assume $\varepsilon \in\mathcal{E}_{\mathrm{top}}$. Then each unknot $\mathsf{n}\in \Unknots_\varepsilon$ has exactly two negative attachment points and is contained in some $\mathbb{A}_{\ell}$. Note that $\mathsf{E}_{\varepsilon_\ell}$ is obtained from $\mathsf{E}_{\varepsilon}$ by attaching strips among unknots in $\bigcup_{\ell':\ \ell'\neq \ell,\ell'\in \mathcal{L}_{\varepsilon_+}}\mathbb{A}_{\ell'}$; see Lemma \ref{lem:top-chains} (2) and Lemma \ref{lem:null-loop-cyclic}. 
    Since such strip operations preserve the property of contributing to the top $q$-degree term, we have $\varepsilon_\ell\in \mathcal{E}_{\mathrm{top}}$.

Next, assume $\varepsilon_\ell \in\mathcal{E}_{\mathrm{top}}$ for all $\ell \in \mathcal{L}_{\varepsilon_+}$. By (1), there are $|\mathcal{Z}_\ell|$ unknots in $\mathsf{E}_{\varepsilon_\ell}$, each containing exactly two negative attachment points. Moreover, each unknot of $\mathsf{E}_{\varepsilon_\ell}$ is contained in $\mathbb{A}_{\ell}$ (by Lemma \ref{lem:top-chains} (2) and Lemma \ref{lem:null-loop-cyclic}). Additionally, by the construction of multicurves, $\mathsf{E}_{\varepsilon}$ and $\mathsf{E}_{\varepsilon_\ell}$ have the same restriction in $\mathbb{A}_{\ell}$. Therefore, $\mathsf{E}_{\varepsilon}$ contains all unknots of $\mathsf{E}_{\varepsilon_\ell}$ for all $\ell \in \mathcal{L}_{\varepsilon_+}$. These unknots in $\mathsf{E}_{\varepsilon}$ contribute a total of $\sum_\ell 2|\mathcal{Z}_\ell|$ negative attachment points, with each unknot providing exactly two. On the other hand, there are only $2|\varepsilon^{-1}(-)|=\sum_\ell 2|\mathcal{Z}_\ell|$ negative attachment points in $\mathsf{E}_\varepsilon$. Then (1) implies that $\varepsilon \in \mathcal{E}_{\mathrm{top}}$.
\end{proof}

  \subsection{Marked annulus}\label{sec:annulus}
  Let $\mathbb{A}_{m,m}$ denote an annulus with $m$ marked points on each boundary component, where $m\geq 1$. We first establish our main result in the local model $\mathbb{A}_{m,m}$, beginning with the base case $m=1$.

 \begin{lem} \label{prop:band-triangular}
  	  The band basis is the common triangular basis for $\mathrm{Sk}_q^{\circ}(\mathbb{A}_{1,1})$.
  \end{lem}
   
  \begin{proof}
  Let $\sd$ be a quantum seed with $B$-matrix $B=\begin{pmatrix} 
  0 & -2 \\\\[-10pt]
  2 & 0 
  \end{pmatrix}$,
  and let $\bUpClAlg_q(\sd)$ denote the corresponding quantum cluster algebra. Its triangular basis in the sense of Berenstein-Zelevinsky \cite{BerensteinZelevinsky2012} coincides with the common triangular basis defined in \cite{qin2017triangular} (cf. \cite{Qin2016compare}). Let $L_g$ denote the $g$-pointed basis element, and set $\delta = f_1 - f_2$. By \cite[Proposition 6.1]{BerensteinZelevinsky2012}, the triangular basis elements $L_{k\delta}$ for $k\in \mathbb{N}$ satisfy the Chebyshev recursion:
  \begin{equation}\label{eq:Chebyshev-coeff-free}
      L_\delta* L_{k\delta}=L_{(k+1)\delta}+L_{(k-1)\delta},\quad \text{for all } k \geq 1.
  \end{equation}

Let $x'_3,x'_4$ denote the inner and outer boundary arcs of $\mathbb{A}_{1,1}$, respectively. Let $\Delta$ be an arbitrary triangulation of $\mathbb{A}_{1,1}$. Without loss of generality, we may write $\Delta=(x'_1,x'_2,x'_3,x'_4)$ such that the $B$-matrix of the associated quantum seed $\sd'=\sd_\Delta$ is 
  $B^\Delta=\begin{pmatrix} 
  0 & -2 \\\\[-10pt]
  2 & 0 \\\\[-10pt]
  -1 & 1\\\\[-10pt]
  -1 & 1
  \end{pmatrix}$.
The compatibility matrix $\Lambda'=\Lambda^\Delta$ is given in Section \ref{subsec:clAlg-surface}. The initial $y$-variables $y'_i$ have degrees $\deg (y'_1)=2f_2-f_3-f_4$, $\deg (y'_2)=-2f_1+f_3+f_4$. 

Since $\sd'$ and $\sd$ share the same principal part of the $B$-matrix, $\bUpClAlg(\sd')$ admits a common triangular basis \cite{qin2017triangular}. Let $b'_{k\delta}$ denote the $k\delta$-pointed common triangular basis element of $\bUpClAlg(\sd')$, where $\delta=f_1-f_2$ and $k\in \mathbb{N}$, with the convention $b'_0=1$. Applying the correction technique from \cite{Qin12} (cf.\ \cite{qin2023analogs}) yields $b'_\delta = x'_1 (x'_2)^{-1} (1 + y'_2 + y'_1 y'_2)$. Furthermore, these basis elements satisfy a relation analogous to \eqref{eq:Chebyshev-coeff-free}:
 \begin{align}
	q^{-\frac{1}{2}\Lambda'(\delta,k\delta )}b'_\delta * b'_{k\delta}=b'_{(k+1)\delta}+q^{\beta'_{k-1}} b'_{(k-1)\delta} * p_{k-1},\quad \text{for some } \beta'_{k-1} \in \Hf\mathbb{Z}, \text{ and all } k\geq 1,
  \end{align}
where $p_{k-1}$ are monomials in $(x'_3)^{\pm 1}, (x'_4)^{\pm 1}$ such that $\deg (b'_{(k-1)\delta}*p_{k-1})=\deg (b'_{(k+1)\delta})+\deg (y'_1) + \deg (y'_2)$. We deduce that $p_{k-1}=1$. Direct computation shows that $\Lambda'(\delta,k\delta )=0$ and $\beta'_{k-1}=0$. Consequently, we obtain the Chebyshev recursion:
 \begin{align}
	b'_\delta*b'_{k\delta}=b'_{(k+1)\delta}+b'_{(k-1)\delta},\ \forall k\geq 1.
  \end{align}

 Finally, $\bUpClAlg(\sd')$ and $\mathrm{Sk}_q^{\circ}(\mathbb{A}_{1,1})$ are identified via their natural isomorphism. Let $\ell$ denote the simple loop in $\mathbb{A}_{1,1}$, with associated band elements $\mathrm{Band}^k(\ell)$ for $k \in \mathbb{N}$. Since all other elements of the band basis are cluster monomials (and are therefore contained in the common triangular basis), it suffices to prove $\mathrm{Band}^k(\ell) = b'_{k\delta}$ for all $k \in \mathbb{N}$. This equality holds for $k=0, 1$ by direct computation, and the general case follows inductively from the Chebyshev recursion.
  \end{proof}
  
  \begin{prop} \label{prop:band-triangular-mm}
  	  The band basis is the common triangular basis for $\mathrm{Sk}_q^{\circ}(\mathbb{A}_{m,m})$.
  \end{prop}

  \begin{proof}
      Consider the annulus $\mathbb{A}_{m,m}$ for $m \geq 1$. There exists a triangulation $\Delta$ of $\mathbb{A}_{m,m}$ such that the associated quantum seed $\sd_\Delta = (B^\Delta, \Lambda^\Delta, M^\Delta)$ is acyclic and of affine type. Specifically, such a triangulation $\Delta$ is given explicitly as follows:

      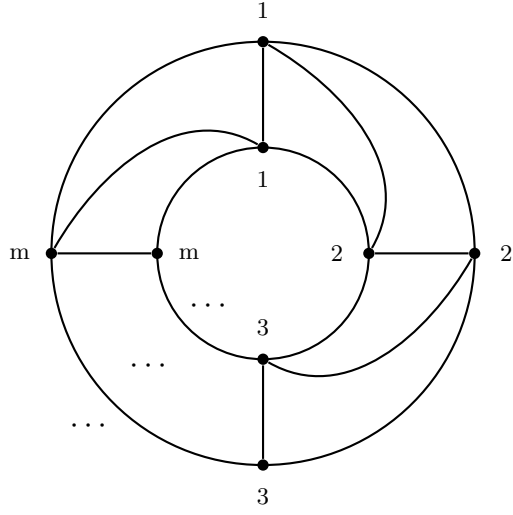
\begin{figure}[H]
          \centering
          \begin{tikzpicture}[scale=1.4]
     	\def\innerR{1}    
     	\def\outerR{2}    
     	
     	\draw[thick] (0,0) circle (\innerR);
     	\draw[thick] (0,0) circle (\outerR);
     	
     	\foreach \angle/\label in {0/2, 270/3, 180/m, 90/1} {
     		\node[circle, fill=black, inner sep=1.5pt] (inner-\label) at (\angle:\innerR) {};
     		\node[font=\small, black] at (\angle:\innerR-0.3) {\label};
     	}
     	
     	\foreach \angle/\label in {0/2, 270/3, 180/m, 90/1} {
     		\node[circle, fill=black, inner sep=1.5pt] (outer-\label) at (\angle:\outerR) {};
     		\node[font=\small, black] at (\angle:\outerR+0.3) {\label};
     	}
     	
     	\node[font=\Large, black] at (225:{(\innerR+\outerR)/2}) {$\cdots$};
     	
     	\node[font=\Large, black] at (225:\innerR-0.3) {$\cdots$};
     	
     	\node[font=\Large, black] at (225:\outerR+0.3) {$\cdots$};
     	
        \draw[thick, black] (inner-1) -- (outer-1);
        \draw[thick, black] (inner-2) -- (outer-2);
        \draw[thick, black] (inner-m) -- (outer-m);
        \draw[thick, black] (inner-3) -- (outer-3);
     	
     	\draw[thick, black] 
     	(outer-1) to [out=-30, in=60] (inner-2);
     	
     	\draw[thick, black] 
     	(outer-2) to [out=240, in=-30] (inner-3);
     	
     	\draw[thick, black] 
     	(outer-m) to [out=60, in=150] (inner-1);
     \end{tikzpicture}
          \caption{The triangulation $\Delta$ of the annulus $\mathbb{A}_{m,m}$.}
          \label{fig:acyclic-tri}
      \end{figure}

      The quiver $Q(\Delta)$ associated with $\Delta$ is depicted in Figure \ref{fig:quiver-tri}, where each circular vertex labeled $ij$ corresponds to the interior arc connecting marked points $i$ and $j$ in Figure \ref{fig:acyclic-tri}.
      
     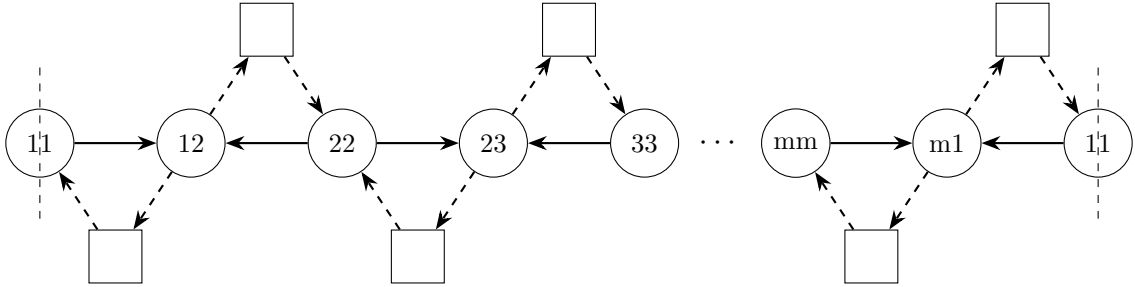
\begin{figure}[H]
         \centering
          \begin{tikzpicture}[scale=1, 
	circle vertex/.style={circle, draw, minimum size=9mm},
	square vertex/.style={rectangle, draw, minimum size=7mm},
	arrow/.style={-{Stealth[length=2.5mm]}, thick}
	]

	\node[circle vertex] (v1) at (0,0) {11};
	\node[square vertex] (v2) at (1,-1.5) {}; 
	\node[circle vertex] (v3) at (2,0) {12};     
	\node[square vertex] (v4) at (3,1.5) {};
	\node[circle vertex] (v5) at (4, 0) {22};
	
	\node[square vertex] (v6) at (5,-1.5) {}; 
	\node[circle vertex] (v7) at (6,0) {23};     
	\node[square vertex] (v8) at (7,1.5) {};
	\node[circle vertex] (v9) at (8, 0) {33};

	\draw[arrow] (v1) -- (v3); 
	\draw[arrow, dashed] (v2) -- (v1); 
	\draw[arrow, dashed] (v3) -- (v2); 
	\draw [arrow, dashed] (v3) -- (v4);
	\draw [arrow, dashed] (v4) -- (v5);
	\draw [arrow] (v5) -- (v3);
	
	\draw[arrow] (v5) -- (v7); 
	\draw[arrow, dashed] (v6) -- (v5); 
	\draw[arrow, dashed] (v7) -- (v6); 
	\draw [arrow, dashed] (v7) -- (v8);
	\draw [arrow, dashed] (v8) -- (v9);
	\draw [arrow] (v9) -- (v7);
	
	\node[font=\Large, black] at (9,0) {$\cdots$};
	\node[circle vertex] (v10) at (10,0) {mm};
	\node[square vertex] (v12) at (11,-1.5) {}; 
	\node[circle vertex] (v13) at (12,0) {m1};     
	\node[square vertex] (v14) at (13,1.5) {};
	\node[circle vertex] (v15) at (14, 0) {11};
	
	\draw[arrow] (v10) -- (v13); 
	\draw[arrow, dashed] (v12) -- (v10); 
	\draw[arrow, dashed] (v13) -- (v12); 
	\draw [arrow, dashed] (v13) -- (v14);
	\draw [arrow, dashed] (v14) -- (v15);
	\draw [arrow] (v15) -- (v13);
	
	\draw [dashed] (0,1) -- (0, -1);
	\draw [dashed] (14,1) -- (14, -1);
     \end{tikzpicture}
         \caption{The acyclic quiver $Q(\Delta)$ associated with $\Delta$.}
         \label{fig:quiver-tri}
    \end{figure}

    Recall that an acyclic quantum cluster algebra admits a common triangular basis, established initially for certain coefficients by \cite{KimuraQin14} and extended to arbitrary full-rank coefficients in \cite{qin2017triangular}. Proposition \ref{prop:inclusion} (\cite[Theorem 7.15, Theorem 9.8]{muller2016skein}) yields
    \[
	               \bClAlg_q (\mathbb{A}_{m,m}) = \mathrm{Sk}_q^\circ (\mathbb{A}_{m,m}) = \bUpClAlg_q (\mathbb{A}_{m,m}).
	           \]
    Consequently, $\mathrm{Sk}_q^\circ (\mathbb{A}_{m,m})$ admits a common triangular basis $\mathbf{L}$ with non-negative structure constants. 

It remains to show that all band elements belong to $\mathbf{L}$. Since every band element is either a cluster monomial (and thus already in $\mathbf{L}$) or of the form $\mathrm{Band}^k(\ell)$ for a simple loop $\ell$ and $k \in \mathbb{N}$, it suffices to prove that $\mathrm{Band}^k(\ell) \in \mathbf{L}$ for all $k \in \mathbb{N}$.

   Let $\Sigma_{\ell} \simeq \mathbb{A}_{1,1}$ be an annular subsurface bounded by arcs $\mathsf{x}$ and $\mathsf{y}$:
    \begin{figure}[H]
        \centering
        \begin{tikzpicture}[scale=1.2]
		\def\innerR{1}    
		\def\outerR{2}    
		
		\draw[thick] (0,0) circle (\innerR);
		\draw[thick] (0,0) circle (\outerR);
            \draw[thick, green] (0,0) circle (1.5);

            \node[red] at (60: \outerR-.3) {$\mathsf{x}$};
            \node[green] at (90: 1.35) {$\ell$};
            \node[blue] at (120: \innerR+.1) {$\mathsf{y}$};

		\foreach \angle/\label in {0/2, 270/3, 180/m, 90/1} {
			\node[circle, fill=black, inner sep=1.5pt] (inner-\label) at (\angle:\innerR) {};
			\node[font=\small, black] at (\angle:\innerR-0.3) {\label};
		}
		
		\foreach \angle/\label in {0/2, 270/3, 180/m, 90/1} {
			\node[circle, fill=black, inner sep=1.5pt] (outer-\label) at (\angle:\outerR) {};
			\node[font=\small, black] at (\angle:\outerR+0.3) {\label};
		}
		
		\node[font=\Large, black] at (225:\innerR-0.3) {$\cdots$};
		
		\node[font=\Large, black] at (225:\outerR+0.3) {$\cdots$};
		
		\draw[thick,red, >=Stealth]
		(90:2) to [out=-90, in=150] 
		(.2, 1.7) to [out= -30, in=90]
		 (1.7,0) to [out=-90, in=0] 
		 (0,-1.7) to [out=180, in=-90] 
		 (-1.7,0) to [out=90, in=-150]
		  (-.2, 1.7) to [out=30, in=-90] (90:2);
		  
		 \draw[thick,blue, >=Stealth]
		 (90:1) to [out=90, in=160] 
		 (.3, 1.23) to [out=-20, in=90] 
		 (1.35,0) to [out=-90, in=0] 
		 (0,-1.35) to [out=180, in=-90]
		 (-1.35,0) to [out=90, in=-160] 
		 (-.3,1.23) to [out=20, in=90] (90:1);

	\end{tikzpicture}
        \caption{The annular subsurface $\Sigma_{\ell} \simeq \mathbb{A}_{1,1}$.}
        \label{fig: subsurface}
    \end{figure}
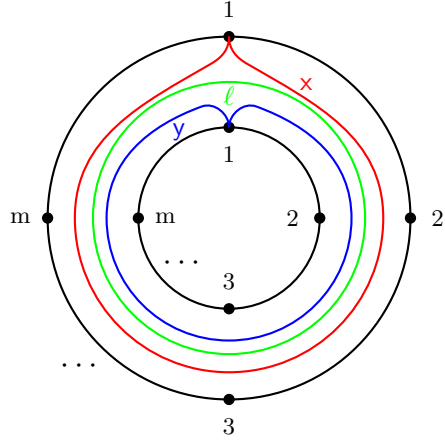

    Extend the arcs $\mathsf{x}, \mathsf{y}$ to form a triangulation $\Delta'$. Cutting along $\mathsf{x}$ and $\mathsf{y}$ yields the decomposition 
    \[
    \mathbb{A}_{m,m}=\Sigma_{\ell} \cup \Sigma_1 \cup \Sigma_2.
    \]

    \begin{figure}[H]
        \centering
        \begin{tikzpicture}[scale=1.2]
        \begin{scope}
            \def\outerR{2}    
	    \draw[thick] (0,0) circle (\outerR);
	    \node[font=\Large, black] at (225:\outerR+0.3) {$\cdots$};
	    \foreach \angle/\label in {0/2, 270/3, 180/m, 90/1} {
	    	\node[circle, fill=black, inner sep=1.5pt] (outer-\label) at (\angle:\outerR) {};
	    	\node[font=\small, black] at (\angle:\outerR+0.3) {\label};
	    }

            \node[font=\Large, red] at (60: \outerR-.5) {$\mathsf{x}$};
	    \draw[thick, red] (0,.8) circle (1.2);
	    
	    \draw[thick, >=Stealth] 
	    (90:2) to [out=0, in=90] 
	    (1.7,0) to [out=-90, in=25] (270:2) ;
	    
	     \draw[thick, >=Stealth] 
	    (90:2) to [out=0, in=75] 
	    (1.4,0) to [out=-105, in=-70] (180:2) ;
	    
	    \node[font=\Large, black] at (260:{3/2}) {$\cdots$};
        \end{scope}

        \begin{scope} [xshift=6cm, yshift=1cm]
	\def\innerR{1}    

	\draw[thick] (0,0) circle (\innerR);

	\foreach \angle/\label in {0/2, 270/3, 180/m, 90/1} {
		\node[circle, fill=black, inner sep=1.5pt] (inner-\label) at (\angle:\innerR) {};
		\node[font=\small, black] at (\angle:\innerR-0.3) {\label};
	}

        \node[font=\Large, blue] at (150: \innerR+.5) {$\mathsf{y}$};
	\draw[thick, blue] (0,-1) circle (2);
	
	\node[font=\Large, black] at (225:\innerR-0.3) {$\cdots$};
	
	\draw[thick, >=Stealth] 
	(90:1) to [out=0, in=75] 
	(1.3,-.5) to [out=-105, in=-30] (270:1) ;
	
	\draw[thick, >=Stealth] 
	(90:1) to [out=0, in=90] 
	(1.7,-.8) to [out=-90, in=0]
	 (0.2,-2.3) to [out=180, in=-120] (180:1);
	 
	  \node[font=\Large, black] at (270:{1.7}) {$\cdots$};
        \end{scope}

    \end{tikzpicture}
        \caption{The subsurfaces $\Sigma_1$ and $\Sigma_2$ with their induced triangulations.}
        \label{fig:subsur}
    \end{figure}
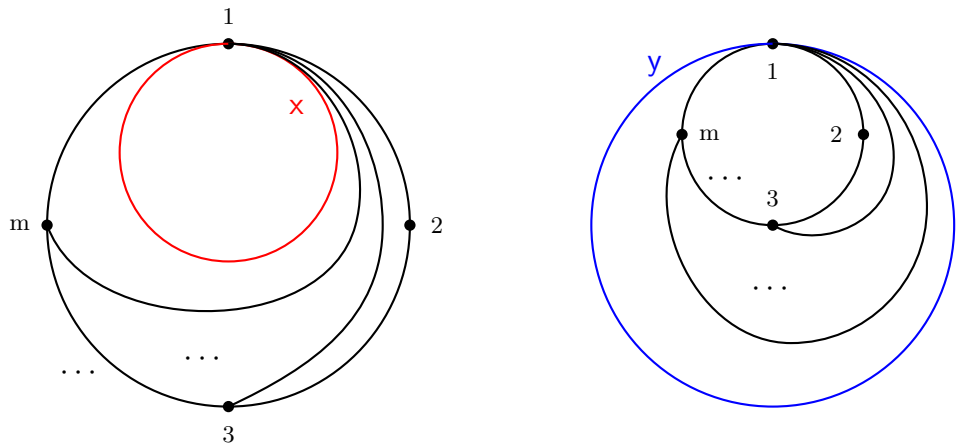

    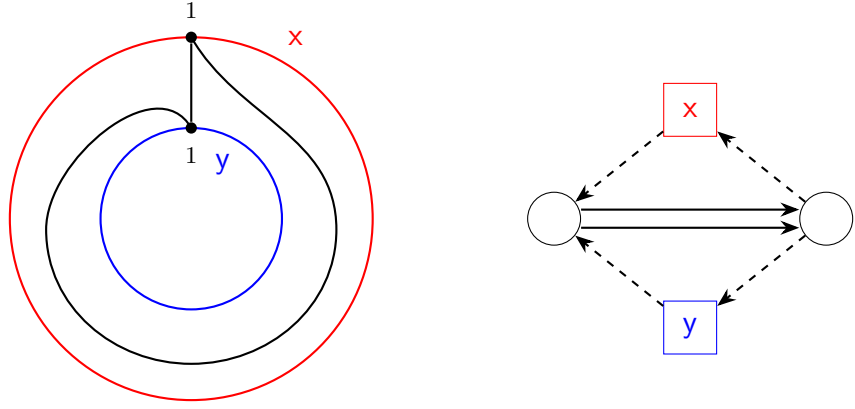
\begin{figure}[H]
        \centering
        \begin{tikzpicture}[scale=1.2]
        \begin{scope}[xshift=-2cm]
	\def\innerR{1}    
	\def\outerR{2}    
	
	\draw[thick, blue] (0,0) circle (\innerR);
	\draw[thick, red] (0,0) circle (\outerR);
        \node[font=\Large, red] at (60: \outerR+.3) {$\mathsf{x}$};
	\node[font=\Large, blue] at (60: \innerR-.3) {$\mathsf{y}$};
    
	\foreach \angle/\label in {90/1} {
		\node[circle, fill=black, inner sep=1.5pt] (inner-\label) at (\angle:\innerR) {};
		\node[font=\small, black] at (\angle:\innerR-0.3) {\label};
	}
	
	\foreach \angle/\label in {90/1} {
		\node[circle, fill=black, inner sep=1.5pt] (outer-\label) at (\angle:\outerR) {};
		\node[font=\small, black] at (\angle:\outerR+0.3) {\label};
	}
	
	\draw[thick, black] (inner-1) -- (outer-1);
	
	\draw[thick, >=Stealth] 
	(90:2) to [out=-60, in=90] 
	(1.6,-0.12) to [out=-90, in=0]
	(0,-1.6) to [out=180, in=-90]
	 (-1.6,-0.12) to [out=90, in=120] (90:1);
        \end{scope}

       \begin{scope}[xshift=2cm,
	circle vertex/.style={circle, draw, minimum size=7mm},
	square vertex/.style={rectangle, draw, minimum size=7mm},
	arrow/.style={-{Stealth[length=2.5mm]}, thick}
	]
	\node[circle vertex] (v1) at (0,0) {};
	\node[square vertex, blue, font=\Large] (v2) at (1.5,-1.2) {$\mathsf{y}$}; 
	\node[circle vertex] (v3) at (3,0) {};     
	\node[square vertex, red, font=\Large] (v4) at (1.5,1.2) {$\mathsf{x}$};

	 \draw[arrow] ([yshift=1mm]v1.east) -- ([yshift=1mm]v3.west);
	  \draw[arrow] ([yshift=-1mm]v1.east) -- ([yshift=-1mm]v3.west); 
	\draw[arrow, dashed] (v2) -- (v1); 
	\draw[arrow, dashed] (v3) -- (v2); 
	\draw [arrow, dashed] (v3) -- (v4);
	\draw [arrow, dashed] (v4) -- (v1);
	       \end{scope}
\end{tikzpicture}
        \caption{The triangulation $\Delta_{\ell}$ and its associated quiver $Q(\Delta_{\ell})$.}
        \label{fig:subtri}
    \end{figure}

  Let $g_\ell^{\Delta'} := \deg([\ell])$, and consider the $kg_\ell^{\Delta'}$-pointed element $L_{kg_\ell^{\Delta'}} \in \mathrm{Sk}^\circ_q (\mathbb{A}_{m,m})$ of the common triangular basis. By definition, every element of the common triangular basis is compatibly pointed at all seeds of $\mathrm{Sk}^\circ_q (\mathbb{A}_{m,m})$. Furthermore, the $kg_\ell^{\Delta'}$-pointed bangle element $[\ell]^k$ is also compatibly pointed at all seeds (cf.\ \cite[Theorem 4.3]{reading2014universal}).

  By Proposition \ref{prop:compati} (\cite[Proposition 3.4.8]{qin2019bases}), the support dimensions of compatibly pointed elements are determined entirely by their degrees. This yields
  \[
  \mathrm{supp}(L_{kg_\ell^{\Delta'}})=\mathrm{supp}([\ell]^k) \subset \mathbf{ex}(\Delta_\ell).
  \]
Restriction to the subsurface $\Sigma_{\ell}$ via the freezing operator from \cite{qin2023analogs} acts trivially on $L_{kg_\ell^{\Delta'}}$. It then follows from \cite[Theorem 3.17]{qin2023analogs} that $L_{kg_\ell^{\Delta'}}$ is also a triangular basis element of $\mathrm{Sk}_q^\circ (\Sigma_\ell)$. On the other hand, Lemma \ref{prop:band-triangular} dictates that $\mathrm{Band}^k(\ell)$ is the corresponding common triangular basis element of $\mathrm{Sk}_q^\circ (\Sigma_\ell)$ with the same degree. Consequently,
$$
L_{kg_\ell^{\Delta'}} = \mathrm{Band}^k(\ell).
$$

 Thus, $\mathrm{Band}^k(\ell)$ is an element of the common triangular basis $\mathbf{L}$ for $\mathrm{Sk}^\circ_q (\mathbb{A}_{m,m})$. This completes the proof that all band elements belong to $\mathbf{L}$.
  \end{proof}
  
  \begin{prop}[Leading degree properties]\label{prop:leading-g}
  	Let $\mathbb{A}_{m,m}$ be an annulus with $m \geq 1$ marked points per boundary component. Consider multicurves $\mathsf{X}, \mathsf{Y}$ such that $[\mathsf{X}]$ is a nonzero cluster monomial, $[\mathsf{Y}] \neq 0$, and their superposition $\mathsf{X} \cdot \mathsf{Y}$ is transverse with minimal crossings. Then:
  	\begin{enumerate}
  		\item $[\mathsf{E}_{\varepsilon_+}]\neq 0$. 
  		\item $[\mathsf{E}_{\varepsilon_-}]\neq 0$.
  		\item Fix a seed $\mathbf{s}$ such that $[\mathsf{X}]$ is an initial monomial. Then, with respect to $\mathbf{s}$,
  		$$\deg ([\mathsf{E}_{\varepsilon_+}])=\deg ([\mathsf{X}])+\deg ([\mathsf{Y}]).$$ 
  		Furthermore, $$\Lambda(\deg ([\mathsf{X}]),\deg ([\mathsf{Y}]))=2\mu(\mathsf{X}, \mathsf{Y})+ 2a,$$
  		where $\mu(\mathsf{X}, \mathsf{Y})$ is the minimal intersection number and $a$ is the boundary intersection number (with $a=0$ if $\mathsf{X}$ and $\mathsf{Y}$ do not intersect at marked points).
  		\item $\mathrm{codeg}([\mathsf{E}_{\varepsilon_-}])=\mathrm{codeg} ([\mathsf{X}])+\mathrm{codeg} ([\mathsf{Y}]).$
  	\end{enumerate}
  \end{prop}
  
  \begin{proof}
  	  (1) By Lemma \ref{lem:Muller-C1}, $[\mathsf{E}_{\varepsilon_+}] \neq 0$ since there are no unknots or contractible arcs in this case.
  	
  	  (3) Let $\ell$ denote the unique simple loop in $\mathbb{A}_{m,m}$ up to homotopy. Since $\mathsf{X}$ is a multicurve consisting of pairwise non-intersecting arcs, we uniquely decompose it as a disjoint union $\mathsf{X} = \mathsf{X}_0 \cup \mathsf{X}_1$, where $\mathsf{X}_0$ consists of peripheral arcs disjoint from $\ell$, and $\mathsf{X}_1$ consists of bridging arcs intersecting $\ell$. Consequently, the intersection number satisfies $\mu(\mathsf{X}, \ell) = \mu(\mathsf{X}_1, \ell) \geq 0$.

 By Proposition \ref{prop:band-triangular-mm}, the product $[\mathsf{X}]\langle\mathsf{Y}\rangle$ admits a $(\prec, \mathbf{m})$-unitriangular decomposition into band elements:
    \begin{equation}\label{eq:decomp1}
  	  	[\mathsf{X}]\langle\mathsf{Y}\rangle = q^{\alpha_1} [\mathsf{F}_1] U_{w_1}([\ell]) + \sum_{j>1} q^{\alpha_j} b_j [\mathsf{F}_j] U_{w_j}([\ell]),
  	  \end{equation}
where $\alpha_1 = \frac{1}{2} \Lambda(\deg [\mathsf{X}],\deg [\mathsf{Y}])$, and $\alpha_j < \alpha_1$ with $b_j \in q^{-1}\mathbb{Z}[q^{-1}]$ for all $j>1$, and the cluster monomials $[\mathsf{F}_j]$ (including $[\mathsf{F}_1]$) commute with their respective Chebyshev polynomials of the second kind $U_{w_j}([\ell])$.
    
On the other hand, we expand the band element directly as $\langle \mathsf{Y} \rangle = [\mathsf{Y}] + \sum_k c_k [\mathsf{Y}_k]$, where $\mathsf{Y} = \Arcs \cup w\ell$, $c_k \in \mathbb{Z}$, and each lower term $\mathsf{Y}_k$ contains a reduced weight $w^{(k)} < w$ of the loop $\ell$. The intersection number for these terms is $\mu(\mathsf{X}, \mathsf{Y}_k) = \mu(\mathsf{X}, \Arcs) + w^{(k)}\mu(\mathsf{X}, \ell)$. We determine the top $q$-degree term of $[\mathsf{X}]\langle \mathsf{Y} \rangle$ by analyzing the skein expansion of each product $[\mathsf{X}][\mathsf{Y}_k]$.

By \eqref{eq:q-deg-inequality} and \eqref{eq:q-deg-top}, the expansion of the product $[\mathsf{X}][\mathsf{Y}]$ with respect to the multicurve basis of $\mathrm{Sk}_q^{\circ}(\mathbb{A}_{m,m})$ is:
      \begin{equation}\label{eq:decompose-XY}
      	[\mathsf{X}][\mathsf{Y}] = q^{{a}+\beta_+} \left( [\mathsf{E}_{\varepsilon_+}] + \sum_{\varepsilon \neq \varepsilon_+,\ \varepsilon\in \mathcal{E}_{\mathrm{top}}}  (-1 - q^{-4})^{|\Unknots_\varepsilon|} [\widehat{\mathsf{E}}_\varepsilon] \right) + \sum_{\varepsilon' \notin \mathcal{E}_{\mathrm{top}}} q^{{a}+\beta_{\varepsilon'}} (-1 - q^{-4})^{|\Unknots_{\varepsilon'}|} \delta_{\varepsilon'} [\widehat{\mathsf{E}}_{\varepsilon'}],
      \end{equation}
      satisfying:
     \begin{itemize}
         \item $\beta_{\varepsilon'}<\beta_{+}=\mu(\mathsf{X},\mathsf{Y})$.
         \item $|\Unknots_\varepsilon| \geq 1$ for $\varepsilon\in \mathcal{E}_{\mathrm{top}}$ when $\varepsilon\neq \varepsilon_+$.
         \item $[\widehat{\mathsf{E}}_\varepsilon], [\widehat{\mathsf{E}}_{\varepsilon'}]$ are multicurves without unknots or contractible arcs.
     \end{itemize}

 \textbf{Case 1: $\mu(\mathsf{X}, \ell) > 0$.}
    The presence of bridging arcs guarantees that the intersection number strictly decreases for the lower Chebyshev terms: $\mu(\mathsf{X}, \mathsf{Y}_k) < \mu(\mathsf{X}, \mathsf{Y})$. Applying \eqref{eq:decompose-XY} to each product $[\mathsf{X}][\mathsf{Y}_k]$, we obtain:
     \begin{equation}\label{eq:decomp2}
      	[\mathsf{X}]\langle \mathsf{Y}\rangle = q^{a+\beta_+} \left( [\mathsf{E}_{\varepsilon_+}] + \sum_{\varepsilon \neq \varepsilon_+,\ \varepsilon\in \mathcal{E}_{\mathrm{top}} } (-1 - q^{-4})^{|\Unknots_\varepsilon|} [\widehat{\mathsf{E}}_\varepsilon] \right) + \sum_{\varepsilon' \notin \mathcal{E}_{\mathrm{top}}} q^{a+\beta_{\varepsilon'}} (-1 - q^{-4})^{|\Unknots_{\varepsilon'}|} \delta_{\varepsilon'} [\widehat{\mathsf{E}}_{\varepsilon'}]+\sum_{h} b_h[\mathsf{E}_h],
      \end{equation}
     where $b_h\in q^{a+\beta_+-1}\mathbb{Z}[q^{-1}]$ and $\mathsf{E}_h$ are the basis elements (multicurves) resulting from the smoothings of $[\mathsf{X}][\mathsf{Y}_k]$.

Note that the right hand side of both expansions is a $\mathbb{Z}_q$-linear combination of the shifted multicurve basis. By comparing the terms of the highest $q$-degree in \eqref{eq:decomp1} and \eqref{eq:decomp2}, the uniqueness of the decomposition implies:
  	\begin{equation}\label{eq:top-q-deg-sum}
q^{\alpha_1} [\mathsf{F}_1] U_{w_1}([\ell])=q^{a+\beta_+} \left( [\mathsf{E}_{\varepsilon_+}] + \sum_{\varepsilon \neq \varepsilon_+,\ \varepsilon \in \mathcal{E}_{\mathrm{top}}} (-1)^{|\Unknots_\varepsilon|} [\widehat{\mathsf{E}}_\varepsilon] \right).
  	\end{equation}
    
Recall that $[\mathsf{E}_{\varepsilon_+}]$ and $[\widehat{\mathsf{E}}_\varepsilon]$ are pointed elements with distinct degrees. Furthermore, the triangular basis element $[\mathsf{F}_1] U_{w_1}([\ell])$ is $(\deg [\mathsf{X}] + \deg [\mathsf{Y}])$-pointed. By comparing the leading degrees on both sides, we deduce that:
  	\begin{itemize}
  		\item $\deg ([\mathsf{E}_{\varepsilon_+}])=\deg ([\mathsf{X}])+\deg ([\mathsf{Y}]).$
  		\item $\alpha_1 = a + \beta_+ = \frac{1}{2}\Lambda(\deg [\mathsf{X}], \deg [\mathsf{Y}])$.
        	\item $[F_1]([\ell])^{w_1}=[\mathsf{E}_{\varepsilon_+}]$.
  		\item For all $\varepsilon \in \mathcal{E}_{\mathrm{top}}$, the multicurves $\mathsf{E}_{\varepsilon_+}$ and $\widehat{\mathsf{E}}_\varepsilon$ share the same essential arc component $[\mathsf{F}_1]$.
  	\end{itemize}
    
    \textbf{Case 2: $\mu(\mathsf{X}, \ell) = 0$.}
    In this case, $\mathsf{X}$ consists entirely of peripheral arcs and is topologically disjoint from $\ell$. The intersection number remains constant across all terms: $\mu(\mathsf{X}, \mathsf{Y}_k) = \mu(\mathsf{X}, \Arcs)$. Because $\mathsf{X}$ and $\ell$ do not intersect, the skein operations on $\mathsf{X} \cdot \mathsf{Y}_k$ occur exclusively between $\mathsf{X}$ and $\Arcs$. The all-positive smoothing of $[\mathsf{X}][\mathsf{Y}_k]$ produces the multicurve $\mathsf{E}_+ \cup w^{(k)}\ell$, where $\mathsf{E}_+$ is the all-positive smoothing of $\mathsf{X} \cdot \Arcs$. The top $q$-degree terms from all lower expansions survive and share the exact same $q$-exponent $a + \mu(\mathsf{X}, \Arcs)$. Collecting these surviving leading terms recombines the Chebyshev polynomial:
    $$ q^{a + \mu(\mathsf{X}, \Arcs)} [\mathsf{E}_+] \left( [\ell]^w + \sum_k c_k [\ell]^{w^{(k)}} \right) = q^{a + \mu(\mathsf{X}, \Arcs)} [\mathsf{E}_+] U_w([\ell]). $$
    Equating this recombined top $q$-degree term with the leading term of the triangular decomposition \eqref{eq:decomp1} gives $[\mathsf{F}_1] = [\mathsf{E}_+]$. By comparing the leading degrees on both sides, we obtain $\deg([\mathsf{E}_{\varepsilon_+}]) = \deg([\mathsf{E}_+ \cup w\ell]) = \deg([\mathsf{X}]) + \deg([\mathsf{Y}])$. Additionally, we have $\alpha_1 = a + \mu(\mathsf{X}, \Arcs) = a + \mu(\mathsf{X}, \mathsf{Y})$.

    In both cases, we establish $\deg([\mathsf{E}_{\varepsilon_+}]) = \deg([\mathsf{X}]) + \deg([\mathsf{Y}])$. Using the formula $\alpha_1 = \frac{1}{2}\Lambda(\deg[\mathsf{X}], \deg[\mathsf{Y}])$ from the triangular decomposition, we conclude that $\Lambda(\deg[\mathsf{X}], \deg[\mathsf{Y}]) = 2\mu(\mathsf{X}, \mathsf{Y}) + 2a$.

The remaining results follow from applying the bar involution $\overline{(\cdot)}$ to the properties established in (1) and (3). 
Recall that the bar involution $\overline{(\cdot)}$ is a ring anti-automorphism of $\mathrm{Sk}_q^{\circ}(\mathbb{A}_{m,m})$ defined by $\overline{q^{1/2}} = q^{-1/2}$ and $\overline{[\mathrm{X}]} = [\overline{\mathrm{X}}]$.
Note that $[\mathsf{X}]$, $[\mathsf{Y}]$, and the band element $\langle\mathsf{Y}\rangle$ are all bar-invariant. 

(2) Applying the bar involution $\overline{(\cdot)}$ to the product $[\mathsf{X}][\mathsf{Y}]$, we obtain:
\[
\overline{[\mathsf{X}][\mathsf{Y}]} = \overline{[\mathsf{Y}]} * \overline{[\mathsf{X}]} = [\mathsf{Y}][\mathsf{X}].
\]
Topologically, the bar involution reverses the crossing data (over/under). Consequently, the all-negative term $[\mathsf{E}_{\varepsilon_-}]$ of the product $[\mathsf{X}][\mathsf{Y}]$ corresponds to the all-positive term $[\mathsf{E}'_{\varepsilon_+}]$ of the reversed product $[\mathsf{Y}][\mathsf{X}]$. 
By Lemma \ref{lem:Muller-C1}, since $[\mathsf{E}'_{\varepsilon_+}] \neq 0$, it follows that $[\mathsf{E}_{\varepsilon_-}] \neq 0$.

(4) Recall that the common triangular basis satisfies triangularity not only with respect to degrees, but also with respect to codegrees; see \cite[Definition 4.2, Theorem 4.10]{qin2020analog}. Therefore, the proof of the corresponding statement for codegrees is formally identical to the proof in (2): one replaces degrees by codegrees throughout and uses the codegree triangularity of the common triangular basis in place of its degree triangularity.  

  \end{proof}

\begin{example}
    Consider the marked surface $\mathbb{A}_{1,1}$ illustrated in Figure~\ref{fig: surface_A11}, with arcs $\mathsf{x}$ (red) and $\mathsf{y}$ (blue) and the superposition $\mathsf{x} \cdot \mathsf{y}$.
    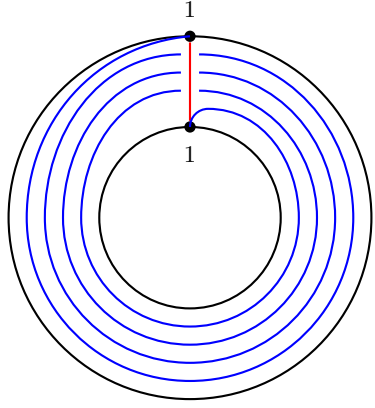
\begin{figure}[H]
        \centering
        \begin{tikzpicture}[scale=1.2]
		  \def\innerR{1}    
		  \def\outerR{2}    
		  \draw[thick, black] (0,0) circle (\innerR);
		  \draw[thick, black] (0,0) circle (\outerR);
		
		  \foreach \angle/\label in {90/1} {
			\node[circle, fill=black, inner sep=1.5pt] (inner-\label) at (\angle:\innerR) {};
			\node[font=\small, black] at (\angle:\innerR-0.3) {\label};
		  }
		
		  \foreach \angle/\label in {90/1} {
			\node[circle, fill=black, inner sep=1.5pt] (outer-\label) at (\angle:\outerR) {};
			\node[font=\small, black] at (\angle:\outerR+0.3) {\label};
		  }
		
		  \draw[thick, red] (inner-1) -- (outer-1);
		
		  \draw[thick, blue, >=Stealth] 
		(90:2) to [out=185, in=90] 
		(-1.8,0) to [out=-90, in=180]
		(0,-1.8) to [out=0, in=-90]
		(1.8,0) to [out=90,in=0] (.1, 1.8);
		
		  \draw[thick, blue, >=Stealth]
		(-.1, 1.8) to [out=180, in=90]
		(-1.6,0) to [out=-90, in=180]
		(0,-1.6) to [out=0, in=-90]
		(1.6,0) to [out=90,in=0] (.1, 1.6);
		
		  \draw[thick, blue, >=Stealth]
		(-.1, 1.6) to [out=180, in=90]
		(-1.4,0) to [out=-90, in=180]
		(0,-1.4) to [out=0, in=-90]
		(1.4,0) to [out=90,in=0] (.1, 1.4);
		
		  \draw[thick, blue, >=Stealth]
		(-.1, 1.4) to [out=180, in=90]
		(-1.2,0) to [out=-90, in=180]
		(0,-1.2) to [out=0, in=-90]
		(1.2,0) to [out=90,in=0] 
		(.2, 1.2) to [out=180, in=90] (90:1);
	    \end{tikzpicture}
        \caption{The marked surface $\mathbb{A}_{1,1}$ with arcs $\mathsf{x}$ and $\mathsf{y}$.}
        \label{fig: surface_A11}
    \end{figure}

   Applying the the choice of smoothings $\varepsilon = (+, +, +)$ along the arc $\mathsf{x}$ (from top to bottom) at the crossings in $\mathsf{x} \cap \mathsf{y}$ yields the configuration in Figure~\ref{fig: +++}. Let $\mathsf{b}_1$ and $\mathsf{b}_2$ denote the red and blue boundary-parallel arcs respectively, and $\ell$ the green homotopic loops, as depicted in the figure.

    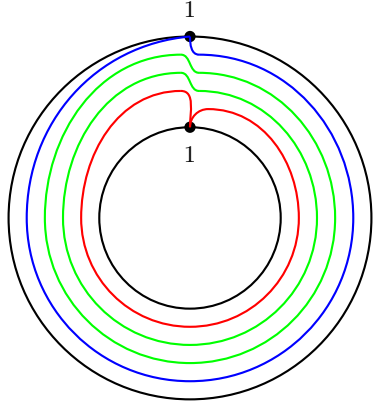
\begin{figure}[H]
        \centering
    	\begin{tikzpicture}[scale=1.2]
		\def\innerR{1}    
		\def\outerR{2}    
		\draw[thick, black] (0,0) circle (\innerR);
		\draw[thick, black] (0,0) circle (\outerR);
		
		\foreach \angle/\label in {90/1} {
			\node[circle, fill=black, inner sep=1.5pt] (inner-\label) at (\angle:\innerR) {};
			\node[font=\small, black] at (\angle:\innerR-0.3) {\label};
		}
		
		\foreach \angle/\label in {90/1} {
			\node[circle, fill=black, inner sep=1.5pt] (outer-\label) at (\angle:\outerR) {};
			\node[font=\small, black] at (\angle:\outerR+0.3) {\label};
		}
		
		\draw[thick, blue, >=Stealth]
		(90:2) to [out=-90, in=180]
		(.1, 1.8);
		
		\draw[thick, green, >=Stealth]
		(-.1, 1.8) to [out=0, in=180]
		(.1, 1.6);
		
		\draw[thick, green, >=Stealth]
		(-.1, 1.6) to [out=0, in=180]
		(.1, 1.4);
		
		\draw[thick, red, >=Stealth]
		(-.1, 1.4) to [out=0, in=90]
		(90:1);
		
		\draw[thick, blue, >=Stealth] 
		(90:2) to [out=185, in=90] 
		(-1.8,0) to [out=-90, in=180]
		(0,-1.8) to [out=0, in=-90]
		(1.8,0) to [out=90,in=0] (.1, 1.8);
		
		\draw[thick, green, >=Stealth]
		(-.1, 1.8) to [out=180, in=90]
		(-1.6,0) to [out=-90, in=180]
		(0,-1.6) to [out=0, in=-90]
		(1.6,0) to [out=90,in=0] (.1, 1.6);
		
		\draw[thick, green, >=Stealth]
		(-.1, 1.6) to [out=180, in=90]
		(-1.4,0) to [out=-90, in=180]
		(0,-1.4) to [out=0, in=-90]
		(1.4,0) to [out=90,in=0] (.1, 1.4);
		
		\draw[thick, red, >=Stealth]
		(-.1, 1.4) to [out=180, in=90]
		(-1.2,0) to [out=-90, in=180]
		(0,-1.2) to [out=0, in=-90]
		(1.2,0) to [out=90,in=0] 
		(.2, 1.2) to [out=180, in=90] (90:1);
	\end{tikzpicture}
        \caption{$\varepsilon=\varepsilon_+=(+, +, +)$.}
        \label{fig: +++}
    \end{figure}

Applying the the choice of smoothings $\varepsilon = (+, -, +)$ along the arc $\mathsf{x}$ (from top to bottom) yields  the configuration in Fig.~\ref{fig: +-+}, where the green component is an unknot.

\begin{figure}[H]
        \centering
\begin{tikzpicture}[scale=1.2]
		\def\innerR{1}    
		\def\outerR{2}

		\draw[thick, black] (0,0) circle (\innerR);
		\draw[thick, black] (0,0) circle (\outerR);
		
		\foreach \angle/\label in {90/1} {
			\node[circle, fill=black, inner sep=1.5pt] (inner-\label) at (\angle:\innerR) {};
			\node[font=\small, black] at (\angle:\innerR-0.3) {\label};
		}
		
		\foreach \angle/\label in {90/1} {
			\node[circle, fill=black, inner sep=1.5pt] (outer-\label) at (\angle:\outerR) {};
			\node[font=\small, black] at (\angle:\outerR+0.3) {\label};
		}
		
			\draw[thick, blue, >=Stealth]
		(90:2) to [out=-90, in=180]
		(.1, 1.8);
		
		\draw[thick, green, >=Stealth]
		(-.1, 1.8) to [out=0, in=0]
		(-.1, 1.6);
		
		\draw[thick, green, >=Stealth]
		(.1, 1.6) to [out=180, in=180]
		(.1, 1.4);
		
		\draw[thick, red, >=Stealth]
		(-.1, 1.4) to [out=0, in=90]
		(90:1);

		\draw[thick, blue, >=Stealth] 
		(90:2) to [out=185, in=90] 
		(-1.8,0) to [out=-90, in=180]
		(0,-1.8) to [out=0, in=-90]
		(1.8,0) to [out=90,in=0] (.1, 1.8);
		
		\draw[thick, green, >=Stealth]
		(-.1, 1.8) to [out=180, in=90]
		(-1.6,0) to [out=-90, in=180]
		(0,-1.6) to [out=0, in=-90]
		(1.6,0) to [out=90,in=0] (.1, 1.6);
		
		\draw[thick, green, >=Stealth]
		(-.1, 1.6) to [out=180, in=90]
		(-1.4,0) to [out=-90, in=180]
		(0,-1.4) to [out=0, in=-90]
		(1.4,0) to [out=90,in=0] (.1, 1.4);
		
		\draw[thick, red, >=Stealth]
		(-.1, 1.4) to [out=180, in=90]
		(-1.2,0) to [out=-90, in=180]
		(0,-1.2) to [out=0, in=-90]
		(1.2,0) to [out=90,in=0] 
		(.2, 1.2) to [out=180, in=90] (90:1);
	\end{tikzpicture}
        \caption{$\varepsilon=(+, -, +)$.}
        \label{fig: +-+}
\end{figure}

Applying the the choice of smoothings $\varepsilon = (-, -, -)$ along the arc $\mathsf{x}$ (from top to bottom) yields the configuration in Figure~\ref{fig: ---}. Let $\mathsf{z}$ denote the two homotopic arcs. 
\begin{figure}[H]
        \centering
\begin{tikzpicture}[scale=1.2]
		\def\innerR{1}    
		\def\outerR{2}    
		\draw[thick, black] (0,0) circle (\innerR);
		\draw[thick, black] (0,0) circle (\outerR);
		
		\foreach \angle/\label in {90/1} {
			\node[circle, fill=black, inner sep=1.5pt] (inner-\label) at (\angle:\innerR) {};
			\node[font=\small, black] at (\angle:\innerR-0.3) {\label};
		}
		
		\foreach \angle/\label in {90/1} {
			\node[circle, fill=black, inner sep=1.5pt] (outer-\label) at (\angle:\outerR) {};
			\node[font=\small, black] at (\angle:\outerR+0.3) {\label};
		}
		
			\draw[thick, red, >=Stealth]
		(90:2) to [out=-90, in=0]
		(-.1, 1.8);
		
		\draw[thick, blue, >=Stealth]
		(.1, 1.8) to [out=180, in=0]
		(-.1, 1.6);
		
		\draw[thick, red, >=Stealth]
		(.1, 1.6) to [out=180, in=0]
		(-.1, 1.4);
		
		\draw[thick, blue, >=Stealth]
		(.1, 1.4) to [out=180, in=90]
		(90:1);

		\draw[thick, blue, >=Stealth] 
		(90:2) to [out=185, in=90] 
		(-1.8,0) to [out=-90, in=180]
		(0,-1.8) to [out=0, in=-90]
		(1.8,0) to [out=90,in=0] (.1, 1.8);
		
		\draw[thick, red, >=Stealth]
		(-.1, 1.8) to [out=180, in=90]
		(-1.6,0) to [out=-90, in=180]
		(0,-1.6) to [out=0, in=-90]
		(1.6,0) to [out=90,in=0] (.1, 1.6);
		
		\draw[thick, blue, >=Stealth]
		(-.1, 1.6) to [out=180, in=90]
		(-1.4,0) to [out=-90, in=180]
		(0,-1.4) to [out=0, in=-90]
		(1.4,0) to [out=90,in=0] (.1, 1.4);
		
		\draw[thick, red, >=Stealth]
		(-.1, 1.4) to [out=180, in=90]
		(-1.2,0) to [out=-90, in=180]
		(0,-1.2) to [out=0, in=-90]
		(1.2,0) to [out=90,in=0] 
		(.2, 1.2) to [out=180, in=90] (90:1);
	\end{tikzpicture}
    \caption{$\varepsilon=(-, -, -)$.}
        \label{fig: ---}
\end{figure}

Expanding the product $[\mathsf{x}]\langle \mathsf{y} \rangle$ in $\mathrm{Sk}_q^{\circ}(\mathbb{A}_{1,1})$ according to the skein relations yields:
\begin{align*}
  [\mathsf{x}]\langle \mathsf{y} \rangle = [\mathsf{x}][\mathsf{y}]= q^{-1}[\mathsf{x} \cup \mathsf{y}]
 &=q^2[\mathsf{b}_1][\mathsf{b}_2][\ell]^2 + (-q^2 - q^{-2})[\mathsf{b}_1][\mathsf{b}_2]  + [\mathsf{b}_1][\mathsf{b}_2] +q^{-2}[\mathsf{b}_1][\mathsf{b}_2] + q^{-4}[\mathsf{z}]^2  \\
  &=q^2[\mathsf{b}_1][\mathsf{b}_2]([\ell]^2-1)+((1-2q^{-2})[\mathsf{b}_1][\mathsf{b}_2]+q^{-4}[\mathsf{z}]^2)\\
  &=q^2[\mathsf{b}_1][\mathsf{b}_2]U_2(\ell)+((1-2q^{-2})[\mathsf{b}_1][\mathsf{b}_2]+q^{-4}[\mathsf{z}]^2),
\end{align*}
where $U_2(\ell) = [\ell]^2 - 1$ is the second Chebyshev polynomial of the second kind.
\end{example}
  
  \subsection{General cases}
  Fix an ideal triangulation $\Delta$ of a marked surface $\Sigma$. Let $\mathsf{x}$ be a simple arc such that $[\mathsf{x}]$ is an initial cluster variable in the cluster algebra associated to $\Delta$, and let $\mathsf{Y}$ be a multicurve with $[\mathsf{Y}] \neq 0$ in $\mathrm{Sk}_q^{\circ}(\Sigma)$.

Based on the neighborhood system $\{\mathbb{A}_\ell\}$ from Lemma \ref{lem:regular-neighborhoods}, locality of skein relations allows us to restrict the top $q$-degree terms in the expansion of $[\mathsf{x}] \langle \mathsf{Y} \rangle$ to these regions. In each $\mathbb{A}_\ell$, the configuration is locally equivalent to the $\mathbb{A}_{m,m}$ case, where $m$ is the number of arcs traversing the neighborhood (see Proposition \ref{prop:leading-g}). Summing terms that are identical outside the neighborhood recovers a band element $\mathrm{Band}^{w_\ell}(\ell)$. Iterating this localization, the leading term of the global expansion is
$[\mathsf{F}] \prod_{\ell} \mathrm{Band}^{w_\ell}(\ell),$
where $\mathsf{F}$ is the remaining arc component of $\mathsf{E}_+$. Since all other terms have strictly lower $q$-degrees, this product is the leading term of the triangular basis.
  
  \begin{prop}\label{prop:initial-triangular}
  	In the skein algebra $\mathrm{Sk}_q^{\circ}(\Sigma)$, the product $[\mathsf{x}]\langle \mathsf{Y}\rangle$ admits a finite $(\prec,\mathbf{m})$-unitriangular decomposition of band basis elements.
  \end{prop} 
  
  \begin{proof} 
 (i) Let $\mathsf{Y}=\Arcs\cup \Loops=(\bigcup_i w_i \gamma_i)\cup (\bigcup_j w_j \ell_j)$, where $\gamma_i$ and $\ell_j$ are non-homotopic arc and loop components, respectively. We assume $\mathsf{x}\cap \ell_j\neq \emptyset$ for all $\ell_j$. The band element $\langle \mathsf{Y} \rangle$ expands as $\langle \mathsf{Y}\rangle=[\mathsf{Y}]+\sum_k b_k [\mathsf{Y}_k]$, where $\mathsf{Y}_k=(\bigcup_i w_i \gamma_i)\cup (\bigcup_j w^{(k)}_j \ell_j)$ with $w^{(k)}_j < w_j$ and $b_k\in \mathbb{Z}$. This implies $\mu(\mathsf{x},\mathsf{Y_k})<\mu(\mathsf{x},\mathsf{Y})$.

  	The annulus case $\mathbb{A}_{m,m}$ establishes the local model. For each annular neighborhood $\mathbb{A}_{\ell}$, the restriction $(\mathsf{x}\cdot \mathsf{Y})|_{\mathbb{A}_{\ell}}$ resolves to the following multicurves:
  	\begin{itemize}
  		\item \textbf{Leading multicurve term}:
  		$$q^{\beta_+}[\mathsf{E}_{\varepsilon_+}(\mathbb{A}_{\ell})]=q^{\beta_+}[\mathsf{F}_\ell][\ell]^{w_+(\ell)}.$$ 
  		where $\mathsf{F}_\ell$ is the local arc component in $\mathbb{A}_{\ell}$ disjoint from $\ell$.
  		\item \textbf{Other top $q$-degree terms}: 
  		$$0\neq q^{\beta_\varepsilon}[\mathsf{E}_\varepsilon (\mathbb{A}_{\ell})]=q^{\beta_{+}}(-1-q^{-4})^{|\Unknots_\varepsilon|}[\mathsf{F}_\ell][\ell]^{w_\varepsilon},\ \varepsilon\in \mathcal{E}_{\mathrm{top}}, $$ 
  		which differ from the leading term only by the choice of smoothings for crossings in $\mathbb{A}_\ell$ that form localized unknots.
  		\item \textbf{Lower $q$-degree terms}: Those not contributing to the top $q$-degree term.
  	\end{itemize}

    By summing the top $q$-degree terms, we obtain (see \eqref{eq:top-q-deg-sum}):
    \[
  		q^{\beta_+} [\mathsf{F}_\ell] \left( [\ell]^{w_+(\ell)} + \sum_{\varepsilon\neq \varepsilon_+ \in \mathcal{E}_{\mathrm{top}}} (-1)^{|\Unknots_\varepsilon|} [\ell]^{w_\varepsilon} \right) = q^{\beta_+} [\mathsf{F}_\ell] \mathrm{Band}^{w_+(\ell)}(\ell)
    \]
  where the equality follows from the unitriangularity of the basis (Proposition \ref{prop:band-triangular}).
    
 We aim to prove that $[\mathsf{x}]\langle\mathsf{Y}\rangle$ admits a $(\prec,\mm)$-unitriangular decomposition into band elements: 
 \begin{align}\label{eq:aim-decomposition}
 [\mathsf{x}]\langle \mathsf{Y}\rangle = q^\gamma \langle \mathsf{Z}_0\rangle +\sum_{s>0} b_s \langle \mathsf{Z}_s \rangle,\quad b_s\in q^{\gamma-1}\mathbb{Z}[q^{-1}].  
 \end{align}
The leading basis element $\langle \mathsf{Z}_0 \rangle$ is identified as the product $[\mathsf{F}]\prod_{\ell}\mathrm{Band}^{w_+(\ell)}(\ell)$, where $\mathsf{F}$ is the multicurve component disjoint from the loops $\{\ell\}$.

By Lemma \ref{lem:null-loop-cyclic}, each unknot $\mathsf{n} \in \Unknots_{\mathrm{top}}$ (contributing to the top $q$-degree term) is contained in some annular neighborhood $\mathbb{A}_\ell$. These unknots are formed by smoothings at crossings within their respective neighborhood $\mathbb{A}_\ell$. By Proposition \ref{prop:indep-config} (independent generation), the global sum of top $q$-degree terms factors into the product of local summations over each $\mathbb{A}_\ell$.
    
    Combining these independent local configurations, the highest $q$-degree component of the expansion of $[\mathsf{x}][\mathsf{Y}]$ (in the shifted multicurve basis) factors into $q^{\gamma}[\mathsf{F}]\prod_{\ell}\mathrm{Band}^{w_+(\ell)}(\ell)$. Furthermore, since $\mu(\mathsf{x}, \mathsf{Y}_k) < \mu(\mathsf{x}, \mathsf{Y})$, the $q$-degrees of all terms arising from $[\mathsf{x}][\mathsf{Y}_k]$ are strictly lower than $\gamma$. It follows that the leading term in the triangular decomposition is $q^{\gamma}[\mathsf{F}]\prod_{\ell}\mathrm{Band}^{w_+(\ell)}(\ell)$ as desired.

We have obtained a $(\prec,\mm)$-unitriangular decomposition
\begin{align*}
[\mathsf{x}]\langle \mathsf{Y}\rangle= q^\gamma \langle \mathsf{Z}_0\rangle +\sum_{s>0} b'_s [\mathsf{Z'}_s] ,\quad b'_s\in q^{\gamma-1}\mathbb{Z}[q^{-1}],  
 \end{align*}
where $\langle \mathsf{Z}_0 \rangle$ is the band element $[\mathsf{F}]\prod_{\ell}\mathrm{Band}^{w_+(\ell)}(\ell)$. Note that each multicurve $[\mathsf{Z'}_s]$ has a $(\prec,\mathbb{Z})$-unitriangular decomposition $\langle \mathsf{Z'}_s\rangle+c_j\sum_j \langle \mathsf{Z'}_{s,j}\rangle$, where $c_j\in \mathbb{Z}$ are the coefficients when we decompose monomials into Chebyshev polynomials. We deduce the desired decomposition \eqref{eq:aim-decomposition}.

    (ii) Let $\mathsf{Y}' = \mathsf{Y} \cup \Loops'$, where $\mathsf{Y}$ satisfies the intersection conditions of (i) and $\Loops' = \bigcup_h w_h\ell_h$ consists of non-intersecting simple loops that are disjoint from $\mathsf{x}$. By (i), the product $[\mathsf{x}]\langle\mathsf{Y}\rangle$ admits an $(\prec, \mathbf{m})$-unitriangular decomposition into band elements:
    $$[\mathsf{x}]\langle \mathsf{Y}\rangle = q^\gamma \langle \mathsf{Z}_0\rangle + \sum_{s>0} b_s \langle \mathsf{Z}_s \rangle, \quad b_s \in q^{\gamma-1}\mathbb{Z}[q^{-1}].$$

By Lemma \ref{lem:leading-term-no-new-loop}, the multicurve $\mathsf{Z}_0$ contains no components isotopic to any loop in $\Loops'$. Since $\mathsf{Z}_0 \cup \Loops'$ is a multicurve, we have $\langle \mathsf{Z}_0 \rangle \langle \Loops' \rangle = \langle \mathsf{Z}_0 \cup \Loops' \rangle$. For $s\geq 0$, since the multicurve $\mathsf{Z}_s$ arises from smoothings of $\mathsf{x} \cdot \mathsf{Y}$ or $\mathsf{x} \cdot \mathsf{Y}_k$, it remains disjoint from $\Loops'=\bigcup_h w_h\ell_h$, although it might have components of the form $w \ell_h$ for some $w\in \mathbb{N}$. Note that $\langle w \ell_h\rangle\langle w_h \ell_h\rangle$ is a $\mathbb{Z}$-linear combination of bands $\langle w'\ell_h\rangle $ for $w'\leq w+w_h$. We deduce that the product $\langle \mathsf{Z}_s \rangle \langle \Loops' \rangle$ is a $\mathbb{Z}$-linear combination of band elements, denoted $\langle W_j\rangle$. We thus have:
\begin{align*}
[\mathsf{x}]\langle\mathsf{Y'}\rangle &= [\mathsf{x}]\langle \mathsf{Y}\rangle \langle \Loops'\rangle \\
&= q^\gamma \langle \mathsf{Z}_0 \cup \Loops' \rangle + \sum_{s>0} b_s \langle \mathsf{Z}_s \rangle \langle \Loops' \rangle \\
&= q^\gamma \langle \mathsf{Z}_0 \cup \Loops' \rangle + \sum_{j \in J} c_j \langle \mathsf{W}_j \rangle,
\end{align*}
where each $c_j \in q^{\gamma-1}\mathbb{Z}[q^{-1}]$, and $\langle \mathsf{W}_j \rangle$ are band elements. This shows that $\langle \mathsf{Z}_0 \cup \Loops' \rangle$ is the unique leading term with coefficient $q^\gamma$. Since the coefficients of all other terms belong to $q^{\gamma-1}\mathbb{Z}[q^{-1}]$, the $(\prec, \mathbf{m})$-unitriangularity is preserved, completing the proof.

  \end{proof}
  
  \begin{thm}\label{thm:band-is-triangular}
  	For unpunctured surfaces, the band basis coincides with the common triangular basis.
  \end{thm}
  \begin{proof}
  	Proposition \ref{prop:initial-triangular} establishes that the band basis is triangular basis relative to any triangulation $\Delta$. Since the triangulation $\Delta$ is arbitrary, the result follows.
  \end{proof}
  
  By Theorem \ref{thm:band-is-triangular}, all conclusions of Proposition \ref{prop:leading-g} generalize to arbitrary unpunctured surfaces. 
  \begin{thm}\label{thm:leading-g}
  	Let $\Sigma$ be a triangulable marked surface (unpunctured). Consider multicurves $\mathsf{X}, \mathsf{Y}$ such that $[\mathsf{X}]$ is a nonzero cluster monomial, $[\mathsf{Y}] \neq 0$, and their superposition $\mathsf{X} \cdot \mathsf{Y}$ is transverse with minimal crossings. Then:
    \begin{enumerate}
  		\item $[\mathsf{E}_{\varepsilon_+}]\neq 0$. 
  		\item $[\mathsf{E}_{\varepsilon_-}]\neq 0$.
  		\item For any initial seed where $[\mathsf{X}]$ is an initial cluster monomial:
  		$$\deg ([\mathsf{E}_{\varepsilon_+}])=\deg ([\mathsf{X}])+\deg ([\mathsf{Y}]).$$ 
  		Moreover, $$\Lambda(\deg ([\mathsf{X}]),\deg ([\mathsf{Y}]))=2\mu(\mathsf{X}, \mathsf{Y})+2a,$$
  		where $\mu(\mathsf{X}, \mathsf{Y})$ is the minimal intersection number and $a$ is the boundary intersection number.
  		\item $\mathrm{codeg}([\mathsf{E}_{\varepsilon_-}])=\mathrm{codeg} ([\mathsf{X}])+\mathrm{codeg} ([\mathsf{Y}]).$
  	\end{enumerate}
  \end{thm}
  \begin{proof}

 (1) Similar to Proposition \ref{prop:leading-g}(1), $[\mathsf{E}_{\varepsilon_+}] \neq 0$ follows directly from Lemma \ref{lem:Muller-C1}.

 (3) In the proof of Proposition \ref{prop:initial-triangular}, the leading term is explicitly determined. While Proposition \ref{prop:initial-triangular} deals with the case where $\mathsf{x}$ is a single arc, the argument naturally extends to the case where $\mathsf{X}$ is a multicurve corresponding to an initial cluster monomial. Following the notation therein, we have:
	$$\deg([\mathsf{E}_{\varepsilon_+}]) = \deg(\langle \mathsf{Z}_0 \cup \Loops' \rangle) = \deg([\mathsf{X}]) + \deg([\mathsf{Y}']).$$
Furthermore, we have
$$ \begin{aligned}
    \Lambda(\deg ([\mathsf{X}]),\deg ([\mathsf{Y}'])) 
    = 2\mu(\mathsf{X}, \mathsf{Y}) + 2a 
    = 2\mu(\mathsf{X}, \mathsf{Y}') + 2a.
\end{aligned} $$

(2)(4) These statements follow from a similar symmetric argument as applied in the proof of Proposition \ref{prop:leading-g}.

  \end{proof}

\appendix

\section{Nested contractible sequence}\label{sec:nested}
We provide a refinement of Proposition \ref{prop:unknot-attatchmentpoints} concerning contractible components. Although this description is not required for the main results of this paper, it may be of independent interest.

\begin{lem}[Existence of inner nested components]\label{lem:inner-existence}
    Let $\mathsf{C}$ be a strip-free unknot with at least one outward angle. Then, in the interior of the disk $D$ bounded by $\mathsf{C}$, there exists another strip-free unknot $\mathsf{n}'$ which is a component of some $\mathsf{E}_{\varepsilon}$.
\end{lem}

\begin{proof}
    Since $\mathsf{C}$ contains no strips, the angle adjacent to its outward angle must be attached to another component $\mathsf{n}'$ of $\mathsf{E}_\varepsilon$. Since $D$ is contractible, the strand corresponding to this adjacent angle is contained within $D$. As $\mathsf{C} \cap \mathsf{n}' = \emptyset$, $\mathsf{n}'$ must be contained in the interior of $D$.
    
    Moreover, since $\mathsf{n}'$ is also contractible and cannot end at a marked point, it must be a distinct unknot component inside $D$. If $\mathsf{n}'$ contains any strips, we can adjust the smoothing choice at those crossings to obtain a strip-free unknot.
\end{proof}

\begin{lem}[Alternating signs]\label{lem:unique-component-alternation}
    Let $D$ be an embedded disk in $\Sigma$ whose boundary $\mathsf{n} = \partial D$ intersects the tubular neighborhood $\mathcal{T}$ of $\mathsf{x}$ transversely.
    If the portion of $D$ outside $\mathcal{T}$, denoted by $D \setminus \mathcal{T}$, is connected, then the smoothing signs must strictly alternate along $\mathsf{n}$ (i.e., $\dots, +, -, +, -, \dots$).
\end{lem}

\begin{proof}
    We prove this using the induced orientation on the boundary and the intersection indices with the tubular neighborhood.
    
    Fix an orientation on $D$, inducing a tangent vector field $V$ on $\mathsf{n}$. For each intersection point $p \in \mathsf{n} \cap \partial \mathcal{T}$, define the intersection index $\iota(p) \in \{+1, -1\}$, corresponding to $V$ pointing out of $\mathcal{T}$ (exit point) or into $\mathcal{T}$ (entry point), respectively. On a fixed side of $\mathsf{x}$, there is a geometric bijection between the smoothing signs $\varepsilon_+$ and $\varepsilon_-$ and the intersection indices $\iota$. Thus, it suffices to prove that the intersection indices alternate.
    
    The boundary $\mathsf{n}$ is decomposed into two types of segments:
    \begin{enumerate}
        \item \textbf{Leaf component segments}: These lie inside $\mathcal{T}$. Since each leaf component $d \subset D \cap \mathcal{T}$ is a disk, by continuity of flow, its boundary segment must originate from an entry point ($\iota=-1$) and terminate at an exit point ($\iota=+1$). Thus, each segment of the leaf component contributes a pair of indices $\{-1, +1\}$.
        \item \textbf{Outer segments}: These lie outside $\mathcal{T}$,  connecting all leaf components in sequence.
    \end{enumerate}
    
    Consider an outer segment connecting two consecutive leaf components $d_i$ and $d_{i+1}$. This segment must originate from the termination point of $d_i$ (exit point, $\iota=+1$) and connect to the starting point of $d_{i+1}$ (entry point, $\iota=-1$).
    
    \textbf{Contradiction}: Suppose a segment connects two points with the same index (e.g., two exit points). Together with:
   \begin{itemize}
       \item The two segments of the leaf components 
       \item An interval connecting the remaining two points in the $\mathcal{T}$-disjoint component,
   \end{itemize}
   they form a handle. Geometrically, this implies the tangent vectors at the connection are "opposing" each other. To connect these vectors while maintaining the embeddedness of $D$ (no Möbius bands), the handle would require a half-twist. This contradicts the fact that $D$ is a two-sided embedded disk (with trivial normal bundle).
    
    Therefore, the outer segments must sequentially connect an exit point to an entry point. This forces the intersection indices (and thus the corresponding smoothing signs) of all adjacent crossings to alternate strictly.
\end{proof}

\begin{prop}[Nested contractible sequence]\label{prop:unknot-chain}
    Let $\mathsf{C}$ be a strip-free contractible component of a  multicurve $\mathsf{E}_{\varepsilon}$. Consider all those $\mathsf{E}_{\varepsilon}$ that contain $\mathsf{C}$.
    \begin{enumerate}
        \item If $\mathsf{C}$ is an unknot, then there exists a sequence of distinct strip-free unknots $\{\mathsf{n}_i\}_{i=1}^s$ such that $\mathsf{n}_s = \mathsf{C}$, and the disks $D_i$ bounded by $\mathsf{n}_i$ satisfy $D_1 \subset D_2 \subset \dots \subset D_s$. Furthermore, the smoothing signs of the innermost unknot $\mathsf{n}_1$ alternate strictly. For each unknot $\mathsf{n}_i$, $i \leq s-1$, if we denote the number of negative angles by $k_i$ (where $k_i \geq 2$), then the total angle count is $2k_i$, and the sequence satisfies $k_1 \leq k_2 \leq \dots \leq k_{s-1}$.
	
        \item If $\mathsf{C}$ is a contractible arc ending at a marked point $p$, then there exists a sequence of distinct strip-free contractible curves (loops or arcs) $\{\mathsf{c}_i\}_{i=1}^s$ such that $\mathsf{c}_s = \mathsf{C}$, and the regions $D_i$ (disks or monogons) bounded by $\mathsf{c}_i$ satisfy $D_1 \subset D_2 \subset \dots \subset D_s$. 
Furthermore, all angles of the innermost curve $\mathsf{c}_1$ are inward. For each curve $\mathsf{c}_i$, if we denote the number of negative angles by $k_i$ (where $k_i \geq 2$), then the total angle count is $2k_i$ (if a loop) or $2k_i - 1$ (if an arc), and the sequence satisfies $k_1 \leq k_2 \leq \dots \leq k_{s-1}$.
    \end{enumerate}
\end{prop}

\begin{proof}
    (1) Let $D$ be the disk bounded by the unknot $\mathsf{C}$.
    
    \textbf{Step 1: Construction of the nested sequence}
    If $\mathsf{C}$ has an outward angle, by Lemma \ref{lem:inner-existence}, there exists a smaller unknot contained in the interior of $D$. Since the components of $\mathsf{E}_\varepsilon$ are finite, iterating this process yields a finite sequence of strip-free unknots $\mathsf{n}_1, \ldots, \mathsf{n}_s = \mathsf{C}$, where $D_1 \subset D_2 \subset \ldots \subset D_s$.
    
    \textbf{Step 2: Properties of the innermost unknot $\mathsf{n}_1$}
    Consider the innermost loop $\mathsf{n}_1$. Since it is innermost, it cannot contain any smaller components. Therefore, all angles of $\mathsf{n}_1$ must be inward angles.
    
    The fact that all angles are inward implies that the disk $D_1$ bounded by $\mathsf{n}_1$ does not cross $\mathsf{x}$ locally. Consequently, the portion of $D_1$ outside the tubular neighborhood is connected, meaning $\mathsf{n}_1$ contains exactly one unique $\mathcal{T}$-disjoint component.
    
    By Lemma \ref{lem:unique-component-alternation}, this forces the smoothing signs (and corresponding angles) along $\mathsf{n}_1$ to alternate strictly.
    Since $\mathsf{n}_1$ has $2k_1$ angles, these angles must exhibit an alternating pattern $(+, -, +, - \dots)$. Note that $k_1 \geq 2$, otherwise the minimal intersection property of $\mathsf{x}$ and $\mathsf{Y}$ would be violated.
    
    \textbf{Step 3: Inductive step and properties of $\mathsf{n}_2$}
    Consider $\mathsf{n}_2$, whose angle adjacent to $\mathsf{n}_1$ is an outward angle.
    Among all possible candidates for $\mathsf{n}_2$, choose the one bounding the disk of minimal area.
    
    We assert that any remaining angles of $\mathsf{n}_2$ must be inward. Otherwise, if there were an additional outward angle, its adjacent angle would attach to another unknot inside $D_2$. Iteratively adjusting the two nearby smoothings of the angle along this unknot would yield a disk smaller than $D_2$, contradicting the minimality assumption of $\mathsf{n}_2$.
    
    Furthermore, along the segment between any two consecutive outward angles, there must be an even number of inward angles, and their signs must alternate. This follows from analyzing the geometric structure of $D_2 \setminus D_1$ near the tubular neighborhood $\mathcal{T}$, as shown in Figure \ref{fig:Reg-leaf}. In addition to the
 type A leaf components, $D_2 \cap \mathcal{T}$ contains only these regions.
    
    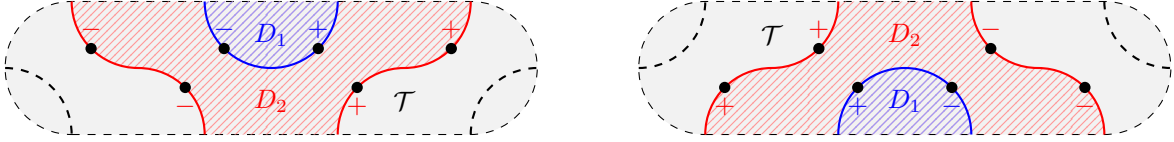
\begin{figure}[H]
        \centering
        \begin{tikzpicture}[scale=1.1]
            \begin{scope} [xshift=-1.5in, scale=.2]
                \draw [dashed, black, thin, fill=gray!10] (-12,-4) to (12,-4) arc (-90:90:4) to (-12,4) arc (90:270:4);
                \draw [thick, dashed, black] (-16,0) to [out=0, in=90] (-12,-4);
                \draw[thick, dashed, black] (12,-4) to [out=90,in=180] (16,0);
                
                \fill [pattern=north east lines, pattern color=red!40] (-12,4) to [out=270,in=180] (-8,0) to [out=0,in=90] (-4,-4) to (4,-4) to [out=90,in=180] (8,0) to [out=0,in=270] (12,4);
                
                \draw[solid, thick, red] (-12,4) to [out=270,in=180] node[marked, pos=.5] {} node[above=.1] {$-$} (-8,0) to [out=0,in=90] node[marked, pos=.5] {} node[below=.1] {$-$} (-4,-4);
                \draw[solid, thick, red] (4,-4) to [out=90,in=180] node[marked, pos=.5] {} node[below=.1] {$+$}  (8,0) to [out=0,in=270] node[marked, pos=.5] {} node[above=.1] {$+$} (12,4);
                
                \fill [pattern=north east lines, pattern color=blue!40] (-4,4) to [out=270,in=180] (0,0) to [out=0, in=270] (4,4);
                \draw[thick, blue] (-4,4) to [out=270,in=180] node[marked, pos=.5] {} node[above=.1] {$-$}(0,0)
                to [out=0, in=270] node[marked, pos=.5] {} node[above=.1] {$+$} (4,4);
                
                \node[blue] at (0,2) {$D_1$};
                \node[red]  at (0,-2) {$D_2$};
                \node at (8,-2) {$\mathcal{T}$};
            \end{scope}
            
            \begin{scope}[xshift=1.5in, scale=.2]
                \draw [fill=gray!10, dashed, thin] (-12,-4) to (12,-4) arc (-90:90:4) to (-12,4) arc (90:270:4);
                \draw[thick, dashed] (-16,0) to [out=0,in=270] (-12,4);
                
                \fill [pattern=north east lines, pattern color=red!40] (-12,-4) to [out=90,in=180] (-8,0) to [out=0,in=270] (-4,4) to (4,4)  to [out=270,in=180] (8,0) to [out=0,in=90] (12,-4);
                
                \draw[thick, red] (-12,-4) to [out=90,in=180] node[marked, pos=.5] {} node[below=.1] {$+$}  (-8,0) to [out=0,in=270] node[marked, pos=.5] {} node[above=.1] {$+$} (-4,4);
                
                \fill [pattern=north east lines, pattern color=blue!40] (-4,-4) to [out=90,in=180] (0,0) to [out=0, in=90] (4, -4);  
                \draw[thick, blue] (-4,-4) to [out=90,in=180] node[marked, pos=.5] {} node[below=.1] {$+$}  (0,0) to [out=0, in=90] node[marked, pos=.5] {} node[below=.1] {$-$} (4,-4);
                
                \draw[thick, red] (4,4) to [out=270,in=180] node[marked, pos=.5] {} node[above=.1] {$-$} (8,0) to [out=0,in=90] node[marked, pos=.5] {} node[below=.1] {$-$} (12,-4);
                
                \draw [thick, dashed] (12,4) to [out=270,in=180] (16,0);
                
                \node[blue] at (0,-2) {$D_1$};
                \node[red]  at (0,2) {$D_2$}; 
                \node at (-8, 2) {$\mathcal{T}$};
            \end{scope}
        \end{tikzpicture}
        \caption{Regions near leaf components}
        \label{fig:Reg-leaf}
    \end{figure}

     Thus, the proposition holds for $\mathsf{n}_2$. Iterating this process completes the proof of (1).
    
    (2) The argument for arcs is completely analogous.
\end{proof}

   \begin{example}
       Consider the marked surface $\Sigma_1$ illustrated in Figure~\ref{fig:surface_arc_multicurve_1} with an arc $\mathsf{x}$ (red) and a multicurve $\mathsf{Y}$ (blue) and the superposition $\mathsf{x} \cdot \mathsf{Y}$. 

       \begin{figure}[H]
   	   \centering
       \begin{tikzpicture}[scale=.3]
           \begin{scope}
               \draw[solid] (-25,-3) arc (270: 90 : 1 and 3);
               \draw [solid] (-25,-3) arc (-90: 90: 1 and 3) node[pos=.3, coordinate] (A) {};
               
               \draw [dashed] (-15,-3) arc (-90: 90: 1 and 3);
               \draw [dashed] (-5,-3) arc (-90: 90: 1 and 3);
               \draw [dashed] (5,-3) arc (-90: 90: 1 and 3);
               \draw [dashed] (15,-3) arc (-90: 90: 1 and 3);
                              
               \draw [solid] (-25,3) to (-24, 3)
                             to [out = 0, in= 270] (-23, 4) 
                             to (-23, 5);

               \draw [solid] (-23,5) arc (-180: 180 : 3 and 1);
               
               \draw [solid] (-17,5) to (-17, 4) 
                             to [out=270, in=180] (-16, 3) 
                             to (16, 3) to [out = 0, in= 270] (17,4) to (17, 5);
               
               \draw [solid] (17, 5) arc (-180: 180 : 3 and 1);
               
               \draw [solid] (23,5) to (23, 4) to [out=270, in=180] (24, 3) to (25, 3);
               
               \draw [solid] (25, 3) arc (90: -90: 1 and 3);
               
               \draw [solid] (-25,-3) to (-14, -3) to [out=0, in= 90] (-13, -4) to (-13, -5);
               
               \draw [solid] (-13, -5) arc (-180: 180 : 3 and 1);
               
               \draw [solid] (-7, -5) to (-7, -4) to [out=90, in=180] (-6,-3) to (-4, -3) to [out=0, in=90] (-3, -4) to (-3, -5);
               
               \draw [solid] (-3, -5) arc (-180: 180 : 3 and 1);
               
               \draw [solid] (3, -5) to (3, -4) to [out=90, in=180] (4,-3) to (6, -3) to [out=0, in=90] (7, -4) to (7, -5);
               
               \draw [solid] (7, -5) arc (-180: 180 : 3 and 1);
               
               \draw [solid] (13, -5) to (13, -4) to [out=90, in=180] (14,-3) to (25, -3);
               
               \draw [thick, red] (-24, 3) arc (90: 0: 1 and 3)
                                  to [out=270, in=0] (A) node[circle, fill=red, inner sep=1pt] {};
               \draw [thick, red, dashed] (-23, 4) arc (180:0: 3 and 1);
               \draw [thick, red] (-16,3) to [out=270, in=180] (-15, 2)
                                  to (-4, 2);
               \draw [thick, red] (-4, -3) arc (270: 360 : 1 and 3) 
                                  to [out=90, in=0] (-4,2);
               \draw [thick, red, dashed] (-3, -4) arc (180:0: 3 and 1); 
               \draw [thick, red] (4,-3) arc (270: 360 : 1 and 3) 
                                  to [out=90, in=180] (6, 2)
                                  to (15,2) 
                                  to [out=0, in=270] (16, 3);
               \draw [thick, red, dashed] (17,4) arc (180:0: 3 and 1);
               \draw [thick, red] (24, 3) arc (90: 0: 1 and 3) 
                                  to [out=270, in=180] ($(A)+(50,0)$) node[circle, fill=red, inner sep=1pt] {};      
               
               \draw [thick, blue] (-20, 4) node[circle, fill=blue, inner sep=1pt] {} to (-20, 2)
                                   to [out=270, in=180] (-19, 1)
                                   to (-3.5, 1);
               \draw [thick, blue] (-2.5,1) to (4.5,1);
               \draw [thick, blue] (5.5, 1) to (19,1)
                                   to [out=0, in=270] (20,2)
                                   to (20,4) node[circle, fill=blue, inner sep=1pt] {};
               
               \draw [thick, blue] (-24, 0) node[circle, fill=blue, inner sep=1pt] {} to (-23.3,0); 
               \draw [thick, blue] (-22.5, 0) to (-3.5, 0); 
               \draw [thick, blue] (-2.5,0) to (4.5,0); 
               \draw [thick, blue] (5.5, 0) to (24.5, 0);
               \draw [thick, blue] (25.3, 0) to (26, 0) node[circle, fill=blue, inner sep=1pt] {}; 
               
               \draw [thick, blue] (-24, -1) node[circle, fill=blue, inner sep=1pt] {} to (-23.4,-1); 
               \draw [thick, blue] (-22.5-.2, -1) to (-3.5, -1); 
               \draw [thick, blue] (-2.5,-1) to (4.5,-1); 
               \draw [thick, blue] (5.5, -1) to (24.7, -1);
               \draw [thick, blue] (25.4, -1) to (26, -1) node[circle, fill=blue, inner sep=1pt] {};
               
               \draw [thick, blue] (-10,-4) node[circle, fill=blue, inner sep=1pt] {} to (-10, -3)
                                   to [out=90, in=180] (-9, -2)
                                   to (-3.7, -2);
               \draw [thick, blue] (-2.7,-2) to (4.3,-2); 
               \draw [thick, blue] (5.3, -2) to (9,-2)
                                   to [out=0, in=90] (10, -3)
                                   to (10, -4) node[circle, fill=blue, inner sep=1pt] {};       
           \end{scope})
       \end{tikzpicture}
       \caption{The marked surface $\Sigma_1$ with an arc $\mathsf{x}$ and a multicurve $\mathsf{Y}$}
       \label{fig:surface_arc_multicurve_1}
   \end{figure}

   By applying smoothings according to the sequence $\varepsilon =(+, -, -, -, +, +, -, -, +, +, +, -)$ along the arc $\mathsf{x}$ from left to right at crossings in $\mathsf{x} \cap \mathsf{Y}$, we obtain the configuration illustrated in Figure~\ref{fig:nested_contractible_chain} where the red and blue unknot components form a nested contractible sequence.

    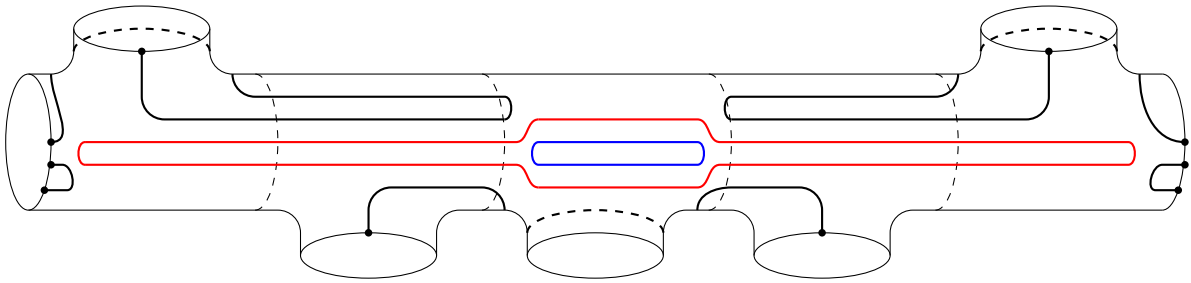
\begin{figure}[H]
   	\centering
   	\begin{tikzpicture}[scale=.3]
   		\begin{scope}
   			\draw [solid] (-25,-3) arc (270: 90 : 1 and 3);
   			\draw [solid] (-25,-3) arc (-90: 90: 1 and 3) node[pos=.25, coordinate] (A) {};
   			
   			\draw [dashed] (-15,-3) arc (-90: 90: 1 and 3);
   			\draw [dashed] (-5,-3) arc (-90: 90: 1 and 3);
   			\draw [dashed] (5,-3) arc (-90: 90: 1 and 3);
   			\draw [dashed] (15,-3) arc (-90: 90: 1 and 3);
   			
   			\draw [solid] (-25,3) to (-24, 3)
   			to [out = 0, in= 270] (-23, 4) 
   			to (-23, 5);
   			
   			\draw [solid] (-23,5) arc (-180: 180 : 3 and 1);
   			
   			\draw [solid] (-17,5) to (-17, 4) 
   			to [out=270, in=180] (-16, 3) 
   			to (16, 3) to [out = 0, in= 270] (17,4) to (17, 5);
   			
   			\draw [solid] (17, 5) arc (-180: 180 : 3 and 1);
   			
   			\draw [solid] (23,5) to (23, 4) to [out=270, in=180] (24, 3) to (25, 3);
   			
   			\draw [solid] (25, 3) arc (90: -90: 1 and 3);
   			
   			\draw [solid] (-25,-3) to (-14, -3) to [out=0, in= 90] (-13, -4) to (-13, -5);
   			
   			\draw [solid] (-13, -5) arc (-180: 180 : 3 and 1);
   			
   			\draw [solid] (-7, -5) to (-7, -4) to [out=90, in=180] (-6,-3) to (-4, -3) to [out=0, in=90] (-3, -4) to (-3, -5);
   			
   			\draw [solid] (-3, -5) arc (-180: 180 : 3 and 1);
   			
   			\draw [solid] (3, -5) to (3, -4) to [out=90, in=180] (4,-3) to (6, -3) to [out=0, in=90] (7, -4) to (7, -5);
   			
   			\draw [solid] (7, -5) arc (-180: 180 : 3 and 1);
   			
   			\draw [solid] (13, -5) to (13, -4) to [out=90, in=180] (14,-3) to (25, -3);
   			
   			\draw [thick] (-24,3) to [out=270, in=0] (-24,0) node[circle, fill=black, inner sep=1pt] {};
   			\draw [thick, dashed] (-23, 4) arc (180:0: 3 and 1);
   			\draw [thick] (-16,3) to [out=270, in=180] (-15, 2)
   			                   to (-4, 2)
   			                   to [out=0, in=0] (-4,1);
   			\draw [thick] (-20, 4) node[circle, fill=black, inner sep=1pt] {} to (-20, 2)
   			to [out=270, in=180] (-19, 1)
   			to (-4, 1);
   			
   			\draw [thick, red] (-22.5, 0) to (-3.5, 0)
   			                    to [out=0, in=180] (-2.5, 1);
   			\draw [thick, red] (-22.5, -1) to (-3.5, -1)
   			                    to [out=0, in=180] (-2.5, -2);
   			\draw [thick, red] (-22.5,0) to [out=180, in=180] (-22.5, -1);
   			\draw [thick, red] (-2.5,1) to (4.5,1)
   			                    to [out=0, in=180] (5.5,0)
   			                    to (23.5, 0)
   			                    to [out=0, in=0] (23.5,-1)
   			                    to (5.5,-1)
   			                    to [out=180, in=0] (4.5, -2)
   			                    to (-2.5, -2);
   			                    
   			 \draw [thick, blue] (-2.5,0) to (4.5,0) 
   			                     to [out=0, in=0] (4.5,-1)
   			                     to (-2.5,-1)
   			                     to [out=180, in=180] (-2.5,0);  
   			                     
   			 \draw [thick] (24,3) to [out=270, in=180] (26,0) node[circle, fill=black, inner sep=1pt] {};
   			 \draw [thick, dashed] (23, 4) arc (0:180: 3 and 1);
   			 \draw [thick] (16,3) to [out=270, in=0] (15, 2)
   			 to (6, 2)
   			 to [out=180, in=180] (6,1);
   			 \draw [thick] (20, 4) node[circle, fill=black, inner sep=1pt] {} to (20, 2)
   			 to [out=270, in=0] (19, 1)
   			 to (6, 1);
   			 
   			 \draw [thick] (-24, -1) node[circle, fill=black, inner sep=1pt] {} to (-23.5,-1) to [out=0, in=0] ($(A)+(1,0)$) to (A) node[circle, fill=black, inner sep=1pt] {};
   			 
   			 \draw [thick] (26, -1) node[circle, fill=black, inner sep=1pt] {} to (25,-1) to [out=180, in=180] ($(A)+(49,0)$) to ($(A)+(50,0)$) node[circle, fill=black, inner sep=1pt] {};
   			
   			 \draw [thick] (-10, -4) node[circle, fill=black, inner sep=1pt] {} to (-10, -3) to [out=90, in=180] (-9, -2) to (-5,-2) 
   			               to [out=0, in=90] (-4,-3);
   			 \draw [thick, dashed] (-3, -4) arc (180:0: 3 and 1);
   			 \draw [thick] (10, -4) node[circle, fill=black, inner sep=1pt] {} to (10, -3) to [out=90, in=0] (9, -2) to (6,-2) 
   			 to [out=180, in=90] (4.5,-3);             
   		\end{scope})
   	\end{tikzpicture}
        \caption{A nested contractible chain}
       \label{fig:nested_contractible_chain}
   \end{figure}
   \end{example}

   \nocite{Przytycki2016SkeinAO}

\bibliographystyle{amsalphaURL}		
\bibliography{referenceEprint}    
\end{document}